\documentclass[12pt]{article}
\usepackage{a4wide}
\usepackage[french,english]{babel}
\usepackage[utf8]{inputenc}
\usepackage[T1]{fontenc}
\usepackage{amsthm}
\usepackage{amssymb}
\usepackage{amsmath}
\usepackage{stmaryrd}
\usepackage{imakeidx}
\usepackage{graphicx}
\usepackage{caption}
\theoremstyle{plain}
\usepackage[colorlinks=true,breaklinks=true,linkcolor=black]{hyperref} 
\usepackage{amsfonts}
\usepackage{appendix}
\usepackage{ulem}
\usepackage{comment}

\newtheorem{Theorem}{Theorem}[section] 
\newtheorem{Proposition}[Theorem]{Proposition}
\newtheorem{Corollary}[Theorem]{Corollary}

\newtheorem{Lemma}[Theorem]{Lemma}

\newtheorem{Remark}[Theorem]{Remark}
\newtheorem{Definition/Proposition}[Theorem]{Definition/Proposition}
\usepackage{dsfont}
\date{}
\makeindex[columns=3]
\usepackage{mathrsfs}
\usepackage[all]{xy}
\theoremstyle{Definition}

\title{Topologies on split Kac-Moody groups over valued fields}
\author{Auguste \textsc{Hébert} \\Université de Lorraine, Institut Élie Cartan de Lorraine, F-54000 Nancy, France\\ UMR 7502,
auguste.hebert@univ-lorraine.fr}

\makeatletter
\renewcommand\theequation%
{\thesection.\arabic{equation}}
\@addtoreset{equation}{section}
\makeatother

\theoremstyle{Definition}
\newtheorem{proposition*}[thmbis]{Proposition}

\theoremstyle{Definition}
\newtheorem{proposition**}[thmter]{Proposition}

\makeatletter \@addtoreset{figure}{section}\makeatother

\newcommand{\R}{\mathbb{R}}
\newcommand{\A}{\mathbb{A}}

\newcommand{\N}{\mathbb{N}}
\newcommand{\Z}{\mathbb{Z}}
\newcommand{\Q}{\mathbb{Q}}
\newcommand{\C}{\mathbb{C}}

\newcommand{\I}{\mathcal{I}}
\newcommand{\T}{\mathcal{T}}

\newcommand{\conv}{\mathrm{conv}}

\newcommand{\cl}{\mathrm{cl}}

\newcommand{\g}{\mathfrak{g}}

\newcommand{\htt}{\mathrm{ht}}
\newcommand{\Hom}{\mathrm{Hom}}

\newcommand{\qp}{\varpi}

\newcommand{\efface}[1]{}

\newcommand{\dw}{d^{W^+}}

\newcommand{\Red}{\mathrm{Red}}
\def\k{{\Bbbk}}
\newcommand{\Ch}{\mathrm{Ch}}
\newcommand{\set}[1]{\{ #1\} } 

\newcommand{\cA}{\mathcal{A}}

\newcommand{\cB}{\mathcal{B}}
\newcommand{\fB}{\mathfrak{B}}

\newcommand{\cC}{\mathcal{C}}

\newcommand{\cF}{\mathcal{F}}

\newcommand{\fG}{\mathfrak{G}}

\newcommand{\cH}{\mathcal{H}}

\newcommand{\cK}{\mathcal{K}}

\newcommand{\fN}{\mathfrak{N}}

\newcommand{\cO}{\mathcal{O}}

\newcommand{\fP}{\mathfrak{P}}

\newcommand{\fQ}{\mathfrak{Q}}

\newcommand{\fR}{\mathfrak{R}}

\newcommand{\sR}{\mathscr{R}}

\newcommand{\cS}{\mathcal{S}}

\newcommand{\cT}{\mathcal{T}}
\newcommand{\fT}{\mathfrak{T}}

\newcommand{\sT}{\mathscr{T}}

\newcommand{\cU}{\mathcal{U}}
\newcommand{\fU}{\mathfrak{U}}

\newcommand{\cV}{\mathcal{V}}

\newcommand{\cZ}{\mathcal{Z}}

\newcommand{\fm}{\mathfrak{m}}

\newcommand{\fq}{\mathfrak{q}}

\renewcommand{\phi}{\varphi}
\renewcommand{\emptyset}{\varnothing}

\renewcommand{\tilde}[1]{\widetilde{#1}}

\makeatletter
\def\Ddots{\mathinner{\mkern1mu\raise\p@
\vbox{\kern7\p@\hbox{.}}\mkern2mu
\raise4\p@\hbox{.}\mkern2mu\raise7\p@\hbox{.}\mkern1mu}}
\makeatother

\newcommand{\Inv}{\mathrm{Inv}}

\newcommand{\kk}{\Bbbk}
\newcommand{\Stab}{\mathrm{Stab}}
\newcommand{\Fix}{\mathrm{Fix}}

\newenvironment{psmallmatrix}
  {\left(\begin{smallmatrix}}
  {\end{smallmatrix}\right)}

\usepackage[left,modulo]{lineno}

\hypersetup{
pdfauthor={Hébert},
pdftitle={Topology Kac-Moody groups},
}

\begin{document}

\maketitle

\begin{abstract}
Let $G$ be a minimal split Kac-Moody group  over a valued field $\cK$. Motivated by the representation theory of $G$, we define two topologies of topological group on $G$, which take into account the topology on $\cK$. 
\end{abstract}

\tableofcontents

\section{Introduction}

\subsection{Motivation from representation theory}

Let $G$ be a reductive group over a nonArchimedean local field $\cK$. As $G$ is finite dimensional over $\cK$, $G$ is naturally equipped with a topological group structure. Its admits a  basis of  neighbourhood of the identity consisting of open compact subgroups. A complex representation $V$ of $G$ is called smooth if for every $v\in V$, the fixator of $v$ in $G$ is open. To every compact open subgroup $K$ of $G$ is associated a Hecke algebra $\cH_K$, which is the space of $K$-bi-invariant functions from $G$ to $\C$ which have compact support. Let $V$ be a smooth representation of $G$. Then the space of $K$-invariant vectors $V^K$ is naturally equipped with the structure of an $\cH_K$-module, and we can prove that this assignment induces a bijection between the irreducible smooth representations of $G$ admitting a non zero $K$-invariant vector and the irreducible representations of $\cH_K$. 

Kac-Moody groups are infinite dimensional generalizations of reductive groups. For example, if  $\mathring{\fG}$ is a split reductive group and $\cF$ is a field,  then the associated affine Kac-Moody group is a central extension of $\mathring{\fG}(\cF[u,u^{-1}])\rtimes \cF^*$, where $u$ is an indeterminate.   Let now $G=\fG(\cK)$ be a split Kac-Moody group over $\cK$. Recently, Hecke algebras were associated to $G$. In \cite{braverman2011spherical} and \cite{gaussent2014spherical}, Braverman and Kazhdan (in the affine case) and Gaussent and Rousseau (in the general case) associated a spherical Hecke algebra $\cH_s$  to $G$, i.e an algebra associated to the spherical subgroup $\fG(\cO)$ of $G$, where $\cO$ is the ring of integers of $\cK$. In \cite{braverman2016iwahori} and \cite{bardy2016iwahori}, Braverman, Kazhdan and Patnaik and Bardy-Panse, Gaussent and Rousseau defined the Iwahori-Hecke algebra $\cH_I$ of $G$ (associated to the Iwahori subgroup $K_I$ of $G$). In \cite{abdellatif2019completed}, together with Abdellatif, we associated Hecke algebras to certain parahoric subgroups of $G$, which generalizes the construction of the Iwahori-Hecke algebra of $G$. In \cite{hebert2022principal}, \cite{hebert2021decompositions} and \cite{hebert2021kato}, we associated and studied principal series representations of $\cH_I$.

For the moment, there is no link between the representations of $G$ and the representations of its Hecke algebras.  It seems natural to try to attach an irreducible representation of $G$ to each irreducible representation of $\cH_I$. A more modest task would be to associate to each principal series representation $I_\tau$ of $\cH_I$ a principal series representation $I(\tau)$ of $G$, which is irreducible when $I_\tau$ is.

Let $T$ be a maximal split torus of $G$ and $Y$ be the cocharacter lattice of $(G,T)$.  Let $B$ be a Borel subgroup of $G$ containing $T$.  Let $T_\C=\Hom_{\mathrm{Gr}}(Y,\C^*)$ and $\tau\in T_\C$. Then $\tau$ can be extended to a character $\tau:B\rightarrow \C^*$. Assume that $G$ is reductive. Then the principal series representation $I(\tau)$  of $G$ is the induction of $\tau\delta^{1/2}$ from $B$ to $G$, where $\delta:B\rightarrow \R^*_+$ is the modulus character of $B$. More explicitly, this  is the space of locally constant functions $f:G\rightarrow \C$ such that $f(bg)=\tau\delta^{1/2}(b)f(g)$ for every $g\in G$ and $b\in B$. Then $G$ acts on $I(\tau)$ by right translation. Then $I_\tau:=I(\tau)^{K_I}$ is a representation of $\cH_I$. Assume now that $G$ is  a Kac-Moody group. Then we do not know what ``locally constant'' mean, but we can define the representation $\widehat{I(\tau)}$ of $G$ as the set of functions $f:G\rightarrow \C$ such that $f(bg)=\tau\delta^{1/2}(b)f(g)$ for every $g\in G$ and $b\in  B$. Let $\sT_G$  be a topology of topological group on $G$ such that $K_I$ is open. Then \begin{equation}\label{e_principal_series_representation}
I(\tau)_{\sT_G}:=\{f\in \widehat{I(\tau)}\mid f\text{ is locally constant for }\sT_G\}
\end{equation} is a subrepresentation of $G$ containing $\widehat{I(\tau)}^{K_I}$. Thus if we look for an irreducible representation containing $\widehat{I(\tau)}^{K_I}$ it is natural to search it inside $I(\tau)_{\sT_G}$. Moreover, the more $\sT_G$ is coarse, the smaller $I(\tau)_{\sT_G}$ is. We thus look for the coarsest topology of topological group on $G$ for which $K_I$ is open. 

\subsection{Topology on $G$, masure and main results}

We now assume that $\cK$ is any field equipped with a valuation $\omega:\cK\rightarrow \R\cup\{+\infty\}$ such that $\omega(\cK^*)\supset \Z$.  We no longer assume $\cK$ to be local, and $\omega(\cK^*)$ can be dense in $\R$. Let $\cO$ be its ring of integers. Let $\fG$ be a split  Kac-Moody group (à la Tits, as defined in \cite{tits1987uniqueness}) and $G=\fG(\cK)$.  In \cite{gaussent2008kac} and \cite{rousseau2016groupes}, Gaussent and Rousseau associated to $G$ a kind of Bruhat-Tits building, called a masure, on which $G$ acts (when $G$ is reductive, $\I$ is the usual Bruhat-Tits building). They defined the spherical subgroup $K_s$ as the fixator of some vertex $0$ in the masure (we prove in Proposition~\ref{p_fixator_zero} that  $K_s=\fG^{\min}(\cO)$, where $\fG^{\min}$ is the minimal Kac-Moody group defined by Marquis in \cite{marquis2018introduction}). They also define the Iwahori subgroup $K_I$ as the fixator of some alcove  $C_0^+$ of $\I$. Then we define the topology $\sT_\Fix$ on $G$ as follows. A subset $V$ of $G$ is open if for every $g\in V$, there exists a finite subset $F$ of $\I$ such that $G_F.g\subset V$, where $G_F$ is the fixator of $F$ in $G$. Then we prove that  $\sT_{\Fix}$ is the coarsest topology of topological group on $G$ for which $K_I$ is open (see Proposition~\ref{p_characterization_Tfix}). However, it is not Hausdorff in general. Indeed, let $\cZ\subset T$ be the center of $G$ and $\cZ_\cO=\cZ\cap \fT(\cO)$. Then $\cZ_\cO$ is the fixator of $\I$ in $G$ and when $\cZ_\cO$ is nontrivial (which already happens for $\mathrm{SL}_2(\cK)$), $\sT_\Fix$ is not Hausdorff. 

To address this issue, we define an other topology, $\sT$, finer than $\sT_\Fix$ and Hausdorff. Let $\A=Y\otimes \R$ be the standard apartment of $\I$ and $\Phi\subset \A^*$ be the set of roots of $(G,T)$. Then $\I=\bigcup_{g\in G} g.\A$.   Let us begin with the case where $G=\mathrm{SL}_2(\cK)$.  Let $\qp\in \cO$ be such that $\omega(\qp)=1$. For $n\in \N^*$, let $\pi_n:\mathrm{SL}_2(\cO)\rightarrow \mathrm{SL}_2(\cO/\qp^n\cO)$ be the natural projection.  Then a basis of the neighbourhood of the identity is given by the $(\ker \pi_n)_{n\in \N^*}$. Let $U^+=\begin{psmallmatrix} 1 & *\\ 0 & 1\end{psmallmatrix}$ and $U^-=\begin{psmallmatrix} 1 & 0\\ * & 1\end{psmallmatrix}$.  Then one can prove that $\ker \pi_n= (U^+\cap \ker \pi_n).(U^-\cap \ker \pi_n).(T\cap \ker \pi_n)$. Let $\alpha,-\alpha$ be the two roots of $(G,T)$. Let  $x_\alpha:a\mapsto \begin{psmallmatrix} 1 & a\\ 0 & 1\end{psmallmatrix}$ and $x_{-\alpha}:a\mapsto \begin{psmallmatrix} 1 & 0\\ a & 1\end{psmallmatrix}$. Then $x_\alpha(\qp^n\cO)$ fixes $\{a\in \A\mid \alpha(a)\geq -n\}$ and $x_{-\alpha}(\qp^n \cO)$ fixes $\{a\in \A\mid \alpha(a)\leq n\}$. Therefore if $\lambda\in \A$ is such that $\alpha(\lambda)=1$ and $[-n\lambda,n\lambda]=\alpha^{-1}([-n,n])$, we have \[\ker \pi_n=\left(U^+\cap \Fix_G([-n\lambda,n\lambda])\right).\left( U^-\cap \Fix_G([-n\lambda,n\lambda])\right).(T\cap \ker \pi_n).\] We now return to the general case for $G$. We prove that the topology associated to $(\ker \pi_n)_{n\in \N^*}$ is not a topology of topological group if $G$ is not reductive (see Lemma~\ref{l_non_conjugacy_invariance_congruence}). Let $(\alpha_i)_{i\in I}$ be the set of simple roots of $(G,T)$ and  $C^v_f=\{x\in \A\mid \alpha_i(x)>0,\forall i\in I\}$. Let $W^v$ be the Weyl group of $(G,T)$ and  $\lambda\in Y\cap \bigsqcup_{w\in W^v}w.C^v_f$. We define the following  subset $\cV_{n\lambda}$ of $G$, for $n\in \N^*$: \[\cV_{n\lambda}=\left(U^+\cap \Fix_G([-n\lambda,n\lambda])\right).\left(U^-\cap \Fix_G([-n\lambda,n\lambda])\right).(T\cap \ker \pi_{2N(\lambda)}),\] where $N(\lambda)=\min \{|\alpha(\lambda)|\mid \alpha\in  \Phi_+\}$. We prove the following theorem: \begin{Theorem}(see Theorem~\ref{t_conjugation_invariance}, Lemma~\ref{l_decomposition_Vlambda}, Proposition~\ref{p_comparison_T_TFix} and Proposition~\ref{p_nonCompactness}):  
\begin{enumerate}
\item For $n\in \N^*$ and $\lambda\in Y\cap \bigsqcup_{w\in W^v} w.C^v_f$, $\cV_{n\lambda}$ is a subgroup of $\fG^{\min}(\cO)$. 

\item The topology $\sT$ associated with $(\cV_{n\lambda})_{n\in \N^*}$ is Hausdorff, independent of the choice of $\lambda$ and equips $G$ with the structure of a topological group. 

\item The topology $\sT$ is finer than $\sT_{\Fix}$ and if $\cK$ is Henselian,  we have $\sT=\sT_\Fix$ if and only if $\cZ_\cO=\{1\}$.

\item Every compact subset of $G$ has empty interior (for $\sT$).
\end{enumerate}
\end{Theorem} 

Note that $\sT$ and $\sT_\Fix$ induce the same topologies on $U^+$ and $U^-$. The main difference comes from what happens in $T$. As the elements of $I(\tau)_{\sT}$ and $I(\tau)_{\sT_\Fix}$ are left $\fT(\cO)$-invariant, these two spaces are actually equal (see Remark~\ref{r_comparison_T_TFix}). 

\medskip

In \cite{hartnick2013topological}, based on works of Kac and Peterson on the topology of  $\fG(\C)$, Hartnick, Köhl and Mars defined a Kac-Peterson topology on $\fG(\cF)$, for any local field $\cF$ (Archimedean or not). Assume that $\cK$ is local and let $\sT_{KP}$ be the Kac-Peterson topology on $G$. We prove that when $G$ is not reductive, then $\sT$ is strictly coarser than $\sT_{KP}$ (see Proposition~\ref{p_T_coarser_TKP}) and thus $\sT$ seems more adapted for our purpose. 

\medskip Assume that $\fG$ is affine $\mathrm{SL}_2$ (with a nonfree set of simple coroots).  Then  $G=\mathrm{SL}_2(\cK[u,u^{-1}])\rtimes \cK^*$. Up to the assumption that $\ker \pi_n\subset \begin{psmallmatrix} 1+\qp^n \cO[u,u^{-1}] & \qp^n \cO[u,u^{-1}]\\
\qp^n \cO[u,u^{-1}] & 1+\qp^n \cO[u,u^{-1}]\end{psmallmatrix}\rtimes (1+\qp^n\cO)$, for $n\in \N^*$, we prove that the topology $\sT$ on $G$ is associated to the  filtration $(H_n)_{n\in \N^*}$, where $H_n=\ker(\pi_n)\cap \left(\begin{psmallmatrix} \cO[(\qp u)^n,(\qp u^{-1})^n] &  \cO[(\qp u)^n,(\qp u^{-1})^n]\\ 
 \cO[(\qp u)^n,(\qp u^{-1})^n]	&  \cO[(\qp u)^n,(\qp u^{-1})^n] \end{psmallmatrix},\cK^*\right)$.
 
\medskip

 The paper is organized as follows. In section~\ref{s_KM_groups_Masures}, we define Kac-Moody groups (as defined by Tits, Mathieu and Marquis) and  the masures. 
 
 In section~\ref{s_Congruence_subgroups}, we define and study the subgroups $\ker \pi_n$ of $\fG^{\min}(\cO)$.
 
 In section~\ref{s_Def_topologies}, we define the topologies $\sT$ and $\sT_\Fix$, and compare them.
 
 In section~\ref{s_Prop_topologies}, we study the properties of $\sT$ and $\sT_\Fix$: we prove that $\sT_{KP}$ is strictly finer than $\sT$, we describe the topology in the case of affine $\mathrm{SL}_2$, and we prove that usual subgroups of $G$ (i.e $T$, $N$, $B$, etc.) are closed for $\sT$.
 
 \medskip
 
 \textbf{Acknowledgements} I would like to thank Nicole Bardy-Panse and Guy Rousseau for the very helpful discussions we had on the subject. I  also thank Stéphane Gaussent,   Timothée Marquis and Dinakar Muthiah  for useful conversations and suggestions and the anonymous referee for his/her careful reading and useful comments.

\section{Kac-Moody groups and masures}\label{s_KM_groups_Masures}

In this section, we define Kac-Moody groups and masures. Let $\cK$ be a field.  There are several possible definitions of Kac-Moody groups and we are interested in the minimal one $\fG(\cK)$, as defined by Tits in \cite{tits1987uniqueness}. However, because of the lack of commutation relations in $\fG(\cK)$, it is convenient to embed it in its Mathieu's positive and negative completions $\fG^{pma}(\cK)$ and $\fG^{nma}(\cK)$. Then one define certain subgroups of $\fG(\cK)$ as the intersection of a subgroup of $\fG^{pma}(\cK)$ and $\fG(\cK)$. For example if $\fG$ is affine $\mathrm{SL}_2$ (with a nonfree set of simple roots and coroots), then $\fG(\cK)=\mathrm{SL}_2(\cK[u,u^{-1}])$, $\fG^{pma}(\cK)=\mathrm{SL}_2(\cK\left(\!(u)\!)\right)$ and $\fG^{nma}(\cK)=\mathrm{SL}_2\left(\cK(\!(u^{-1})\!)\right)$. 

As we want to define congruence subgroups in our framework, we also need to work with Kac-Moody groups over rings: if $\cK$ is equipped with a valuation $\omega$ and $\qp$ is such that $\omega(\qp)=1$, then we want to define $\ker \pi_n\subset \fG(\cO)$, where $\pi_n:\fG(\cO)\rightarrow \fG(\cO/\qp^n \cO)$ is the natural projection. The functor defined by Tits in \cite{tits1987uniqueness} goes from the category of rings to the category of groups.  However the fact that it satisfies the axioms defined by Tits is proved only for  fields (see \cite[3.9 Theorem 1]{tits1987uniqueness}) and we do not know if it is ``well-behaved'' on rings, so we will consider it only as a functor from the category of fields to the category of groups. In \cite[8.8]{marquis2018introduction}, Marquis introduces a functor $\fG^{\min}$ which goes from the category of rings to the category of groups and he proves that it has nice properties (see \cite[Proposition 8.128]{marquis2018introduction}), especially on Bézout domains. We will use its functor $\fG^{\min}$. We have  $\fG^{\min}(\cF)\simeq \fG(\cF)$  for any field $\cF$. This functor is defined as a subfunctor of $\fG^{pma}$, so we first define Tits's functor, then Mathieu's functors and then Marquis's functor.

\subsection{Standard apartment of a masure}\label{subRootGenSyst}

\subsubsection{Root generating system}\label{sub_Root_generating_system}

A \textbf{ Kac-Moody matrix} (or generalized Cartan matrix) is a square matrix $A=(a_{i,j})_{i,j\in I}$ indexed by a finite set $I$, with integral coefficients, and such that :
\begin{enumerate}
\item[\tt $(i)$] $\forall \ i\in I,\ a_{i,i}=2$;

\item[\tt $(ii)$] $\forall \ (i,j)\in I^2, (i \neq j) \Rightarrow (a_{i,j}\leq 0)$;

\item[\tt $(iii)$] $\forall \ (i,j)\in I^2,\ (a_{i,j}=0) \Leftrightarrow (a_{j,i}=0$).
\end{enumerate}
A \textbf{root generating system} is a $5$-tuple $\mathcal{S}=(A,X,Y,(\alpha_i)_{i\in I},(\alpha_i^\vee)_{i\in I})$\index{S@$\mathcal{S}$}\index{Y@$Y$} made of a Kac-Moody matrix $A$ indexed by the finite set $I$, of two dual free $\Z$-modules $X$ and $Y$ of finite rank, and of a family $(\alpha_i)_{i\in I}$ (respectively $(\alpha_i^\vee)_{i\in I}$) of elements in $X$ (resp. $Y$) called \textbf{simple roots} (resp. \textbf{simple coroots}) that satisfy $a_{i,j}=\alpha_j(\alpha_i^\vee)$ for all $i,j$ in $I$. Elements of $X$ (respectively of $Y$) are called \textbf{characters} (resp. \textbf{cocharacters}).

Fix such a root generating system $\mathcal{S}=(A,X,Y,(\alpha_i)_{i\in I},(\alpha_i^\vee)_{i\in I})$ and set $\A:=Y\otimes \R$\index{A@$\A$}. Each element of $X$ induces a linear form on $\A$, hence $X$ can be seen as a subset of the dual $\A^*$. In particular, the $\alpha_{i}$'s (with $i \in I$) will be seen as linear forms on $\A$. This allows us to define, for any $i \in I$, a \textbf{simple reflection} $r_{i}$ of $\A$ by setting $r_{i}.v := v-\alpha_i(v)\alpha_i^\vee$ for any $v \in \A$.   One defines the \textbf{Weyl group of $\mathcal{S}$} as the subgroup $W^{v}$\index{W@$W^v$} of $\mathrm{GL}(\A)$ generated by $\{r_i\mid i\in I\}$. The pair $(W^{v}, \{r_i\mid i\in I\})$ is a Coxeter system, hence we can consider the length $\ell(w)$ with respect to $\{r_i\mid i\in I\}$ of any element $w$ of $W^{v}$.

The following formula defines an action of the Weyl group $W^{v}$ on $\A^{*}$:  
\[\displaystyle \forall \ x \in \A , w \in W^{v} , \alpha \in \A^{*} , \ (w.\alpha)(x):= \alpha(w^{-1}.x).\]
Let $\Phi:= \{w.\alpha_i\mid (w,i)\in W^{v}\times I\}$\index{P@$\Phi,\Phi^\vee$} (resp. $\Phi^\vee=\{w.\alpha_i^\vee\mid (w,i)\in W^{v}\times I\}$) be the set of \textbf{real roots} (resp. \textbf{real coroots}): then $\Phi$ (resp. $\Phi^\vee$) is a subset of the \textbf{root lattice} $Q := \displaystyle \bigoplus_{i\in I}\Z\alpha_i$ (resp. \textbf{coroot lattice} $Q^\vee=\bigoplus_{i\in I}\Z\alpha_i^\vee$). By \cite[1.2.2 (2)]{kumar2002kac}, one has $\R \alpha^\vee\cap \Phi^\vee=\{\pm \alpha^\vee\}$ and $\R \alpha\cap \Phi=\{\pm \alpha\}$ for all $\alpha^\vee\in \Phi^\vee$ and $\alpha\in \Phi$.

We define the \textbf{height} $\htt:Q\rightarrow \Z$\index{h@$\htt$} by $\htt(\sum_{i\in I} n_i\alpha_i)=\sum_{i\in I} n_i$, for $(n_i)\in \Z^I$. 

\subsubsection{Vectorial apartment}

As in the reductive case, define the \textbf{fundamental chamber} as $C_{f}^{v}:= \{v\in \A \ \vert \ \forall i \in I,  \alpha_i(v)>0\}$\index{C@$C^v_f$}. 

 Let $\mathcal{T}:= \displaystyle \bigcup_{w\in W^{v}} w.\overline{C^{v}_{f}}$\index{T@$\T$} be the \textbf{Tits cone}. This is a convex cone (see \cite[1.4]{kumar2002kac}).
 
 For $J\subset I$, set $F^v(J)=\{x\in \A\mid  \alpha_j(x)=0,\ \forall j\in J\text{ and }\alpha_j(x)>0,\ \forall j\in I\setminus J\}$\index{F@$F^v(J)$}. A \textbf{positive vectorial face} (resp. \textbf{negative}) is a set of the form $w.F^v(J)$ ($-w.F^v(J)$) for some $w\in W^v$ and $J\subset I$. Then by \cite[5.1 Th\'eor\`eme (ii)]{remy2002groupes}, the family of positive vectorial faces of $\A$ is a partition of $\T$ and the stabilizer of $F^v(J)$ is $W_J=\langle J\rangle$. 

One sets $Y^{++}=Y\cap\overline{C^v_f}$\index{Y@$Y^+$, $Y^{++}$} and $Y^+=Y\cap \T$. An element of $Y^+$ is called \textbf{regular} if it does not belong to any wall, i.e if it belongs to $\bigsqcup_{w\in W^v} w.C^v_f$. 

\begin{Remark}
By \cite[§4.9]{kac1994infinite} and \cite[§ 5.8]{kac1994infinite} the following conditions are equivalent:\begin{enumerate}
\item the Kac-Moody matrix $A$ is of finite type (i.e. is a Cartan matrix),

\item $\A=\T$

\item $W^v$ is finite.
\end{enumerate}
\end{Remark}

\subsection{Split Kac-Moody groups over fields}

\subsubsection{Minimal  Kac-Moody groups over fields}\label{ss_m_KM_grp}

 Let  $\mathfrak{G}=\fG_{\mathcal{S}}$\index{G@$\mathfrak{G}$} be the group functor associated in \cite{tits1987uniqueness} with  the  root generating system $\mathcal{S}$, see also \cite[8]{remy2002groupes}.
 Let $\cK$\index{K@$\cK$}\index{o@$\omega$} be a  field. Let $G=\mathfrak{G}(\cK)$ be the \textbf{split Kac-Moody group over $\cK$ associated with $\mathcal{S}$}.  The group $G$ is generated by the following subgroups:\begin{itemize}
\item the fundamental torus $T=\mathfrak{T}(\cK)$\index{T@$T$}, where $\mathfrak{T}=\mathrm{Spec}(\Z[X])$,

\item the root subgroups $U_\alpha=\mathfrak{U}_\alpha(\cK)$\index{U@$U_\alpha$}, for $\alpha\in \Phi$, each isomorphic to $(\cK,+)$ by an isomorphism $x_\alpha$\index{x@$x_\alpha$}.
\end{itemize}

The groups $X$ and $Y$ correspond to the character lattice $\Hom (\fT,\mathbb{G}_m)$  and cocharacter lattice $\Hom(\mathbb{G}_m,\fT)$ of $\mathfrak{T}$ respectively. One writes $\mathfrak{U}^{\pm}$\index{U@$\mathfrak{U}^{\pm},U^{\pm}$} the subgroup of $\mathfrak{G}$ generated by the $\mathfrak{U}_{\alpha}$, for $\alpha\in \Phi^{\pm}$ and $U^{\pm}=\mathfrak{U}^{\pm}(\cK)$.    

By a simple computation in $\mathrm{SL}_2$, we have for  $\alpha\in \Phi$ and  $a,b\in \cK$  such that $ab\neq -1$:
\begin{equation}\label{e_Commutation_relation}
\begin{aligned} x_{-\alpha}(b)x_\alpha(a)&=x_\alpha(a(1+ab)^{-1})\alpha^\vee(1+ab)x_{-\alpha}(b(1+ab)^{-1})\\
&=x_\alpha(a(1+ab)^{-1})x_{-\alpha}(b(1+ab))\alpha^\vee(1+ab), \end{aligned}
\end{equation}

where $\alpha^\vee=w.\alpha_i^\vee$ if $\alpha=w.\alpha_i$, for $i\in I$ and $w\in W^v$.

Let $\mathfrak{N}$\index{N@$\mathfrak{N}$} be the group functor on rings such that if $\sR'$ is a ring, $\mathfrak{N}(\sR')$ is the subgroup of $\mathfrak{G}(\sR')$ generated by $\mathfrak{T}(\sR')$ and the $\tilde{r}_{i}$, for $i\in I$, where \begin{equation}\label{e_tilde_s_alpha}
\tilde{r}_{i}=x_{\alpha_i}(1) x_{-\alpha_i}(-1)x_{\alpha_i}(1).
\end{equation} Then if $\sR'$ is a field with at least $4$ elements, $\mathfrak{N}(\sR')$ is the normalizer of $\mathfrak{T}(\sR')$ in $\mathfrak{G}(\sR')$. 

Let $N=\mathfrak{N}(\cK)$\index{N@$\mathfrak{N}$} and $\mathrm{Aut}(\A)$ be the group of affine automorphisms of $\A$. Then by \cite[1.4 Lemme]{rousseau2006immeubles}, there exists a group morphism $\nu^v:N\rightarrow \mathrm{GL}(\A)$\index{n@$\nu^v$} such that:\begin{enumerate}
\item for $i\in I$, $\nu^v(\tilde{r}_{i})$ is the simple reflection $r_i\in W^v$,

\item $\ker \nu^v=T$.

\end{enumerate}

The aim of the next two subsubsections is to define Mathieu's Kac-Moody group. This group is defined by assembling three ingredients: the group $\fU^{pma}$, which corresponds to a maximal positive unipotent subgroup of $\fG^{pma}$, the torus $\fT$ and copies  of $\mathrm{SL}_2$, one for  each simple root $\alpha_i$, $i\in I$.

\subsubsection{The affine group scheme $\fU^{pma}$}

In this subsubsection, we define $\fU^{pma}$. Let $\g$\index{g@$\g$} be the Kac-Moody Lie algebra over $\C$ associated with $\cS$ (see \cite[1.2]{kumar2002kac}) and $\cU_\C(\g)$\index{U@$\cU_\C(\g)$} be its enveloping algebra. The group $\fU^{pma}(\C)$,  will be defined as a subgroup of a completion of $\cU_\C(\g)$. As we want to define $\fU^{pma}(\sR)$, for any ring $\sR$,  we will also consider $\Z$-forms of $\g$ and $\cU_\C(\g)$.

 The Lie algebra $\g$ decomposes as $\g=\bigoplus_{\alpha\in \Delta} \g_\alpha$, where $\Delta\subset Q$ is the \textbf{set of roots} and $\g_\alpha$ is the proper space associated with $\alpha$, for $\alpha\in \Delta$ (see \cite[1.2]{kumar2002kac}). We have $\Delta=\Delta_+\sqcup \Delta_-$, where $\Delta_+=\Delta\cap Q_+$ and $\Delta_-=-\Delta_+$. We have $\Phi\subset \Delta$. The elements of $\Phi=\Delta_{re}$ are called \textbf{real roots }and the elements of $\Delta_{im}= \Delta\setminus \Phi$\index{D@$\Delta$, $\Delta_{re}$, $\Delta_{im}$} are called \textbf{imaginary roots}.
 
  Following \cite[4]{tits1987uniqueness} one defines $\cU$\index{U@$\cU$} as the $\Z$-subalgebra of $\cU_\C(\g)$ generated by  $e_i^{(n)}:=\frac{e_i^n}{n!},f_i^{(n)}:=\frac{f_i^n}{n!}$, $\binom{h}{n}$, for  $i\in I$ and $h\in Y$ (where the $e_i, f_i$ are the generators of $\g$, see \cite[1.1]{kumar2002kac}). This is a $\Z$-form of $\cU_\C(\g)$. The algebra $\cU_\C(\g)$  decomposes as $\cU_\C(\g)=\bigoplus_{\alpha\in Q}\cU_\C(\g)_\alpha$ where we use the standard $Q$-graduation on $\cU_\C(\g)$ induced by the $Q$-graduation of $\g$ (for $i\in I$,  $\deg(e_i)=\alpha_i$, $\deg(f_i)=-\alpha_i$, $\deg(h)=0$, for $h\in Y$, $\deg(xy)=\deg(x)+\deg(y)$ for all $x,y\in \cU_\C(\g)$ which can be written as a product of nonzero elements of $\g$).  For $\alpha\in Q$, one sets $\cU_\alpha=\cU_\C(\g)_\alpha\cap \cU$\index{U@$\cU_{\alpha}$} and $\cU_{\alpha,\sR}=\cU_\alpha\otimes \sR$\index{U@$\cU_{\alpha,\sR}$}.

  For a ring $\sR$, we set $\cU_\sR=\cU\otimes_\Z \sR$\index{U@$\cU_{\sR}$}. One sets $\widehat{\cU}^+=\prod_{\alpha\in Q_+} \cU_\alpha$\index{U@$\widehat{\cU}$} and  $\widehat{\cU}^+_\sR=\prod_{\alpha\in Q_+} \cU_{\alpha,\sR}$\index{U@$\widehat{\cU}_\sR$}. This is the completion of $\cU^+$ with respect to the $Q_+$-gradation. 

If $(u_\alpha)\in \prod_{\alpha\in Q_+} \cU_{\alpha,\sR}$, we write $\sum_{\alpha\in Q_+} u_\alpha$ the corresponding element of $\widehat{\cU}^+_{\sR}$. A sequence $(\sum_{\alpha\in Q_+} u_\alpha^{(n)})_{n\in \N}$ converges in $\widehat{\cU}_\sR$ if and only if for every $\alpha\in \Delta_+$, the sequence $(u_\alpha^{(n)})_{n\in \N}$ is stationary.

Let $(E,\leq)$ be a totally ordered set. Let $(u^{(e)})\in  (\widehat{\cU}_{\sR})^E$. For $e\in E$, write  $u=\sum_{\alpha\in Q_+} u_\alpha^{(e)}$, with $u_\alpha^{(e)}\in \cU_{\alpha,\sR}$, for $\alpha\in Q_+$. We assume that for every $\alpha\in Q_+$, $\{e\in E\mid u_\alpha^{(e)}\neq 0\}$ is finite. Then one sets $\prod_{e\in E} u^{(e)}=\sum_{\alpha\in Q_+} u_\alpha$,  where \[u_\alpha=\sum_{\stackrel{(\beta_1,\ldots,\beta_k)\in Q_+^{(\N)},}{\beta_1+\ldots+\beta_k=\alpha}}\sum_{\stackrel{(e_1,\ldots,e_k)\in E,}{e_1<\ldots <e_k}} u_{\beta_1}^{(e_1)}\ldots u_{\beta_k}^{(e_k)}\in  \cU_{\alpha,\sR},\]  for $\alpha\in \Delta_+.$ This is well-defined since in the sum defining $u_\alpha$, only finitely many nonzero terms appear.

Let $\cA=\bigoplus_{\alpha\in Q_+} \cU_\alpha^*$\index{A@$\cA$}, where $\cU_\alpha^*$ denotes the dual of $\cU_\alpha$ (as a $\Z$-module). We have a natural $\sR$-modules isomorphism  between $\widehat{\cU}_{\sR}^+$ and $\Hom_{\Z-\mathrm{lin}}(\cA,\sR)$, for any ring $\sR$ (see \cite[(8.26)]{marquis2018introduction}) and we now identify these two spaces. The algebra $\cA$ is equipped with a Hopf algebra structure (see \cite[Definition 8.42]{marquis2018introduction}). This additional structure equips \[\fU^{pma}(\sR):= \Hom_{\Z-\mathrm{Alg}}(\cA,\sR)\] with the structure of a group (see \cite[Appendix A.2.2]{marquis2018introduction}). Otherwise said, $\cA$ is the representing algebra of the (infinite dimensional in general) affine group scheme $\fU^{pma}:\Z$-$\mathrm{Alg}\rightarrow \mathrm{Grp}$.

Let $\alpha\in \Delta\cup \{0\}$ and $x\in \g_{\alpha,\Z}$. An \textbf{exponential sequence} for $x$ is a sequence $(x^{[n]})_{n\in \N}$ of elements of $\cU$ such that $x^{[0]}=1$, $x^{[1]}=x$ and $x^{[n]}\in \cU_{n\alpha}$ for $n\in \Z_{\geq 1}$ and satisfying the conditions of \cite[Definition 8.45]{marquis2018introduction}. By \cite[Proposition 2.7]{rousseau2016groupes} or \cite[Proposition 8.50]{marquis2018introduction}, such a sequence exists. Note that it is not unique in general. However, if $\alpha\in \Phi_+$, the unique exponential sequence for $x$ is $(x^{[n]})_{n\in \N}=(\frac{1}{n!}x^n)$ by \cite[2.9 2)]{rousseau2016groupes} (this sequence is often denoted $(x^{(n)})_{n\in \N}$ in the literature).

For $r\in \sR$, one then sets \[[\exp](rx)=:\sum_{n\in \N} x^{[n]}\otimes r^n\in \widehat{\cU}_\sR^+.\]\index{e@$[\exp]$} This is the \textbf{twisted exponential} of $rx$ associated with the sequence $(x^{[n]})_{n\in \N}$. 

We fix for every $\alpha\in \Delta_+$ a $\Z$-basis $\cB_\alpha$ of $\g_{\alpha,\Z}:=\g_\alpha\cap \cU$\index{g@$\g_{\alpha,\Z}$}.  Set $\cB=\bigcup_{\alpha\in \Delta_+}\cB_\alpha$\index{B@$\cB$}. We fix an order on each  $\cB_\alpha$ and on $\Delta_+$.  
Let $\alpha\in \Delta_+$. One defines $X_\alpha:\g_{\alpha,\Z}\otimes \sR\rightarrow \fU^{pma}(\sR)$\index{X@$X_\alpha$} by $X_\alpha(\sum_{x\in \cB_\alpha} \lambda_x.x)=\prod_{x\in \cB_\alpha} [\exp]\lambda_x.x$, for $(\lambda_x)\in \sR^{\cB_\alpha}$.  When $\alpha\in \Phi_+$, we have $\g_{\alpha,\Z}=\Z e_\alpha$, where $e_\alpha$ is defined in \cite[Remark 7.6]{marquis2018introduction}. One sets $x_\alpha(r)=[\exp](r e_\alpha)$, for $r\in \sR$. One has $X_\alpha(\g_{\alpha,\Z}\otimes \sR)=x_\alpha(\sR):=\fU_{\alpha}(\sR)$.   By \cite[Theorem 8.5.1]{marquis2018introduction}, every $g\in \fU^{pma}(\sR)$ can be written in a unique way as a   product \begin{equation}\label{e_normal_form_Upma}
g=\prod_{\alpha\in \Delta_+}X_\alpha(c_\alpha),
\end{equation}  where $c_\alpha\in \g_{\alpha,\Z}\otimes \sR$, for $\alpha\in \Delta_+$, where the product is taken in the given order  on $\Delta_+$.

 Let  $\Psi\subset \Delta_+$. We say that $\Psi$ is \textbf{closed} if for all $\alpha,\beta\in \Psi$, for all $p,q\in \N^*$, $p\alpha+q\beta\in \Delta_+$ implies $p\alpha+q\beta\in \Psi$. Let  $\Psi\subset \Delta_+$ be a closed subset. One sets \[\fU_{\Psi}^{\mathrm{pma}}(\sR)=\prod_{\alpha\in \Psi} X_\alpha(\g_{\alpha,\Z}\otimes \sR)\subset \fU^{pma}(\sR).\]\index{u@$\fU_{\Psi}^{\mathrm{pma}}$, $\fU_{\Psi}^{\mathrm{pma}}(\sR)$}

This is a subgroup of $\fU^{pma}$, which does not depend on the chosen order on $\Delta_+$ (for the product). This is not the definition given in \cite{rousseau2016groupes} or \cite[page 210]{marquis2018introduction}, but it is equivalent by \cite[Theorem 8.51]{marquis2018introduction}.

\subsubsection{Mathieu's group $\fG^{pma}$}

The \textbf{Borel subgroup} (it will be a subgroup of $\fG^{pma}$) is $\fB=\fB_{\cS}=\fT_{\cS}\ltimes \fU^{pma}$\index{B@$\fB$}, where $\fT$ acts on $\fU^{pma}$ as follows.  Let  $\sR$ be a ring, $\alpha\in \Delta_+$, $t\in \fT(\sR)$, $r\in \sR$ and $x\in \g_{\alpha,\sR}$,  \begin{equation}\label{e_KMT4}
t[\exp](rx)t^{-1}=[\exp](\alpha(t)rx).
\end{equation}
In particular, if $\alpha\in \Phi$, we have   \[\  tx_\alpha(r)t^{-1}=x_\alpha(\alpha(t)r).\]

For $i\in I$, let $\fU_{\alpha_i}^Y$ be the reductive group associated with the root generating system $\left((2),X,Y,\alpha_i,\alpha_i^\vee\right)$. For each $i\in I$, Mathieu defines an (infinite dimensional) affine group scheme $\fP_{i}=\fU_{\alpha_i}^Y \ltimes \fU^{ma}_{\Delta_+\setminus \{\alpha_i\}}$ (see \cite[Definition 8.65]{marquis2018introduction} for the definition of the action of $\fU_{-\alpha_i}$ on $\fU^{ma}_{\Delta_+\setminus \{\alpha_i\}}$), where $\fU_{-\alpha_i}$ is an affine group scheme over $\Z$ isomorphic to $\mathbb{G}_a$ (see \cite[7.4.3]{marquis2018introduction} for more details).

 We do not detail the definition of $\fG^{pma}$\index{G@$\fG^{pma}$} and we refer to \cite{mathieu1989construction}, \cite[8.7]{marquis2018introduction} or \cite[3.6]{rousseau2016groupes}. This is an ind-group scheme containing the $\fP_{i}$ for every $i\in I$. Let $w\in W^v$ and write $w=r_{i_1}\ldots r_{i_k}$, with $k=\ell(w)$ and $i_1,\ldots,i_k\in I$. Then the multiplication map $\fP_{i_1}\times \ldots  \times \fP_{i_k}\rightarrow \fG^{pma}$ is  a  scheme morphism, and we have $\fG^{pma}(\sR)=\bigcup_{(i_1,\ldots,i_n)\in \Red(W^v)} \fP_{i_1}(\sR)\times \ldots \times \fP_{i_k}(\sR)$, where $\Red(W^v)$ is the set of reduced words of $W^v$ (i.e $\Red(W^v)=\{(i_1,\ldots,i_k)\in I^{(\N)}\mid \ell(r_{i_1}\ldots r_{i_k})=k\}$). 

Let $w\in W^v$, $i\in I$ and $\alpha=w.\alpha_i$. One sets $\fU_{\alpha}=\tilde{w}.\fU_{\alpha_i} .\tilde{w}^{-1}$, where \begin{equation}\label{e_tilde_w}
\tilde{w}=\tilde{r}_{i_1}\ldots \tilde{r}_{i_k},
\end{equation} if $w=r_{i_1}\ldots r_{i_k}$ is a reduced decomposition of $w$. There is an isomorphism of group schemes $x_\alpha: \mathbb{G}_a\rightarrow \fU_{\alpha}$ (see \cite[page 262]{marquis2018introduction}). The group $\fG^{pma}$ is generated by the $\fP_i$, $i\in I$. Moreover, if $i\in I$, then $\fP_i$ is generated by $\fT$, $\fU_{\pm \alpha_i}$ and $\tilde{r}_{i}=x_{\alpha_i}(1)x_{-\alpha_i}(1)x_{\alpha_i}(1)$. Thus $\fG^{pma}$ is generated by $\fU^{pma}$, $\fT$, $\fU_{-\alpha_i}$ and the $\tilde{r}_{i}$, for $i\in I$  and thus we have: \begin{equation}\label{e_generators_Gpma}
\fG^{pma}=\langle \fU^{pma},\fT,\fU_\alpha,\alpha\in \Phi_-\rangle.
\end{equation}

There is a group functor morphism $\iota:\fG\rightarrow \fG^{pma}$ such that for any ring $\sR$, $\iota_\sR$ maps $x_\alpha(r)$ to $x_\alpha(r)$ and $t$ to $t$, for each $\alpha\in \Phi$, $r\in \sR$, $t\in \fT(\sR)$. When $\sR$ is a field, this morphism is injective (see \cite[3.12]{rousseau2016groupes} or \cite[Proposition 8.117]{marquis2018introduction}).

\begin{Proposition}\label{p_explicit_morphism}
Let $\sR$ and $\sR'$ be two rings and $\varphi:\sR\rightarrow \sR'$ be a ring morphism. 

 Let $f_\varphi^{\widehat{\cU}^+}:\widehat{\cU}_\sR^+\rightarrow \widehat{\cU}_{\sR'}^+$ and $f_\varphi:\fG^{pma}(\sR)\rightarrow \fG^{pma}(\sR')$ be the induced morphisms. Then $f_\varphi^{\widehat{\cU}^+}(\fU^{pma}(\sR))\subset \fU^{pma}(\sR')$ and we have: \begin{enumerate}
\item For every $(r_x)\in \sR^{\cB}$, $f_\varphi^{\widehat{\cU}^+}\left(\prod_{x\in \cB}[\exp](r_x x)\right)=\prod_{x\in \cB}[\exp](\varphi(r_x) x)$.

\item For $\alpha\in \Delta_+$ and $(\lambda_x)\in \sR^{\cB_\alpha}$, we have $f_\varphi\left(X_\alpha(\sum_{x\in \cB_\alpha} \lambda_x x)\right)=X_\alpha\left(\sum_{x\in \cB_\alpha}\varphi(\lambda_x) x\right)$. 

\item We have $f_{\varphi}(u)=f_\varphi^{\widehat{\cU}^+}(u)$ for $u\in \fU^{pma}(\sR)$, $f_{\varphi}(x_{\alpha}(r))=x_\alpha(\varphi(r))$, for $\alpha\in \Phi$ and $r\in \sR$, and $f_\varphi(\chi(r))=\chi(\varphi(r))$, for $\chi\in Y$ and $r\in \sR^\times$. 

\item If $\varphi$ is surjective, then $f_\varphi$ is surjective.
\end{enumerate}
\end{Proposition}

\begin{proof}
(1), (2) By definition, we have \[f_\varphi^{\widehat{\cU}^+}(\sum_{\alpha\in Q^+} \sum_{j\in J_\alpha} u_{\alpha,j}\otimes r_j)=\sum_{\alpha\in Q^+} \sum_{j\in J_\alpha} u_{\alpha,j}\otimes \varphi(r_j)\] if $J_\alpha$ is a finite set and $(r_j)\in \sR^{J_\alpha}$ and $u_{\alpha,j}\in \cU_{\alpha,\sR}$, for every $\alpha\in Q_+$. Thus $\varphi$ commutes with infinite sums and product, which proves (1) and (2). 

(3) Let $i\in I$. Then the morphism $\fP_i(\sR)\rightarrow \fP_i(\sR')$ induced by $\varphi$ satisfies the formula above. Using the fact that $x_\alpha=\tilde{w} x_{-\alpha_i} \tilde{w}^{-1}$, for $\alpha=-w.\alpha_i$, with $w\in W^v$, $i\in I$ and $\tilde{w}$ defined as in \eqref{e_tilde_w}, we have (3).

(4)  Assume $\varphi$ is surjective. By \eqref{e_normal_form_Upma} and (1), the restriction of $f_\varphi$ to $\fU^{pma}(\sR)$ is surjective. By (3), the restriction of  $f_{\varphi}(\fU^-(\sR))=\fU^-(\sR')$ and $f_\varphi(\fT(\sR))=\fT(\sR')$. We conclude by using the fact that $\fG^{pma}$ is generated by $\fU^{pma}$, $\fU^-$ and $\fT$ (see \eqref{e_generators_Gpma}).
\end{proof}

\subsubsection{Minimal Kac-Moody group over rings}

For $i\in I$, there is a natural group morphism $\varphi_i:\mathrm{SL}_2\rightarrow \fU_{\alpha_i}^Y$\index{P@$\varphi_i$}. 

For a ring $\sR$, one sets \[\fG^{\min}(\sR)=\langle \varphi_i\left(\mathrm{SL}_2(\sR)\right),\fT(\sR)\rangle\subset \fG^{pma}(\sR).\]\index{G@$\fG^{\min}$} This group is introduced by Marquis in  \cite[Definition 8.126]{marquis2018introduction}. By \cite[Proposition 8.129]{marquis2018introduction}, it is a nondegenerate Tits functor in the sense of \cite[Definition 7.83]{marquis2018introduction} and we have $\fG^{\min}(\sR)\simeq \fG(\sR)$, for any field $\sR$. 

Note that if $\varphi$ is a ring morphism between two rings $\sR$ and $\sR'$, the induced morphism $\fG^{pma}(\sR)\rightarrow \fG^{pma}(\sR')$ restricts to a morphism $\fG^{\min}(\sR)\rightarrow \fG^{\min}(\sR')$. 

Let $\sR$ be a semilocal ring, i.e a ring with finitely many maximal ideals, then by \cite[4.3.9 Theorem]{hahn1989classical}, $\mathrm{SL}_2(\sR)$ is generated by $\begin{psmallmatrix} 1 & \sR\\ 0 & 1\end{psmallmatrix}$ and  $\begin{psmallmatrix} 1 & 0\\ \sR & 1\end{psmallmatrix}$. Therefore, \begin{equation}\label{e_minimal_group_semilocal_ring}
\fG^{\min}(\sR)=\langle \fU_{\pm\alpha_i}(\sR),\fT(\sR)\mid i\in I\rangle\subset \fG^{pma}(\sR).
\end{equation}

\subsection{Split Kac-Moody groups over valued fields and masures}

We now fix a field $\cK$ equipped with a valuation $\omega:\cK\rightarrow \R\cup\{+\infty\}$\index{o@$\omega$} such that $\Lambda:=\omega(\cK^*)$\index{L@$\Lambda$} contains $\Z$. Let $\cO=\{x\in \cK\mid \omega(x)\geq 0\}$\index{O@$\cO$} be its ring of valuation.  We defined Mathieu's positive completion $\fG^{pma}$. Replacing $\Delta_+$ by $w.\Delta_+$, for $w\in W^v$, one can also define a group $\fG^{pma,w}$\index{G@$\fG^{pma,w}$}. Replacing $\Delta_+$ by $\Delta_-$ or by $w.\Delta_-$, for $w\in W^v$, one can also define $\fG^{nma}$\index{G@$\fG^{nma}$} or $\fG^{nma,w}$.

We set $G=\fG(\cK)$, $G^{pma}=\fG^{pma}(\cK)$ and $G^{nma}=\fG^{nma}(\cK)$.

\subsubsection{Action of $N$ on $\A$}

Let $N=\fN(\cK)$\index{N@$N$} and $\mathrm{Aut}(\A)$ be the group of affine automorphism of $\A$. Then by \cite[4.2]{rousseau2016groupes}, there exists a group morphism $\nu:N\rightarrow \mathrm{Aut}(\A)$\index{n@$\nu$} such that:\begin{enumerate}
\item for $i\in I$, $\nu(\tilde{r}_{i})$ is the simple reflection $r_i\in W^v$, it fixes $0$,

\item for $t\in \fT(\cK)$, $\nu(t)$ is the translation on $\A$ by the vector $\nu(t)$ defined by $\chi(\nu(t))=-\omega(\chi(t))$, for all $\chi\in X$. 

\item we have $\nu(N)=W^v\ltimes (Y\otimes \Lambda):=W_\Lambda$\index{W@$W_\Lambda$}.

\end{enumerate}

\subsubsection{Affine apartment}  
\label{2.2b}

 A \textbf{local face} in $\A$ is the germ $F(x,F^v)=germ_{x}(x+F^v)$ where $x\in\A$ and $F^v$ is a vectorial face (i.e, $F(x,F^v)$ is the filter of all neighbourhoods of $x$ in $x+F^v$).
It is a \textbf{local panel}, positive, or negative if $F^v$ is. If $F^v$ is a chamber, we call $F(x,F^v)$ an \textbf{alcove} (or a local chamber). 
We denote by $C_0^+$\index{C@$C_0^+$} the  \textbf{fundamental alcove}, i.e, $C_0^+=germ_0(C^v_f)$.

  A \textbf{sector} in $\A$ is a subset $\fq=x+C^v$, for $x$ a point in $\A$ and $C^v$ a vectorial chamber.
Its \textbf{sector germ} (at infinity) is the filter $\fQ=germ_{\infty}(\fq)$ of subsets of $\A$ containing another sector $x+y+C^v$, with $y\in C^v$.
It is entirely determined by its direction $C^v$.
This sector or sector germ is said positive (resp. negative) if $C^v$ has this property. We denote by $\pm\infty$\index{i@$\pm \infty$} the germ at infinity of $\pm C^v_f$.

For $\alpha\in \Delta$ and $k\in \Lambda\cup \{+\infty\}$, we set  $D(\alpha,k)=\{x\in \A\mid \alpha(x)+k\geq 0\}$\index{D@$D(\alpha,k)$}.  A set of the form $D(\alpha,k)$, for $\alpha\in \Delta$ and $k\in \Lambda$ is called a \textbf{half-apartment}.

\subsubsection{Parahoric subgroups}\label{3.1}

In \cite{rousseau2016groupes} and \cite{gaussent2008kac}, the masure $\I$ of $G$ is constructed as follows. To each $x\in \A$ is associated a group $\hat{P}_x=G_x$. Then $\I$ is defined  in such a way that $G_x$ is the fixator of $x$ in $G$ for the action on $\I$.  We actually associate to each filter $\Omega$ on $\A$ a subgroup $G_\Omega\subset G$ (with $G_{\{x\}}=G_x$ for $x\in \A$).  Even though the masure is not yet defined, we use the terminology ``fixator'' to speak of $G_\Omega$, as this will be the fixator of $\Omega$ in $G$. The definition of $G_\Omega$ involves the completed groups $G^{pma}$ and $G^{nma}$.

If $\Omega$ is a non empty subset of $\A$ we sometimes regard it as a filter on $\A$ by identifying it with the filter consisting of the subsets of $\A$ containing $\Omega$.  Let $\Omega\subset\A$ be a non empty set or filter.
One defines a function $f_{\Omega}:\Delta\to \R$ by \[f_{\Omega}(\alpha)=\inf\set{r\in\R \mid \Omega\subset D(\alpha+r)}=\inf\set{r\in\R \mid \alpha(\Omega)+r\subset[0,+\infty[},\]\index{f@$f_\Omega$}  for $\alpha\in \Delta$. For $r\in\R$, one sets $\cK_{\omega\geq r}=\set{x\in\cK \mid \omega(x)\geq r}$, $\cK_{\omega= r}=\set{x\in\cK \mid \omega(x) = r}$.  

If $\Omega$ is a set, we  define the subgroup $U^ {pma}_{\Omega}=\prod_{\alpha\in\Delta_+}\, X_{\alpha}(\g_{\alpha,\Z}\otimes \cK_{\omega\geq f_{\Omega}(\alpha)})\subset G^{pma}$\index{U@$U^{pma}_{\Omega}$}.
Actually, for $\alpha\in\Phi^+=\Delta_{re}^+$, $X_{\alpha}(\g_{\alpha,\Z}\otimes \cK_{\omega\geq f_{\Omega}(\alpha)})=x_{\alpha}(\cK_{\omega\geq f_{\Omega}(\alpha)})=:U_{\alpha,\Omega}$\index{U@$U_{\alpha,\Omega}$}. 
We then define \[U_{\Omega}^ {pm+}=U_{\Omega}^ {pma}\cap G=U_{\Omega}^ {pma}\cap U^+,\]\index{U@$U_\Omega^{pm+}$} see \cite[4.5.2, 4.5.3 and 4.5.7]{rousseau2016groupes}. When $\Omega$ is a filter, we set $U^ {pma}_{\Omega}:=\cup_{S\in\Omega}\, U^ {pma}_{S}$ and $U_{\Omega}^ {pm+}:=U_{\Omega}^ {pma}\cap G$

\par We may also consider the negative completion $G^ {nma}=\mathfrak{G}^ {nma}(\cK)$ of $G$, and define the subgroup $U^ {ma-}_{\Omega}=\prod_{\alpha\in\Delta_-}\, X_{\alpha}(\g_{\alpha,\Z}\otimes \cK_{\omega\geq f_{\Omega}(\alpha)})$.
For $\alpha\in\Phi^-=\Delta_{re}^-$, $X_{\alpha}(\g_{\alpha,\Z}\otimes \cK_{\omega\geq f_{\Omega}(\alpha)})=x_{\alpha}(\cK_{\omega\geq f_{\Omega}(\alpha)})=:U_{\alpha,\Omega}$.
We then define $U_{\Omega}^ {nm-}=U_{\Omega}^ {ma-}\cap G=U_{\Omega}^ {ma-}\cap U^-$.

Let $\Psi$ be a closed subset of $\Delta_+$. One sets $U_{\Omega}^{pm}(\Psi)=\fU_\Psi^{pma}(\cK)\cap U_\Omega^{pm+}$.\index{u@$U_{\Omega}^{pm}(\Psi)$} By the uniqueness in the decomposition of the elements of $U^{pma}$ as a product, every element of  $U_{\Omega}^{pm}(\Psi)$ belongs to  $\prod_{\alpha\in \Psi} X_\alpha(\cK_{\omega\geq f_\Omega(\alpha)})$. If $\Psi$ is a closed subset of $\Delta^-$, one sets $U_{\Omega}^{nm}(\Psi)=\fU_\Psi^{nma}(\cK)\cap U_{\Omega}^{nm-}$\index{u@$U_{\Omega}^{nm}(\Psi)$}. Note that $U_{\Omega}^{pm+}=U_{\Omega}^{pm}(\Delta_+)$ and $U_{\Omega}^{nm-}=U_{\Omega}^{nm}(\Delta_-)$.

 Let $\Omega$ be a filter on $\A$. We denote by $N_\Omega$\index{N@$N_\Omega$} the fixator of $\Omega$ in $N$ (for the action of $N$ on $\A$). If $\Omega$ is not a set, we have $N_\Omega=\bigcup_{S\in \Omega} N_S$. Note that we drop  the hats used in \cite{rousseau2016groupes}. When $\Omega$ is  open  one has $N_{\Omega}=N_{\A}=\fT(\cO):=\mathfrak{T}(\cK_{\omega\geq0})=\mathfrak{T}(\cK_{\omega=0})$\index{T@$\fT(\cO)$}.

  If $x\in \A$, we set $G_x=U_x^{pm+}.U_x^{nm-}.N_x$. This is a subgroup of $G$. If $\Omega\subset \A$ is a set, we set $G_\Omega=\bigcap_{x\in \Omega} G_x$\index{G@ $G_x$, $G_\Omega$} and if $\Omega$ is a filter, we set $G_\Omega=\bigcup_{S\in \Omega} G_S$. Note that in \cite{rousseau2016groupes}, the definition of   $G_x$ is much more complicated (see \cite[D\'efinition 4.13]{rousseau2016groupes}). However it is equivalent to this one by \cite[Proposition 4.14]{rousseau2016groupes}.

A filter is said to have a ``good fixator'' if it satisfies \cite[D\'efinition 5.3]{rousseau2016groupes}. There are many examples of filters with good fixators (see \cite[5.7]{rousseau2016groupes}): points, local faces, sectors, sector germs, $\A$, walls, half apartments, \ldots For such a filter $\Omega$, we have: \begin{equation}\label{e_good_fixator}
G_{\Omega}=U_{\Omega}^ {pm+}.U_{\Omega}^ {nm-}.N_{\Omega}=U_{\Omega}^ {nm-}.U_{\Omega}^ {pm+}.N_{\Omega}.
\end{equation}

 We then have: \begin{equation}\label{e_UOmega_fixators}
 U_{\Omega}^ {pm+}=G_{\Omega}\cap U^+ \text{ and }U_{\Omega}^ {nm-}=G_{\Omega}\cap U^-, 
\end{equation}
as $U^-\cap U^+.N=U^+\cap N=\set1$, by \cite[Remarque 3.17]{rousseau2016groupes}  and \cite[1.2.1 (RT3)]{remy2002groupes}.

When $\Omega=C_0^+=germ_0(C^v_f)$ is the (fundamental) positive local chamber in $\A$, $K_I:=G_\Omega$ is called the (fundamental) \textbf{Iwahori subgroup}. \index{K@$K_I$} When $\Omega$ is a face of $C_0^+$, $G_\Omega$ is called a \textbf{parahoric subgroup}.

For $\Omega$ a set or a filter, one defines:

\par  $U_{\Omega}=\langle U_{\alpha,\Omega} \mid \alpha\in\Phi \rangle$\quad, \quad $U^\pm_{\Omega}=U_{\Omega}\cap U^\pm$\quad and \quad$U_{\Omega}^ {\pm\pm}=\langle U_{\alpha,\Omega} \mid \alpha\in\Phi^\pm \rangle$. Then one has $U_{\Omega}^{++}\subset U_{\Omega}^+\subset U_{\Omega}^{pm+}$, but these inclusions are not equalities in general, contrary to the reductive case (see \cite[4.12.3 a and 5.7 3)]{rousseau2016groupes}).

\begin{Lemma}\label{l_uniqueness_decomposition_UUT}
Let   $(u_+,u_-,t),(u_+',u_-',t')\in U^+\times U^-\times T$. Assume  that $u_+ t u_-=u'_+ t'u'_-$ or $u_+ u_-t=u'_+  u'_-t'$ or $tu_+u_-=t'u_+' u_-'$. Then $u_-=u'_-$, $u_+=u'_+$ and $t=t'$.
\end{Lemma}

\begin{proof}
Assume $u_+t u_-=u'_+ t' u'_-$. We have $(u'_+)^{-1}u_+t =t'u_- (u'_-)^{-1}$. As $t$ normalizes $U^-$, we deduce the existence of $u_-''$ such that $(u'_+)^{-1} u_+ tt'^{-1}=u_-''$. By \cite[Proposition 1.5 (DR5)]{rousseau2006immeubles} (there is a misprint in this proposition, $Z$ is in fact $T$), we deduce $(u'_+)^{-1} u_+ tt'^{-1}=1$ and hence $u'_+=u_+$ and $t=t'$. Therefore $u_-=u'_-$. The other cases are similar.
\end{proof}

By \cite[4.10]{rousseau2016groupes} and \eqref{e_KMT4}, we have the following lemma.

\begin{Lemma}\label{l_action_T_U}
Let $\Omega$ be a filter on $\A$, $t\in T$ and $\Psi$ be a closed subset of $\Delta_+$ (resp. $\Delta_-$). Then $tU_{\Omega}^{pm+} t^{-1}=U_{t.\Omega}^{pm+}$, $tU_{\Omega}^{pm}(\Psi)t^{-1}=U_{t.\Omega}^{pm}(\Psi)$ (resp. $tU_{\Omega}^{nm-}(\Psi) t^{-1}=U_{t.\Omega}^{nm-}(\Psi)$).
\end{Lemma}

 \subsubsection{Masure}\label{n1.3.1}
 
 We now define the masure  $\I=\I(\fG,\cK,\omega)$\index{I@$\I$}. As a set, $\I=G\times \A/\sim$, where $\sim$ is defined as follows: \[\forall (g,x),(h,y)\in G\times \A, (g,x)\sim (h,y)\Leftrightarrow \exists n\in N\ \mid y=\nu(n).x\text{ and }g^{-1}hn\in G_x.\]
  We regard $\A$ as a subset of $\I$ by  identifying $x$ and $(1,x)$, for $x\in \A$. The group $G$ acts on $\I$ by $g.(h,x)=(gh,x)$, for $g,h\in G$ and $x\in \A$. An \textbf{apartment} is a set of the form $g.\A$, for $g\in G$. The stabilizer of $\A$ in $G$ is $N$ and if $x\in \A$, then the fixator of $x$ in $G$ is $G_x$. More generally, when $\Omega\subset \A$, then $G_{\Omega}$ is the fixator of $\Omega$ in $G$ and $G_{\Omega}$ permutes transitively the apartments containing $\Omega$.
   If $A$ is an apartment, we transport all the notions that are preserved by $W_\Lambda$ (for example segments, walls, faces, chimneys, etc.) to $A$.  Then by \cite[Corollary 3.7]{hebert2022new}, if $(\alpha_i)_{i\in I}$ and $(\alpha_i^\vee)_{i\in I}$ are free, then $\I$ satisfies the following axioms: 
 
 \medskip
 
 (MA II) : Let $A,A'$ be two apartments.  Then $A\cap A'$ is a finite intersection of half-apartments and there exists $g\in G$ such that $g.A=A'$ and $g$ fixes $A\cap A'$.

\medskip

\par (MA III): if $\fR$ is the germ of a splayed chimney and if $F$ is a local  face or a germ of a chimney, then there exists an apartment containing $\fR$ and $F$.

\medskip

We did not recall the definition of a chimney and we refer to \cite{rousseau2011masures} for such a definition. We will only use the fact that a sector-germ is a particular case of a germ of a splayed chimney.

We also have: \begin{itemize}
\item The stabilizer of $\A$ in $G$ is $N$ and $N$ acts on $\A\subset \I$ via $\nu$.

\item  The group $U_{\alpha,r}:=\{x_\alpha(u)\mid u\in \cK, \omega(u)\geq r\}$\index{U@$U_{\alpha,r}$}, for $\alpha\in\Phi,r\in\Lambda$, fixes the half-apartment $D(\alpha,r)$.   It acts simply transitively on the set of apartments in $\I$ containing $D(\alpha,r)$.
\end{itemize}

 The first point of the next proposition extends \cite[Proposition 7.4.8]{bruhat1972groupes} to masures.

\begin{Proposition}\label{p_BT_7.4.8}
\begin{enumerate}
    \item  Let $g\in G$. Then $\A\cap g^{-1}.\A$ is a finite intersection of half-apartments and there exists $n\in N$ such that $g.x=n.x$ for every $x\in \A\cap g^{-1}.\A$.
    
    \item Let $g\in G$. Then $\{x\in \A\mid g.x=x\}$ is convex. In particular if $\Omega$ is a subset of  $\A$, then $G_{\Omega}=G_{\mathrm{conv}(\Omega)}$, where $\conv(\Omega)$ is the convex hull of $\Omega$.
\end{enumerate}

\end{Proposition}

\begin{proof}
 1). Let $g\in G$.  We assume that $\A\cap g^{-1}.\A$ is non-empty. Then it is a finite intersection of half-apartments by (MA II) and  there exists $h\in G$ such that $hg.\A=\A$ and $h$ fixes $\A\cap g.\A$. Then $hg$ stabilizes $\A$ and thus it belongs to $N$, by \cite[5.7 5)]{rousseau2016groupes}. We get 1) by setting $n=hg$.
    
    2) Let $g\in G$, $\Omega_1=\A\cap g^{-1}.\A$ and $\Omega=\{x\in \A\mid g.x=x\}$. We have $\Omega\subset \Omega_1$. Let $n\in N$ be such that $g.x=n.x$ for all $x\in \Omega_1$. Let $f=\nu(n):\A\rightarrow \A$. Then $\Omega=\Omega_1\cap \{x\in \A\mid f(x)=x\}$. As $f$ is affine and $\Omega_1$ is convex, we have that $\Omega$ is convex.
\end{proof}

\begin{Remark}
In \ref{sub_Root_generating_system}, we did not assume the freeness of the families $(\alpha_i)_{i\in I}$ and $(\alpha_i^\vee)_{i\in I}$, since there are interesting Kac-Moody groups, which do not satisfy this assumption. For example,  $G:=\mathrm{SL}_n(\cK[u,u^{-1}])\rtimes \cK^*$ is naturally equipped with the structure of a Kac-Moody group associated with a root generating system $\mathcal{S}$ having nonfree coroots. This group is particularly interesting for examples, since it is one of the only Kac-Moody groups in which we can make explicit computations. In \cite{hebert2022new}, we proved that if $(\alpha_i)_{i\in I}$ and $(\alpha_i^\vee)_{i\in I}$ are free families, then the masure associated with $G$ satisfies (MA II) and (MA III). Without this assumption we do not know. In \cite[Théorème 5.16]{rousseau2016groupes}, Rousseau proves that $\I$ satisfies the axioms (MA2) to (MA5) of \cite{rousseau2011masures}. We did not introduce these axioms since they are more complicated and a bit less convenient. However it is easy to adapt the  proofs of this paper to use the axioms of \cite{rousseau2011masures} instead of those of \cite{hebert2022new}, for example,   retractions are already available in \cite{rousseau2011masures}.
\end{Remark}

\subsubsection{Retraction centred at a sector-germ}

Let $\fQ$ be a sector-germ of $\A$. If $x\in \I$, then by (MA III), there exists an apartment $A$ of $\I$ containing $\fQ$ and $x$. By (MA II), there exists $g\in G$ such that $g.A=\A$ and $g$ fixes $A\cap \A$. One sets $\rho_{\fQ}(x)=g.x\in \A$. This is well-defined, independently of the choices of $A$ and $g$, by (MA II). This defines the \textbf{retraction} $\rho_{\fQ}:\I\rightarrow \A$\index{r@$\rho_{\fQ}, \rho_{+\infty}$} \textbf{onto $\A$ and centred at $\fQ$}. When $\fQ=+\infty$, we denote it $\rho_{+\infty}$. If $x\in \I$, then $\rho_{+\infty}(x)$ is the unique element of $U^+.x\cap \A$.

\subsubsection{Topology defined by a filtration}\label{ss_topology}

A \textbf{filtration of $G$ by subgroups} is a sequence $(V_n)_{n\in \N^*}$ of subgroups of $G$ such that $V_{n+1}\subset V_n$ for all $n\in \N^*$. Let $(V_n)$ be a filtration of $G$ by subgroups. The associated topology $\sT\left((V_n)\right)$\index{T@$\sT\left((V_n)\right)$}  is the topology on $G$ for which a set $V$ is open if for all $g\in V$, there exists $n\in \N^*$ such that $g.V_n\subset V$.

Let $(V_n),(\tilde{V}_n)$ be two filtrations of $G$ by subgroups. We say that $(V_n)$ and $(\tilde{V}_n)$ are equivalent if for all $n\in \N$, there exist $m,\tilde{m}\in \N$ such that $V_{m}\subset \tilde{V}_n$ and $\tilde{V}_{\tilde{m}}\subset V_n$. This defines an equivalence relation on the set of filtrations of $G$ by subgroups. Then  $(V_n)$ and $(\tilde{V}_n)$ are equivalent filtrations, if and only if $\sT\left((V_n)\right)=\sT\left((\tilde{V}_n)\right)$. 

 We say that $(V_n)$ is \textbf{conjugation-invariant} if for all $g\in G$, $(gV_n g^{-1})$ is equivalent to $(V_n)$. Then $\sT\left((V_n)\right)$ equips $G$ with the structure of a topological group if and only if $(V_n)$ is conjugation invariant, by \cite[Exercise 8.5]{marquis2018introduction}.

\section{Congruence subgroups}\label{s_Congruence_subgroups}

In this section, we define and study the congruence subgroups. They will be  a key tool in order to define the topology $\sT$ in section~\ref{s_Def_topologies}. We prove however that the filtration $(\ker \pi_n)_{n\in \N^*}$  is not conjugation-invariant. We also study how they decompose.

\subsection{Definition of the congruence subgroup}

\begin{Proposition}\label{p_fixator_zero}
The fixator $G_0$ of $0$ in $G$ is the  group $\fG^{\mathrm{min}}(\cO)$. 
\end{Proposition}

\begin{proof}
For $i\in I$, $x_{\alpha_i}(\cO)$, $x_{-\alpha_i}(\cO)$ and $\fT(\cO)$ fix $0$. Therefore by \eqref{e_minimal_group_semilocal_ring},  $\fG^{\min}(\cO)\subset G_0$. 

By \cite[Proposition 4.14]{rousseau2016groupes} \begin{equation}\label{e_G0}
G_0=U_0^{pm+} U_0^{nm-} N_0,
\end{equation} where $N_0=\{n\in N\mid n.0=0\}$.  By \cite[2.4.1 2)]{bardy2022twin}, we have \[U_0^{pm+}=U_0^+:=\langle x_\alpha(u)\mid \alpha\in \Phi, u\in \cO\rangle\cap U^+\subset \fG^{\min}(\cO)\] and \[U_0^{nm-}=U_0^-:=\langle x_\alpha(u)\mid \alpha\in \Phi, u\in \cO\rangle\cap U^-\subset \fG^{\min}(\cO).\]  For $i\in I$, set $\tilde{r}_{i}=x_{\alpha_i}(1) x_{-\alpha_i}(-1)x_{\alpha_i}(1)\in \fG^{\min}(\cO)$.  We have  $N=\langle \fT(\cK), \tilde{r}_{i}\mid i \in I\rangle$. Let $n\in N_0$. Write $\nu^v(n)=w\in W^v$, where $\nu^v$ was defined in \ref{ss_m_KM_grp}. Write $w=r_{i_1}\ldots r_{i_k}$, with $k=\ell(w)$ and $i_1,\ldots,i_k\in I$. Let $n'=\tilde{r}_{i_1}\ldots \tilde{r}_{i_k}\in N_0$. By  \cite[1.6 4)]{rousseau2016groupes} $\nu^v(n')=w$ and  $t:=n'^{-1}n\in T\cap \ker(\nu)$. By \cite[4.2 3)]{rousseau2016groupes}, $t\in \fT(\cO)$. Therefore
 \begin{equation}\label{e_N0}
N_0=\langle \tilde{r}_{i}\mid i\in I\rangle .\fT(\cO).
\end{equation} and in particular, $N_0\subset \fG^{\min}(\cO)$. Proposition follows. 
\end{proof}

Recall that we assumed that $\Lambda=\omega(\cK^*)\supset \Z$. If $\omega(\cK^*)$ is discrete, we can normalize $\omega$ so that $\Lambda=\Z$. We fix $\qp\in \cO$\index{p@$\qp$} such that $\omega(\qp)=1$.

For $n\in \N^*$, we denote by $\pi_n^{pma}:\fG^{pma}(\cO)\rightarrow \fG^{pma}(\cO/\qp^n \cO)$\index{p@$\pi_n^{pma}$, $\pi_n^{nma}$, $\pi_n$}  and $\pi_n^{nma}:\fG^{nma}(\cO)\rightarrow \fG^{nma}(\cO/\qp^n \cO)$ the morphisms associated with the canonical projection $\cO\twoheadrightarrow \cO/\qp^n \cO$. We denote by $\pi_n$ the restriction of $\pi_n^{pma}$ to $\fG^{\min}(\cO)$.  By Proposition~\ref{p_explicit_morphism}, $\pi_n$ is also the restriction of $\pi^{nma}_n$ to $\fG^{\min}(\cO)$ (it is also the restrictions of $\pi_n^{pma,w}:\fG^{pma,w}(\cO)\rightarrow \fG^{pma,w}(\cO/\qp^n \cO)$ and $\pi_n^{nma,w}:\fG^{nma,w}(\cO)\rightarrow \fG^{nma,w}(\cO/\qp^n \cO)$, for $w\in W^v$). By Proposition~\ref{p_explicit_morphism} and \eqref{e_minimal_group_semilocal_ring}, $\pi_n$, $\pi_n^{pma}$ and $\pi_n^{nma}$ are surjective. Their kernels are respectively called the \textbf{$n$-th congruence subgroups} of $\fG^{\min}(\cO)$, $\fG^{pma}(\cO)$ and $\fG^{nma}(\cO)$.

The family $(\ker \pi_n)_{n\in \N^*}$ is a filtration of $G$. We prove below that it is not conjugation-invariant when $W^v$ is infinite, which   motivates the introduction of other filtrations $(\cV_{n\lambda})_{n\in\N^*}$, for $\lambda\in Y^+$  regular, in section~\ref{s_Def_topologies}.

\begin{Lemma}\label{l_action_pin_masure}
Let $x\in \A$ be such that $\alpha_i(x)>0$ for all $i\in I$. Suppose that $W^v$ is infinite. Then for all $n\in\N^*$, there exists $g\in \ker(\pi_n)$ such that $g.x\neq x$. 
\end{Lemma}

\begin{proof}
Let $n\in \N^*$. As $\Phi^+$ is infinite, there exists $\beta\in \Phi^+$ such that $\htt(\beta)>\frac{n}{\min_{i\in I} \alpha_i(x)}$. Then $\beta(x)>n$. Let $g=x_{-\beta}(\qp^n)\in \ker \pi_n$. Then the subset of $\A$ fixed by $g$ is $\{y\in \A\mid -\beta(y)+n\geq 0\}$, which does not contain $x$.
\end{proof}

\begin{Lemma}\label{l_non_conjugacy_invariance_congruence}
Assume that $W^v$ is infinite. Then $(\ker(\pi_n))_{n\in\N^*}$ is not conjugation-invariant. 
\end{Lemma}

\begin{proof}
Suppose that $(\ker (\pi_n))$ is conjugation-invariant. Then the topology $\sT((\ker (\pi_n))$ equips $G$ with the structure of a topological group. We have $\ker(\pi_1)\subset \fG^{\min}(\cO)=G_0$ and in particular $G_0=\bigcup_{g\in G_0} g.\ker(\pi_1)$ is open. Let $\lambda\in  Y^+$ be such that $\alpha_i(\lambda)=1$ for all $i\in I$ and $t\in T$ be such that $t.0=\lambda$. Then $H:=t G_0 t^{-1}$ is open (since $G$ is a topological group). As $1\in H$, we deduce the existence of $n\in \N^*$ such that $\ker(\pi_n)\subset H$. As $H$ fixes $\lambda$, this implies that $W^v$ is finite, by Lemma~\ref{l_action_pin_masure}. 
\end{proof}

\subsection{On the decompositions of the congruence subgroups}

Let $\fm=\{x\in \cO\mid \omega(x)>0\}$\index{m@$\fm$} be the maximal ideal of $\cO$ and $\kk=\cO/\fm$\index{k@$\kk$}. Let $\pi_\kk:\fG^{\min}(\cO)\rightarrow \fG^{\min}(\kk)$\index{p@$\pi_\kk$} be the morphism induced by the natural projection $\cO\twoheadrightarrow \kk$. When $\omega(\cK^*)=\Z$, $\pi_\kk=\pi_1$.

In this subsection we study $\ker \pi_\kk$: we prove that it decomposes as the product of its intersections with $U^-$, $U^+$ and $T$ (see Proposition~\ref{p_ker_pi_k}), using the masure $\I$ of $G$. We also describe  $U^-\cap \ker \pi_\kk$ and $U^+\cap \ker \pi_\kk$ through their actions on $\I$ and we deduce that $\ker \pi_\kk$ fixes $C_0^+\cup C_0^-$.  It would be interesting to prove similar properties for $\ker \pi_n$ instead of $\ker \pi_\kk$, for $n\in \N^*$. The difficulty is that when $\omega$ is not discrete or $n\geq 2$, $\cO/\qp^n\cO$ is no longer a field and very few is known for Kac-Moody groups over rings. 

Let $C,C'$ be two alcoves of the same sign based at $0$. By \cite[Proposition 5.17]{hebert2020new}, there exists an apartment $A$ containing $C$ and $C'$. Let $g\in G$ be such that  $g.A=\A$ and $g.C=C_0^+$. Then $g.C'$ is an alcove of $\A$ based at $0$ and thus there exists $w\in W^v$ such that $g.C'=w.C_0^+$. We set $\dw(C,C')=w$\index{d@$\dw$}, which is well-defined, independently of the choices we made (note that in   \cite[1.11]{bardy2016iwahori} the ``$W$-distance'' $\dw$ is defined for more general pairs of alcoves). Then $\dw$ is $G$-invariant.

\begin{Lemma}\label{l_fiber_retraction}
Let $C$ be a positive alcove of $\I$ based at $0$ and $w\in W^v$.  Write $w=r_{i_1}\ldots r_{i_k}$, with $k=\ell(w)$ and $i_1,\ldots,i_k\in I$. Let $\beta_{1}=\alpha_{i_1}$, $\beta_2=r_{i_1}.\alpha_{i_2}$, $\ldots$, $\beta_k=r_{i_1}\ldots r_{i_{k-1}}.\alpha_{i_k}$. Then $\beta_1,\ldots,\beta_k\in \Phi_+$  and we have $\rho_{+\infty}(C)=w.C_0^+$ if and only if  there exists $a_1,\ldots,a_k\in \cO$ such that $C=x_{\beta_1}(a_1)\ldots x_{\beta_k}(a_k).\tilde{r}_{i_1}\ldots \tilde{r}_{i_k}.C_0^+$. 
\end{Lemma}

\begin{proof}
As $x_{\beta_1}(\cO).\ldots.x_{\beta_k}(\cO)$ fixes $0$, an element of $x_{\beta_1}(\cO)\ldots x_{\beta_k}(\cO).\tilde{r}_{i_1}\ldots \tilde{r}_{i_k}.C_0^+$ is a positive alcove based at $0$. The fact that $\beta_1,\ldots,\beta_k\in \Phi_+$ follows from \cite[1.3.14 Lemma]{kumar2002kac}. Thus $x_{\beta_1}(\cO).\ldots.x_{\beta_k}(\cO).\tilde{r}_{i_1}\ldots \tilde{r}_{i_k}.C_0^+\subset U^+.w.C_0^+$  and we have one implication. 

 We prove the reciprocal by induction on $\ell(w)$. Assume $w=1$. Then $\rho_{+\infty}(C)=C_0^+$. Let $A$ be an apartment containing $+\infty$ and $C$. Let $g\in G$ be such that $g.A=\A$ and $g$ fixes $A\cap \A$. Then $C=g^{-1}.C_0^+$, by definition of $\rho_{+\infty}$. Moreover, $A$ contains $0$ and $+\infty$ and thus it contains $\conv(0,+\infty)\supset C_0^+$. Therefore $C=C_0^+$ and the lemma is clear in this case. Assume now that $\ell(w)\geq 1$. 

Let $C_0'=C_0^+$, $C_1'=r_{i_1}.C_0^+$, $\ldots$, $C_k'=r_{i_1}\ldots r_{i_k}.C_0^+=C'$. Let $C$ be a positive alcove based at $0$ and such that $\rho_{+\infty}(C)=C_k'$. Let $A$ be an apartment containing $C$ and $+\infty$. Let $g\in G$ be such that $g.A=\A$ and $g$ fixes $A\cap \A$.  Set $C_i=g^{-1}.C_i'$, for $i\in \llbracket 0,k\rrbracket$. Then $g$ fixes $+\infty$ and hence $\rho_{+\infty}(x)=g.x$ for every $x\in \A$. Therefore $\rho_{+\infty}(C_i)=C_i'$, for $i\in \llbracket 0,k\rrbracket$. In particular, $\rho_{+\infty}(C_{k-1})=C_{k-1}'$. By induction, we may assume that there exist  $a_1,\ldots,a_{k-1}\in \cO$ such that  $C_{k-1}=u\tilde{v}.C_0^+$, where $u=x_{\beta_{1}}(a_1)\ldots x_{\beta_k}(a_{k-1})$ and $\tilde{v}=\tilde{r}_{i_1}\ldots \tilde{r}_{i_{k-1}}.$  Moreover we have \[\dw(C_{k-1},C_k)=\dw(C_{k-1}',C_k')=r_{i_k}=\dw\left(u^{-1}.C_{k-1},u^{-1}.C_k\right)=\dw\left(\tilde{v}.C_0^+,u^{-1}.C_k\right).\] 

Let $P$ be the panel common to $\tilde{v}.C_0^+$ and $u^{-1}.C_k$. Then $P\subset \beta_k^{-1}(\{0\})$. Let $D$ be the half-apartment delimited by $\beta_{k}^{-1}(\{0\})$ and containing $C_{k-1}$. Then as $\beta_{k}(C_{k-1})=r_{i_1}\ldots r_{i_{k-1}}.\alpha_{i_k}(r_{i_1}\ldots r_{i_{k-1}}.C_0^+)>0$, $D$ contains $+\infty$. By \cite[Proposition 2.9]{rousseau2011masures}, there exists an apartment $B$ containing $D$ and $u^{-1}.C_k$. Let $g'\in G$ be such that $g'.B=\A$ and $g'$ fixes $\A\cap B$. We have $g'.u^{-1}.C_k=\rho_{+\infty}(C_k)=C_k'$. By \cite[5.7 7)]{rousseau2016groupes}, $g'\in \fT(\cO) U_{\beta_k,0}$ and as $\fT(\cO)$ fixes $\A$, we can assume $g'\in U_{\beta_k,0}=x_{\beta_k}(\cO)$. Write $g'=x_{\beta_k}(-a_k)$, with $a_k\in \cO$. Then $C_k=u.x_{\beta_k}(a_k).C_k=x_{\beta_1}(a_1)\ldots x_{\beta_k}(a_k).\tilde{r}_{i_1}\ldots \tilde{r}_{i_k}.C_0^+$, which proves the lemma.
\end{proof}

\begin{Proposition}\label{p_ker_pi_k}
\begin{enumerate}
\item We have $U_{C_0^+}^{nm-}=U^-\cap \ker(\pi_\kk)$ and $U_{-C_0^+}^{pm+}=U^+\cap \ker(\pi_\kk)$.

\item We have $\ker(\pi_\kk)=\left(\ker(\pi_\kk)\cap U^+\right).\left(\ker(\pi_\kk)\cap U^-\right).\left(\ker(\pi_\kk)\cap \fT(\cO)\right).$

\item We have $\ker(\pi_\kk)\subset G_{C_0^+\cup C_0^-}$. 
\end{enumerate}  
\end{Proposition}

\begin{proof}
2) Let $u\in U_{C_0^+}^{nm-}$.   By definition, there exists $\Omega\in C_0^+$ such that $u\in U_{\Omega}^{nm-}$. Let $x\in C^v_f\cap \Omega$. Then   $u\in \prod_{\alpha\in \Delta_-} X_\alpha(\g_{\alpha,\Z}\otimes \cK_{\omega\geq -\alpha(x)})\cap G_0$. As $-\alpha(x)>0$ for every $\alpha\in \Delta_-$, Proposition~\ref{p_explicit_morphism} implies: \begin{equation}\label{e_relation_kerpi_UCO}
  U_{C_0^+}^{nm-}\subset \ker(\pi_{\kk}).
  \end{equation}

 Let $g\in \ker(\pi_\kk)\subset \fG^{\min}(\cO)$. Then $g$ fixes $0$ and $g.C_0^+$ is a positive alcove based at $0$. Write $\rho_{+\infty}(g.C_0^+)=w.C_0^+$, with $w\in W^v$. Write $w=r_{i_1}\ldots r_{i_m}$, with $m=\ell(w)$ and $i_1,\ldots,i_m\in I$. Let $\tilde{n}=\tilde{r}_{i_1}\ldots \tilde{r}_{i_m}\in \fN(\cK)$. By Lemma~\ref{l_fiber_retraction}, there exists $u\in \langle x_\beta(\cO)\mid \beta\in \Phi_+\rangle$ such that $g.C_0^+=u\tilde{n}.C_0^+$. Then $g=u\tilde{n}i$, with $i\in G_{C_0^+}$. As $C_0^+$ has a good fixator (\cite[5.7 2)]{rousseau2016groupes}), we have (by \eqref{e_good_fixator}: \[G_{C_0^+}=U_{C_0^+}^{pm+} U_{C_0^+}^{nm-} N_{C_0^+}.\] As every element of $C_0^+$ has non empty interior, $N_{C_0^+}=\fT(\cO)$.  Moreover, $U_{C_0^+}^{pm+}=U_0^{pm+}$ and $\cT(\cO)$ normalizes $U_{0}^{pm+}$ and $U_{C_0^+}^{nm-}$. Therefore, \[G_{C_0^+}=\fT(\cO).U_{0}^{pm+}.U_{C_0^+}^{nm-}.\] Write $i=tu_+ u_-$, with $t\in \fT(\cO)$, $u_+\in U_{0}^{pm+}$ and $u_-\in U_{C_0^+}^{nm-}$.

 Therefore by \eqref{e_relation_kerpi_UCO}, we have \[\pi_{\kk}(g)=1=\pi_{\kk}(u\tilde{n}tu_+u_-)=\pi_{\kk}(u)\pi_{\kk}(\tilde{n}t)\pi_{\kk}(u_+)\pi_{\kk}(u_-)=\pi_{\kk}(u)\pi_{\kk}(\tilde{n}t)\pi_{\kk}(u_+) .\]
 
 By \cite[3.16 Proposition]{rousseau2016groupes}  or  \cite[Theorem 8.118]{marquis2018introduction}, 

$(\fG^{pma}(\k),\fN(\kk),\fU^{pma}(\kk),\fU^{-}(\kk),\fT(\kk),\{r_i\mid i\in I\})$ is a refined Tits system. By \cite[3.16 Remarque]{rousseau2016groupes}, we have the Birkhoff decomposition \[\fG^{pma}(\kk)=\bigsqcup_{n\in \fN(\kk)} \fU^{+}(\kk) n \fU^{pma}(\kk).\]  As $\pi_{\kk}(u)\in \fU^+(\kk)$, $\pi_{\kk}(\tilde{n}t)\in \fN(\kk)$ and $\pi_{\kk}(u_+)\in \fU^{pma}(\kk)$, we deduce $\pi_{\kk}(\tilde{n}t)=1$.

By \cite[1.4 Lemme and 1.6]{rousseau2006immeubles}, there exists a group morphism $\nu^v_\kk:\fN(\kk)\rightarrow W^v$ such that $\nu^v_\kk(\tilde{r}_{i})=r_i$ for $i\in I$ and $\nu^v_\kk(\fT(\kk))=1$. Then $\nu_\kk^v(\tilde{n}t)=w=1$. Therefore $w=1$ and $g=utu_+u_-=u'tu_-$, for some $u'\in U_0^{pm+}$, since $t$  normalizes $U_0^{pm+}$. By Lemma~\ref{l_uniqueness_decomposition_UUT} and by symmetry of the roles of $U^{nm-}$ and $U^{pm+}$, we have $u'\in U_{-C_0^+}^{pm+}$. By \eqref{e_relation_kerpi_UCO} applied to $U_{-C_0^+}^{pm+}$, we have $\pi_\kk(g)=1=\pi_\kk(u')\pi_{\kk}(t)\pi_\kk(u-)=\pi_\kk(t)$ and thus \[g\in U_{-C_0^-}^{pm+}. (T\cap \ker \pi_\kk). U_{C_0^+}^{nm-}= U_{-C_0^-}^{pm+}.U_{C_0^+}^{nm-}.(T\cap \ker \pi_\kk).\]  
By \eqref{e_relation_kerpi_UCO}, we deduce 2). 

3) We have $U_{-C_0^{+}}^{pm+}=U_{-C_0^+\cup C^v_f}^{pm+}\subset G_{C_0^+\cup -C_0^+}$, $U_{C_0^+}^{nm-}=U_{C_0^+\cup -C^v_f}^{nm-}\subset G_{C_0^+\cup-C_0^+}$ and $T\cap \ker \pi_\kk \subset T\cap G_0\subset G_\A$, which proves 3). 

1) We already proved one inclusion. Let $u\in \ker \pi_\kk\cap U^-$. Then by what we proved above, $u\in U_{-C_0^-}^{pm+}. (T\cap \ker \pi_\kk). U_{C_0^+}^{nm-}$. By Lemma~\ref{l_uniqueness_decomposition_UUT}, $u\in U_{C_0^+}^{nm-}$, and the proposition follows. 
\end{proof}

\begin{Corollary}\label{c_ker_pin}
Let $n\in \N^*$. Then $\ker \pi_n\subset U_{-C_0^+}^{pm+} .\fT(\cO).U_{C_0^+}^{nm-}$.
\end{Corollary}

\begin{Remark}
Let $u\in U^-\cap \ker \pi_\kk=U_{C_0}^{nm-}$. Write $u=\prod_{\alpha\in \Delta_-}X_\alpha(v_\alpha)$, where $v_\alpha\in \g_{\alpha,\Z}\otimes \cO$, for every $\alpha\in \Delta_-$. Define $\omega:\g_{\alpha,\Z}\otimes \cK\rightarrow \R\cup \{+\infty\}$ by $\omega(v)=\inf \{x\in \R\mid v\in \g_{\alpha,\Z}\otimes \cK_{\omega\geq x}\}$, for $v\in \g_{\alpha,\Z}\otimes \cK$. Let $\lambda\in Y$ be such that $\alpha_i(\lambda)=1$ for every $i\in I$. Let $\Omega\in C_0^+$ be such that $u\in U_\Omega^{nm-}$ and  $\eta\in \R^*_+$ be such $\eta\lambda\in \Omega$. Then $u\in U_{\Omega}^{nm-}$ implies $\omega(v_\alpha)\geq |\alpha(\eta \lambda)|=\eta \htt(\alpha)$ for every $\alpha\in \Delta_-$. In particular, $\omega(v_\alpha)$ goes to $+\infty$ when $-\htt(\alpha)$ goes to $+\infty$.
\end{Remark}

\section{Definition of topologies on $G$}\label{s_Def_topologies}

In this section, we define two topologies $\sT$ and $\sT_{\Fix}$ on $G$ and compare them. For the first one, we proceed as follows. We define a set $\cV_\lambda$ for every regular $\lambda\in Y^+$. We prove that it is actually a subgroup of $\fG^{\min}(\cO)$ (Lemma~\ref{l_decomposition_Vlambda}) and we define $\sT$ as the topology associated with $(\cV_{n\lambda})_{n\in \N^*}$.  We then prove that $\sT$ does not depend on the choice of $\lambda$ and that it is conjugation-invariant (Theorem~\ref{t_conjugation_invariance}) and thus that $(G,\sT)$ is a topological group.  We then introduce the topology $\sT_\Fix$ associated with the fixators of finite subsets of $\I$ and we end up by a comparison of $\sT$ and $\sT_\Fix$. 

\subsection{Subgroup $\cV_{\lambda}$}

For $n\in \N^*$, we set $T_n=\ker \pi_n\cap T\subset \fT(\cO)$\index{T@$T_n$}. For $\lambda\in Y^+$ regular, we set \[N(\lambda)=\min \{|\alpha(\lambda)|\mid \alpha\in \Phi\}\in \N^*\text{ and }\cV_{\lambda}=U_{[-\lambda,\lambda]}^{pm+}. U_{[-\lambda,\lambda]}^{nm-}.T_{2N(\lambda)}.\]\index{V@$\cV_{\lambda}$}\index{N@$N(\lambda)$} By \eqref{e_UOmega_fixators}, we have \begin{equation}\label{e_cV_as_fixators}
\cV_{\lambda}=(U^+\cap G_{[-\lambda,\lambda]}).(U^-\cap G_{[-\lambda,\lambda]}). (T\cap \ker \pi_{2N(\lambda)})\subset G_{[-\lambda,\lambda]}.
\end{equation}

The $2N(\lambda)$ appearing comes from $x_{-\alpha}(\qp^n \cO).x_{\alpha}(\qp^n \cO)\subset x_\alpha(\qp^n \cO) x_{-\alpha}(\qp^n \cO) \alpha^\vee(1+\qp^{2n}\cO)$ for $\alpha\in \Phi$ and $n\in \N^*$, which follows from \eqref{e_Commutation_relation}. To prove that $\cV_\lambda$ is a group, the main difficulty is to prove that it is stable by left multiplication by $U_{[-\lambda,\lambda]}^{nm-}$. If $G$ is reductive, we have  $U_{[-\lambda,\lambda]}^{nm-}=U_{[-\lambda,\lambda]}^{--}:=\langle x_\alpha(a)\mid \alpha\in \Phi_-,a\in \cK, \omega(a) \geq |\alpha(\lambda)|\rangle$. By induction,  it then suffices to prove that $x_{-\alpha}(\qp^{|\alpha(\lambda)|} \cO) \cV_\lambda\subset \cV_\lambda$, for  $\alpha\in \Phi_+$.  When $G$ is no longer reductive, we have $U_{[-\lambda,\lambda]}^{--}\subsetneq U_{[-\lambda,\lambda]}^{nm-}$ in general (see \cite[4.12 3)]{rousseau2016groupes}). The group $U_{[-\lambda,\lambda]}^{nm-}$ is defined as a set of infinite products and so its seems difficult to reason by induction in our case. We could try to use the group $U_{[-\lambda,\lambda]}^{-}:=\langle x_\alpha(a)\mid \alpha\in \Phi,a\in \cK, \omega\geq |\alpha(a)|\rangle\cap U^-$ since it is sometimes equal to $U_{[-\lambda,\lambda]}^{nm-}$ (for example when $\lambda\in C^v_f$, $U_{[-\lambda,\lambda]}^{nm-}=U_{-\lambda}^{nm-}=U_{\lambda}^{-}=U_{[-\lambda,\lambda]}^{-}$ by \cite[2.4.1 2)]{bardy2022twin}). However it seems difficult since if $\alpha\in \Phi_+$, the condition $\omega(a)+\alpha(\lambda)\geq 0$ allows elements with a negative valuation. In order to overcome these difficulties, we use the morphisms $\pi_{n}$, for $n\in \N$.

By definition, \[U^{pm+}_{[-\lambda,\lambda]}=G\cap \prod_{\alpha\in \Delta_+} X_\alpha \left(\g_{\alpha,\Z}\otimes \cK_{\omega\geq f_\Omega([-\lambda,\lambda])}\right)\subset G\cap \prod_{\alpha\in \Delta_+} X_\alpha \left(\g_{\alpha,\Z}\otimes \qp^{N(\lambda)}\cO \right) .\]
 
 By Proposition~\ref{p_explicit_morphism} we deduce $U_{[-\lambda,\lambda]}^{pm+}\subset \ker \left(\pi_{N(\lambda)}\right)$. Using a similar reasoning for $U_{[-\lambda,\lambda]}^{nm-}$ we deduce  \begin{equation}\label{e_inclusion_VnLmabda_ker_Pi_nLambda}
 \cV_{\lambda}\subset \ker(\pi_{N(\lambda)}). 
 \end{equation}

For $n\in \N$ and $\alpha^\vee\in \Phi^\vee$, one sets $T_{\alpha^\vee,n}=\alpha^\vee(1+\qp^n \cO)\subset \fT(\cO)$\index{T@$T_{\alpha^\vee,n}$}. 

\begin{Lemma}\label{l_refinement_GR08_L3.3}
Let $\lambda\in Y^+$ be regular and $\alpha\in \Phi$. Then $U_{\alpha,[-\lambda,\lambda]}.U_{-\alpha,[-\lambda,\lambda]}.T_{\alpha^\vee,2N(\lambda)}$ is a subgroup of $G$. 
\end{Lemma}

\begin{proof}
Set $\Omega=[-\lambda,\lambda]$.  Set $H=U_{\alpha,\Omega}. U_{-\alpha,\Omega}.T_{\alpha^\vee,2N(\lambda)}$. It suffices to prove that $H$ is stable under left multiplication by $U_{\alpha,\Omega}, U_{-\alpha,\Omega}$ and  $T_{\alpha^\vee,2N(\lambda)}$. The first stability is clear and the third follows from Lemma~\ref{l_action_T_U} and the fact that $T_{\alpha^\vee,2N(\lambda)}\subset \fT(\cO)$ fixes $\A$.  Let $u_-,\tilde{u_-}\in U_{-\alpha,\Omega}$, $u_+\in U_{\alpha,\Omega}$ and $t\in T_{\alpha^\vee,2N(\lambda)}$. Write $u_-=x_{-\alpha}(a_-)$, $\tilde{u}_-=x_{-\alpha}(\tilde{a}_-)$ and $u_+=x_\alpha(a_+)$, for $a_-,\tilde{a}_-,a_+\in \cK$. We have $\omega(\tilde{a}_-),\omega(a_-),\omega(a_+)\geq |\alpha(\lambda)|\geq N(\lambda)$.  Then by \eqref{e_Commutation_relation}, we have \[\tilde{u}_- u_+ u_-t=x_\alpha\left(a_+(1+\tilde{a}_-a_+)^{-1}\right) x_{-\alpha}\left(\tilde{a}_-(1+\tilde{a}_- a_+)^{-1}+a_-\right)\alpha^\vee(1+\tilde{a}_- a_+)t.\] As $\omega(1+\tilde{a}_-a_+)=1$, we deduce that $\tilde{u}_- u_+ u_-t\in H$, which proves the lemma. 
\end{proof}

Let $w\in W^v$ and $\Omega$ be a filter on $\A$. Recall that $\fG^{pma,w}$ and $\fG^{nma,w}$ are the  completions of $\fG$  with respect to $w.\Delta_+$ and $w.\Delta_-$ respectively. One defines $U_{\Omega}^{pm}(w.\Delta_+)$\index{U@$U_{\Omega}^{pm}(w.\Delta_+), U_{\Omega}^{nm}(w.\Delta_-)$} and $U_{\Omega}^{nm}(w.\Delta_-)$ similarly as $U_{\Omega}^{pm+}$ and $U_{\Omega}^{nm-}$ in these groups. 

\begin{Lemma}\label{l_decomposition_Vlambda}
Let $\lambda\in Y^+$ be regular.   Then \begin{enumerate}
\item $\cV_{\lambda}=U_{[-\lambda,\lambda]}^{pm}(w.\Delta_+). U_{[\lambda,\lambda]}^{nm}(w.\Delta_-).T_{2N(\lambda)}$ for every $w\in W^v$,

\item $\cV_\lambda$ is a subgroup of $G$.
\end{enumerate}
\end{Lemma}

\begin{proof}
(1) This follows the proof of \cite[Proposition 3.4]{gaussent2008kac}. Set $\Omega=[-\lambda,\lambda]$. Let $i\in I$ and $\alpha=\alpha_i$. By \cite[3.3.4)]{gaussent2008kac} and Lemma~\ref{l_refinement_GR08_L3.3}, \[\begin{aligned}\cV_\lambda&=U_{\Omega}^{pm}(\Delta_+\setminus \{\alpha\}).U_{\Omega}^{nm}(\Delta_-\setminus\{\alpha\}).U_{\alpha,\Omega} .U_{-\alpha,\Omega} .T_{2N(\lambda)}\\ &=U_{\Omega}^{pm}(\Delta_+\setminus \{\alpha\}).U_{\Omega}^{nm}(\Delta_-\setminus\{\alpha\}).U_{-\alpha,\Omega} .U_{\alpha,\Omega}. T_{2N(\lambda)}\\
&= U_{\Omega}^{pm}(\Delta_+\setminus \{\alpha\}).U_{-\alpha,\Omega}.U_{\Omega}^{nm}(\Delta_-\setminus\{\alpha\}).U_{\alpha,\Omega}.  T_{2N(\lambda)}\\
&=U_{\Omega}^{pm}(r_i.\Delta_+).U_{\Omega}^{nm}(r_i.\Delta_-).T_{2N(\lambda)} \end{aligned}\] Therefore $\cV_\lambda$ does not change when $\Delta_+$ is replaced by $w.\Delta_+$, for $w\in W^v$, which proves (1).

Let $w\in W^v$ be such that $\lambda\in w.C^v_f$. By (1) we have \[\cV_\lambda=U_{\Omega}^{pm}(w.\Delta_+).U_{\Omega}^{nm}(w.\Delta_-).T_{2N(\lambda)}.\] Let $t\in T$ be such that $t.0=\lambda$. By Lemma~\ref{l_action_T_U}, we have  \[t U_{\Omega}^{pm}(w.\Delta_+) t^{-1}=U_{t.\Omega}^{pm}(w.\Delta_+)=U_{[0,2\lambda]}^{pm}(w.\Delta_+)=U_0^{pm}(w.\Delta_+).\] Similarly, $t U_{\Omega}^{nm}(w.\Delta_-) t^{-1}=U_{2\lambda}^{nm}(w.\Delta_-)$. As $T$ is commutative, we also have $t  T_{2N(\lambda)} t^{-1}=T_{2N(\lambda)}$. In order to prove (2), it suffices to prove that \[H:=t U_{\Omega}^{pm}(w.\Delta_+) U_{\Omega}^{nm}(w.\Delta_-) T_{2N(\lambda)} t^{-1}=U_{0}^{pm}(w.\Delta_+) U_{2\lambda}^{nm}(w.\Delta_-) T_{2N(\lambda)}\] is a subgroup of $G$. It suffices to prove that $H$ is stable under left multiplication by $U_{0}^{pm}(w.\Delta_+)$, $U_{2\lambda}^{nm}(w.\Delta_-)$ and $T_{2N(\lambda)}$. The first stability is clear and the third follows from Lemma~\ref{l_action_T_U}.  

First note that by \cite[5.7 1)]{rousseau2016groupes}, \begin{equation}\label{e_fixator_G_02n}
G_{[0,2\lambda]}=U_{0}^{pm}(w.\Delta_+) U_{2\lambda}^{nm}(w.\Delta_-) \fT(\cO)
\end{equation} is a subgroup of $G$.

Let $u_-,\tilde{u_-}\in U_{2\lambda}^{nm}(w.\Delta_-)$, $u_+\in U_{0}^{pm}(w.\Delta_+)$, $t\in T_{2N(\lambda)}$. Let us prove that $\tilde{u}_- u_+ u_- t\in H$. We have $\tilde{u}_- u_+=u_+ (u_+^{-1}\tilde{u}_- u_+)$. By Proposition~\ref{p_explicit_morphism}, $U_{2\lambda}^{nm}(w.\Delta_-)\subset \ker \pi_{2N(\lambda)}$.  By Proposition~\ref{p_fixator_zero}, $u_+\in \fG^{\min}(\cO)$  and we have $u_+^{-1}\tilde{u}_- u_+\in G_{[0,2\lambda]}\cap \ker \pi_{2N(\lambda)}$. Therefore \eqref{e_fixator_G_02n} implies that we can  write $u_+^{-1}\tilde{u}_- u_+=u_1^+ u_1^- t_1$, where $u_1^+\in U_0^{pm}(w.\Delta_+)$, $u_1^-\in U_{2\lambda}^{nm}(w.\Delta_-)$ and $t_1\in \fT(\cO)$. We have \[\pi_{2N(\lambda)}(u_1^+ u_1^- t_1)=1=\pi_{2N(\lambda)}(u_1^+ t_1)=\pi_{2N(\lambda)} (u_1^+)\pi_{2N(\lambda)}(t_1).\] As $\fB(\cO/\qp^{2N(\lambda)}\cO)=\fT(\cO/\qp^{2N(\lambda)}\cO)\ltimes \fU^{pma,w}(\cO/\qp^{2N(\lambda)}\cO)$ and as $\pi_{2N(\lambda)}\left(\fT(\cO)\right)\subset \fT(\cO/\qp^{2N(\lambda)} \cO)$ and $\pi_{2N(\lambda)}\left(\fU^{pma,w}(\cO)\right)\subset \fU^{pma,w}(\cO/\qp^{2N(\lambda)} \cO)$, we deduce $\pi_{2N(\lambda)}(u_1^+)=\pi_{2N(\lambda)}(t_1)=1$ and $t_1\in T_{2N(\lambda)}$.  We have \[\tilde{u}_- u_+u_- t=u_+u_1^+ u_1^-t_1 u_- t\] and as $T_{2N(\lambda)}$ normalizes $U_{2\lambda}^{nm}(w.\Delta_-)$, $\tilde{u}_- u_+u_- t\in H$, which proves the lemma. 

\end{proof}

\begin{Remark}
\begin{enumerate}
\item Note that if $\lambda\in Y^+$ is regular, then  $U_{[-\lambda,\lambda]}^{pm+} U_{[-\lambda,\lambda]}^{nm-} T_{2N(\lambda)+1}$ is not a subgroup of $G$. Indeed, take $\alpha\in \Phi^+$ such that $|\alpha(\lambda)|=N(\lambda)$. Then $x_{-\alpha}(\qp^{N(\lambda)})x_{\alpha}(\qp^{N(\lambda)})=x_\alpha(\qp^{N(\lambda)}(1+\qp^{2N(\lambda)})^{-1})x_{-\alpha}(\qp^{N(\lambda)}(1+\qp^{2N(\lambda)})^{-1})\alpha^{\vee}(1+\qp^{2N(\lambda)})$ and by Lemma~\ref{l_uniqueness_decomposition_UUT}, this does not belong to $U_{[-\lambda,\lambda]}^{pm+} U_{[-\lambda,\lambda]}^{nm-} T_{2N(\lambda)+1}$.

\item For every $k\in \llbracket 0,2N(\lambda)\rrbracket$, $U_{[-\lambda,\lambda]}^{pm+}.U_{[-\lambda,\lambda]}^{nm-}.T_k=\cV_\lambda.T_k$ is  a subgroup of $G$, since $T_k$ normalizes $U_{[-\lambda,\lambda]}^{pm+}$ and $U_{[-\lambda,\lambda]}^{nm-}$ by Lemma~\ref{l_action_T_U}. Note that for  the definition of a	 topology, we could also have taken the filtration $(U_{[-n\lambda,n\lambda]}^{pm+}.U_{[-n\lambda,n\lambda]}^{nm-}.T_{k(n)})_{n\in\N^*}$, for any $k(n)\in \llbracket 0,2nN(\lambda)\rrbracket$ such that $k(n)\underset{n\to +\infty}{\rightarrow} +\infty$.

\item As $\cV_{\lambda}=\cV_{\lambda}^{-1}$ and by Lemma~\ref{l_action_T_U}, we have $\cV_{\lambda}=U_{[-\lambda,\lambda]}^{nm-} U_{[-\lambda,\lambda]}^{pm+}T_{2N(\lambda)}$.

\end{enumerate}
\end{Remark}

\subsection{Filtration ($\cV_{n\lambda})_{n\in \N^*}$}

Let $\Omega$ be a filter on $\A$. One defines $\cl^{\Delta}(\Omega)$\index{c@$\cl^{\Delta}$} as the filter on $\A$ consisting of the subsets $\Omega'$ of $\A$ for which there exists $(k_\alpha)\in \prod_{\alpha\in \Delta} \Lambda_\alpha\cup\{+\infty\}$ such that $\Omega'\supset \bigcap_{\alpha\in \Delta} D(\alpha,k_\alpha)\supset \Omega$, where $\Lambda_\alpha=\Lambda$ if $\alpha\in \Phi$ and $ \Lambda_\alpha=\R$ otherwise. Note that $\cl^{\Delta}$ is denoted $\cl$ in \cite{rousseau2011masures} and \cite{rousseau2016groupes}. By definition of $U_{\Omega}^{pm+}$ and $U_{\Omega}^{nm-}$, we have \begin{equation}\label{e_enclosure_U}
U_{\Omega}^{pm+}=U_{\Omega'}^{pm+}=U_{\cl^{\Delta}(\Omega)}^{pm+}\text{ and }U_{\Omega}^{nm-}=U_{\Omega'}^{nm-}=U_{\cl^\Delta(\Omega)}^{nm-},
\end{equation}

for any filter $\Omega'$ such that $\Omega\subset \Omega'\subset \cl^{\Delta}(\Omega)$. 

\begin{Lemma}\label{l_enclosure_regular_segment}
Let $\lambda\in C^v_f$ and $w\in W^v$. Then $\cl^{\Delta}([-w.\lambda,w.\lambda])\supset(-w.\lambda+w. \overline{C^v_f})\cap (w.\lambda-w.\overline{C^v_f})$. 
\end{Lemma}

\begin{proof}
As $\Delta$ and $\Phi$ are $W^v$-invariant, we have $w.\cl^{\Delta}(\Omega)=\cl^{\Delta}(w.\Omega)$ for every $w\in W^v$. Thus it suffices to determine $\cl^{\Delta}([-\lambda,\lambda])$. Let $(k_\alpha)\in \prod_{\alpha\in \Delta} \Lambda_\alpha\cup \{+\infty\}$ be such that $\bigcap_{\alpha\in \Delta} D(\alpha,k_\alpha)\supset [-\lambda,\lambda]$. Let $\alpha\in \Delta_+$. Write $\alpha=\sum_{i\in I} n_i\alpha_i$, with $n_i\in \N$ for $i\in I$. Then $k_\alpha\geq \alpha(\lambda)=\sum_{i\in I} n_i \alpha_i(\lambda)$.  Let $\alpha\in \Delta_-$. We also have  $k_{-\alpha}\geq \sum_{i\in I} n_i \alpha_i(\lambda)$. 

Let $x\in (-\lambda+C^v_f)\cap (\lambda-C^v_f)$. Then $-\alpha_i(\lambda)\leq \alpha_i(x)\leq \alpha_i(\lambda)$ for every $i\in I$. Then $k_{-\alpha}\leq \sum_{i\in I} n_i \alpha_i(x)\leq k_\alpha$ and thus $k_\alpha+\alpha(x)\geq 0$ and $k_{-\alpha}-\alpha(x)\geq 0$. Consequently, $x\in \bigcap_{\alpha\in \Delta}D(\alpha,k_\alpha)$ and thus \[(-\lambda+C^v_f)\cap (\lambda-C^v_f)\subset\bigcap_{\alpha\in \Delta} D(\alpha,k_\alpha),\] which proves the lemma.
\end{proof}

The following lemma will be crucial throughout the paper. This is a rewriting of \cite[Lemma 3.2 and Lemma 3.5]{bardy2022twin}. Although $\omega$ is assumed to be discrete in  \cite{bardy2022twin}, the proofs of these lemma do not use this assumption. 

\begin{Lemma}\label{l_BPHR}
\begin{enumerate}
\item Let $a\in \A$ and $g\in U^+$. Then there exists $b\in a-C^v_f$ such that $g^{-1} U_{b}^{pm+} g\subset U_a^{pm+}$.

\item Let $y\in \I$. Then there exists $a\in \A$ such that $U_a^{pm+}$ fixes $y$. 

\item Let $\lambda\in Y^+$ be regular and $y\in \I$. Then for $n\in \N$ large enough, $U_{[-n\lambda,n\lambda]}^{pm+}$ fixes $y$. 
\end{enumerate}

\end{Lemma}

\begin{proof}
By \cite[Lemma 3.2 and 3.5]{bardy2022twin}, we have 1) and 2).  Let $\lambda\in Y^+$ be regular. Write $\lambda=w.\lambda^{++}$, with $\lambda^{++}\in C^v_f$ and $w\in W^v$.  Let $a\in \A$ be such that $U_{a}^{pm+}$ fixes $y$.  For $n\in \N^*$, we have $\cl^{\Delta}([-n\lambda,n\lambda])\supset n\left((-w.\lambda^{++}+w.\overline{C^v_f})\cap (w.\lambda^{++}-w.\overline{C^v_f})\right)$, which contains $a$, for $n\gg 0$. So for $n\gg 0$, we have $U_{[-n\lambda,n\lambda]}^{pm+}\subset U_a^{pm+}$, which proves 3). 
\end{proof}

\begin{Lemma}\label{l_equivalence_filtrations}
Let $\lambda,\mu\in Y^+$ be regular. Then $(\cV_{n\lambda})_{n\in \N^*}$ and $(\cV_{n\mu})_{n\in \N^*}$ are equivalent. 
\end{Lemma}

\begin{proof}
Write   $\mu=v.\mu^{++}$, where $v\in W^v$ and $\mu^{++}\in C^v_f$. For $m\in \N$,  set \[\Omega_{m\mu}=(-m\mu+v.\overline{C^v_f})\cap (m\mu-v.\overline{C^v_f}).\] By Lemma~\ref{l_enclosure_regular_segment}, $\Omega_{m\mu}\subset \cl^\Delta([-m\mu,m\mu])$. Let $n\in \N^*$. As $\Omega_{m\mu}=m\Omega_\mu$ and $\Omega_\mu$ contains $0$ in its interior, there exists $m\in \Z_{\geq n}$ such that $\Omega_{m\mu}\supset [-n\lambda,n\lambda]$.  Moreover, by \eqref{e_enclosure_U}, $U_{[-m\mu,m\mu]}^{pm+}=U_{\Omega_{m\mu}}^{pm+}\subset U_{[-n\lambda,n\lambda]}^{pm+}$, since $\Omega'\mapsto U_{\Omega'}^{pm+}$ is decreasing for $\subset$. With the same reasoning for $U^{nm-}$, we deduce $\cV_{m\mu}\subset \cV_{n\lambda}$.  By symmetry of the roles of $\lambda$ and $\mu$, we deduce the lemma. 
\end{proof}

The end of this subsection is devoted to the proof of the fact that for every $\lambda\in Y^+$ regular, $(\cV_{n\lambda})_{n\in\N^*}$ is conjugation-invariant. 

\begin{Lemma}\label{l_conjugation_U}
Let $\alpha\in \Phi$ and $a\in \cK$. Let $\lambda\in Y^+$ be regular.  Then $(x_{\alpha}(a). \cV_{n\lambda} .x_\alpha(-a))_{n\in \N^*}$ is equivalent to $(\cV_{n\lambda})_{n\in \N^*}$. 
\end{Lemma}

\begin{proof}
Let $\epsilon\in \{-,+\}$ be such that $\alpha\in \Phi_\epsilon$. Let  $w\in W^v$ be such that $\epsilon w^{-1}.\alpha$  is simple. By symmetry, we may assume that $ \epsilon=+$. By Lemma~\ref{l_equivalence_filtrations} we may assume that $\lambda\in w.C^v_f$. Let $n\in \N^*$.  By Lemma~\ref{l_decomposition_Vlambda}, we have  \[\begin{aligned} \cV_{n\lambda}&=U_{[-n\lambda,n\lambda]}^{pm}(w.\Delta_+) U_{[-n\lambda,n\lambda]}^{nm}(w.\Delta_-) T_{2nN(\lambda)}.\\
&= U_{-n\lambda}^{pm}(w.\Delta_+) U_{n \lambda}^{nm}(w.\Delta_-) T_{2nN(\lambda)}.\end{aligned} \] Write $\alpha=w.\alpha_i$, for $i\in I$.   By Lemma~\ref{l_BPHR}, \begin{equation}\label{e_inclusion_U_4}
x_{w.\alpha_i}(a) U_{-m\lambda}^{pm}(w.\Delta_+) x_{w.\alpha_i}(-a)\subset U_{-n\lambda}^{pm}(w.\Delta_+),
\end{equation} for  $m\gg 0$. 

By \cite[Lemma 3.3]{rousseau2016groupes}, $U_{[-n\lambda,n\lambda]}^{nm}(w.\Delta_-)=U_{[-n\lambda,n\lambda]}^{nm}(w.(\Delta_-\setminus\{-\alpha_i\})). U_{-w.\alpha_i,[-n\lambda,n\lambda]}$ and $U_{-w.\alpha_i,[-n\lambda,n\lambda]}$ normalizes $U_{[-n\lambda,n\lambda]}^{nm}(w.(\Delta_-\setminus\{-\alpha_i\}))$. 

By \cite[1.3.11 Theorem (b4)]{kumar2002kac}, $r_i.(\Delta_-\setminus\{-\alpha_i\})=\Delta_-\setminus\{-\alpha_i\}$  and thus for $m\in \N$, we have \[\begin{aligned} U^{nm}_{[-m\lambda,m\lambda]}\left(w.(\Delta_-\setminus\{-\alpha_i\})\right)&=U_{[-m\lambda,m\lambda]}^{nm}\left(wr_i (\Delta_-\setminus\{-\alpha_i\})\right)\\
&=U^{nm}(wr_i(\Delta_-\setminus\{-\alpha_i\}))\cap U_{[-m\lambda,m\lambda]}^{nm}(wr_i.\Delta_-).\end{aligned}\] Moreover, \begin{equation}\label{e_inclusion_U_1}
x_{w.\alpha_i}(a)U^{nm}(wr_i(\Delta_-\setminus\{-\alpha_i\})) x_{w.\alpha_i}(-a)=U^{nm}(wr_i(\Delta_-\setminus\{-\alpha_i\})),
 \end{equation} by \cite[Lemma 3.3]{rousseau2016groupes} (applied to $wr_i.(\Delta_-\setminus \{-\alpha_i\})\subset wr_i.\Delta_-$). By Lemma~\ref{l_BPHR},  for $m\gg 0$, \begin{equation}\label{e_inclusion_U_2}
 x_{w.\alpha_i}(a) U_{[-m\lambda,m\lambda]}^{nm}(wr_i.\Delta_-)x_{w.\alpha_i}(-a)\subset U_{[-n\lambda,n\lambda]}^{nm}(wr_i.\Delta_-).
 \end{equation}  Combining  \eqref{e_inclusion_U_1} and \eqref{e_inclusion_U_2}, we get \begin{equation}\label{e_inclusion_U_3}
  x_{w.\alpha_i}(a) U_{[-m\lambda,m\lambda]}^{nm}(w.\Delta_-\setminus\{-\alpha_i\})x_{w.\alpha_i}(-a)\subset U_{[-n\lambda,n\lambda]}^{nm}\left(w.(\Delta_-\setminus\{-\alpha_i\})\right),
  \end{equation} for $m\gg 0$. 
  
  For $b\in \cK$ such that $1+ab\neq 0$, we have \[x_{w.\alpha_i}(a) x_{-w.\alpha_i}(b)x_{w.\alpha_i}(-a)=x_{-w.\alpha_i}\left(b(1+ab)^{-1}\right)\alpha^\vee(1+ab) x_{w.\alpha_i}(-a^2b(1+ab)^{-1}).\] Therefore if $m\gg 0$, $x_{w.\alpha_i}(a) U_{-w.\alpha_i,[-m\lambda,m\lambda]}x_{w.\alpha_i}(-a)\subset \cV_{n\lambda}$.  Combined with  \eqref{e_inclusion_U_3} we get  \begin{equation}\label{e_inclusion_U_5}
x_{w.\alpha_i}(a) U_{m\lambda}^{nm}(w.\Delta_-) x_{w.\alpha_i}(-a)\subset \cV_{n\lambda},
\end{equation} for  $m\gg 0$, since $\cV_{n\lambda}$ is a group.

Let $m\in \N^*$ and  $t\in T_{2m}$. Then: \[x_{w.\alpha_i}(a)t x_{w.\alpha_i}(-a)=t x_{w.\alpha_i}\left(a\left(w.\alpha_i(t^{-1})-1\right)\right)\] Therefore if $m\geq nN(\lambda)$ and $\omega\left(a\left( w.\alpha_i(t^{-1})-1\right)\right)\geq \omega(a)+2mN(\lambda)$ is greater than $|w.\alpha_i(n\lambda)|$, then $x_{w.\alpha_i}\left(a\left(w.\alpha_i(t^{-1})-1\right)\right) \in \cV_{n\lambda}$ and $x_{w.\alpha_i}(a)t x_{w.\alpha_i}(-a)\in \cV_{n\lambda}$. Combined with \eqref{e_inclusion_U_5} and \eqref{e_inclusion_U_4}, we get $x_{\alpha}(a) .\cV_{m\lambda}. x_{\alpha}(-a)\subset \cV_{n\lambda}$, for $m\gg 0$. Applying this to $(x_\alpha(-a).\cV_{k\lambda}.x_{\alpha}(a))_{k\in \N^*}$, we get the other inclusion needed to prove that $(\cV_{n\lambda})$ and $(x_{\alpha}(a).\cV_{n\lambda}.x_{\alpha}(-a))$ are equivalent.
\end{proof}

\begin{Theorem}\label{t_conjugation_invariance}
Let $\lambda\in Y^+$ be regular. For $n\in \N^*$, set $\cV_{n\lambda}=U_{[-n\lambda,n\lambda]}^{pm+}.U_{[-n\lambda,n\lambda]}^{nm-}.T_{2N(n\lambda)}$.  Then $(\cV_{n\lambda})$ is conjugation-invariant. Therefore, the associated topology $\sT\left((\cV_{n\lambda} )\right)$ equips $G$ with the structure of a Hausdorff topological group. 
\end{Theorem}

\begin{proof}
We need to prove that for every $g\in G$, $g (\cV_{n\lambda}) g^{-1}$ is equivalent to $(\cV_{n\lambda})$. Using Lemma~\ref{l_equivalence_filtrations}, we may assume $ \lambda\in C^v_f$. By \cite[Proposition 1.5]{rousseau2006immeubles}, $G$ is generated by $T$ and the $x_\alpha(a)$, for $\alpha\in \Phi$ and $a\in \cK$. By Lemma~\ref{l_conjugation_U}, it remains only to prove that if $t\in T$, then $ (t\cV_{n\lambda}t^{-1})_{n\in \N^*}$ is equivalent to $(\cV_{n\lambda})$. Let $t\in T$ and $m\in \N^*$. Then by Lemma~\ref{l_action_T_U}, \[\begin{aligned} t\cV_{m\lambda}t^{-1}&= tU_{[-m\lambda,m\lambda]}^{pm+} t^{-1}. tU_{[-m\lambda,m\lambda]}^{nm-} t^{-1}. T_{2mN(\lambda)}\\ &=U_{[t.(-m\lambda),t.m\lambda]}^{pm+} U_{[t.(-m\lambda),t.m\lambda]}^{nm-} T_{2mN(\lambda)} \end{aligned}\]

Set $\Omega=(-\lambda+\overline{C^v_f})\cap (\lambda-\overline{C^v_f})$. Then by Lemma~\ref{l_enclosure_regular_segment}, $\cl^\Delta([-m\lambda,m\lambda])\supset m\Omega$. Moreover, $\cl^\Delta([t.(-m\lambda),t.m\lambda])=t.\cl^\Delta([-m\lambda,m\lambda])\supset t.m\Omega$. Let $n\in \N^*$. Then as $\Omega$ contains $0$ in its interior, $t.m\Omega\supset n\Omega$ for $m\gg 0$. Therefore (by \eqref{e_enclosure_U}) $U_{[t.(-m\lambda),t.m\lambda]}^{pm+}\subset  U_{[-n\lambda,n\lambda]}^{pm+} $ and $U_{[t.(-m\lambda),t.m\lambda]}^{nm-}\subset U_{[-n\lambda,n\lambda]}^{nm-}$ for $m\gg 0$. Consequently, $t \cV_{m\lambda} t^{-1} \subset \cV_{n\lambda}$ for $m\gg 0$, which proves that $(\cV_{n\lambda})$ is conjugation-invariant. 

It remains to prove that $\sT\left((\cV_{n\lambda})\right)$ is Hausdorff. For that it suffices to prove that  $\bigcap_{n\in \N^*} \cV_{n\lambda}=\{1\}$. Let $g\in \bigcap_{n\in \N^*} \cV_{n\lambda}$. Let $n\in \N^*$. Then as $[-n\lambda,n\lambda]$ has a good fixator (\cite[5.7 1)]{rousseau2016groupes}  and \eqref{e_good_fixator}) $g\in G_{[-n\lambda,n\lambda]}=U_{[-n\lambda,n\lambda]}^{pm+}. U_{[-n\lambda,n\lambda]}^{nm-}.\fT(\cO)$, so we can write $g=u_n^+ u_n^- t_n$, with $(u_n^+,u_n^-,t_n)\in U_{[-n\lambda,n\lambda]}^{pm+}\times  U_{[-n\lambda,n\lambda]}^{nm-}\times \fT(\cO)$. By Lemma~\ref{l_uniqueness_decomposition_UUT}, $u_+:=u_n^+$ does not depend on $n$ and thus $u^+\in \bigcap_{n\in \N^*} U_{[-n\lambda,n\lambda]}^{pm+}=U_\A^{pm+}=\{1\}$. Similarly, $u^-:=u_n^-=1$. Therefore $t\in T\cap \bigcap_{n\in \N^*}\ker \pi_n$. Let $(\chi_i)_{i\in \llbracket 1,m\rrbracket}$ be a $\Z$-basis of $Y$. Write $t=\prod_{i=1}^m \chi_i(a_i)$, with $a_i\in \cO^*$, for $i\in \llbracket 1,m\rrbracket$. Let $n\in \N$. Then $\pi_n(t)=\prod_{i=1}^m \chi_i(\pi_n(a_i))=1$ and thus $a_i\in \bigcap_{n\in \N^*}\qp^n \cO=\{0\}$. Consequently, $t=1$ and $g=1$. Therefore $\bigcap_{n\in \N^*} \cV_{n\lambda}=\{1\}$ and $\sT((\cV_{n\lambda}))$ is Hausdorff. 
\end{proof}

We denote by $\sT$\index{T@$\sT$} the topology $\sT\left((\cV_{n\lambda})\right)$, for any $\lambda\in Y^+$ regular.

\begin{Corollary}\label{c_comparison_kernels_filtration}
Let $\lambda\in Y^+$ be regular. The  filtrations $(\ker \pi_n)_n$ and $(\cV_{n\lambda})_{n\in \N^*}$ are equivalent if and only if $W^v$ is finite.
\end{Corollary}

\begin{proof}
If $W^v$ is infinite, this follows from Lemma~\ref{l_non_conjugacy_invariance_congruence} and Theorem~\ref{t_conjugation_invariance}. Assume now that $W^v$ is finite. Then by \eqref{e_inclusion_VnLmabda_ker_Pi_nLambda}, $\cV_{n\lambda}\subset \ker \pi_{nN(\lambda)}$ for every $n\in \N^*$. 

Let $n\in \N^*$ and $g\in \ker \pi_n$. By Corollary~\ref{c_ker_pin}, $g\in U_{-C_0^+}^{pm+} .\fT(\cO).U_{C_0^+}^{nm-}$. As $W^v$ is finite, $\Phi=\Delta$ (by \cite[Theorem 5.6]{kac1994infinite}), $U_{-C_0^+}^{pm+}=U_{-C_0^+}^{+}=\prod_{\alpha\in \Phi_+}x_\alpha(\fm)$ and $U_{C_0^+}^{nm-}=U_{C_0^+}^-=\prod_{\alpha\in \Phi_-}x_\alpha(\fm)$, for any (fixed) orders on $\Phi_-$ and  $\Phi_+$, where $\fm$ is the maximal ideal of $\cO$.  Write $g=u_+tu_-$, with $g\in U_{C_0^+}^-$, $t\in \fT(\cO)$ and $u_-\in U_{C_0^+}^-$. Then $\pi_n(g)=\pi_n(u_+t)\pi_n(u_-)=1$. Moreover, $\pi_n(u_+t)\in \fB(\cO/\qp^n\cO)$ and $\pi_n(u_-)\in \fB^-(\cO/\qp^n\cO)$, where $\fB=\cU^+\rtimes \fT$ and $\fB^-=\fU^-\rtimes \fT$ are two opposite Borel subgroups. Therefore $\pi_n(u_+t)\in \fB(\cO/\qp^n\cO)\cap \fB^-(\cO/\qp^n\cO)=\fT(\cO/\qp^n\cO)$, by \cite[Theorem 14.1]{borel1991linear}. Therefore $\pi_n(u_+)=\pi_n(u_-)=\pi_n(t)=1$.  Write $u_+=\prod_{\alpha\in \Phi_+}x_\alpha(u_\alpha)$, with $u_\alpha\in \cO$ for $\alpha\in \Phi_+$. Then $\pi_n(u_+)=\prod_{\alpha\in \Phi_+}x_\alpha(\pi_n(u_\alpha))=1$, by Proposition~\ref{p_explicit_morphism}, and by \cite[Theorem 8.51]{marquis2018introduction}, this implies $\pi_n(u_\alpha)=1$, for every $\alpha\in \Phi_+$. Let $N'(\lambda)=\max_{\alpha\in \Phi} |\alpha(\lambda)|$. Then we have $u_+\in U^+\cap G_{[-n\lambda/N'(\lambda),n\lambda/N'(\lambda)]}$. Using a similar reasoning for $U^-$ and \eqref{e_cV_as_fixators}, we deduce $\ker \pi_n\subset \cV_{n\lambda/N'(\lambda)}$. Therefore $(\cV_{n\lambda})$ and $(\ker \pi_n)$ are equivalent. 
\end{proof}

\begin{Remark}
\begin{enumerate}
\item If $n\in \N^*$, then   $\ker \pi_n$ is open for $\sT$, by \eqref{e_inclusion_VnLmabda_ker_Pi_nLambda}. 

\item The Iwahori subgroup $G_{C_0^+}=K_I$ is open. Indeed, if $\lambda\in Y\cap C^v_f$, then $\cV_\lambda\subset K_I$ by \eqref{e_cV_as_fixators}. In particular, $\fG^{\min}(\cO)\supset K_I$ is open. 

\item Assume $\fG$ is reductive, i.e $W^v$ is finite. Then by Corollary~\ref{c_comparison_kernels_filtration}, $\sT$ is the usual topology on $G$. 
\end{enumerate} 
\end{Remark}

\subsection{Topology of the fixators}
\subsubsection{Definition of the topology}

Recall that if $F$ is a subset of $\I$, we denote by $G_F$ its fixator in $G$.  In this subsection we study the topology $\sT_\Fix$\index{t@$\sT_{\Fix}$} which is defined as follows. A subset $V$ of $G$ is  open if for every $v$ in $V$, there exists a finite subset $F$ of $\I$ such that $vG_F\subset V$.  Note that $G_0=\fG^{\min}(\cO)$ is open for this topology.

We begin by constructing  increasing sequences of finite sets of vertices $(F_n)=(F_n(\lambda))_{n\in \N^*}$ such that $\sT_{\Fix}$ is the topology associated with $(G_{F_n})_{n\in \N}$. 

  We fix $\lambda\in Y^+$ regular. We set $F_0=\emptyset$. For $n\in \N^*$, we set \begin{equation}\label{e_def_F_n}
F_n=F_n(\lambda)=\{n\lambda,-n\lambda\}\cup \{x_{\alpha_i}(\qp^{-n}).0\mid i\in I\}\cup F_{n-1}.
\end{equation}\index{F@$F_n$}

Let $n\in \N^*$. By \cite[5.7 1)]{rousseau2016groupes}, $[-n\lambda,n\lambda]$ has a good fixator and by Proposition~\ref{p_BT_7.4.8}, we have $G_{\{n\lambda,-n\lambda\}}=G_{[-n\lambda,n\lambda]}$. Therefore \begin{equation}\label{e_decomposition_GFn}
G_{F_n}\subset G_{[-n\lambda,n\lambda]}=U_{[-n\lambda,n\lambda]}^{pm+}.U_{[-n\lambda,n\lambda]}^{nm-}.\fT(\cO).
\end{equation}

We chose $F_n$ as above for the following reasons. We want that when $x\in \I$ and $n\gg 0$, an element of $G_{F_n}$ fixes $x$. By Lemma~\ref{l_BPHR}, if $u\in U^+$ (resp. $U^-$), and if $u$ fixes $-n\lambda$ (resp. $n\lambda$), for some $n$ large, then $u$ fixes $x$. However, if $t\in T$, as   $\fT(\cO)$ fixes $\A$, we need to require that $t$ fixes elements outside of $\A$, and the choice of $x_{\alpha_i}(\qp^n).0$ is justified by the lemma below.

\begin{Lemma}\label{l_condition_t_fixes_Fn}
Let $i\in I$, $n\in \N$ and $t\in T$. Then $t$ fixes $x_{\alpha_i}(\qp^{-n}).0$ if and only if $\omega\left(\alpha_i(t)-1\right)\geq n$. 
\end{Lemma}

\begin{proof}
 We have $t.x_{\alpha_i}(\qp^{-n}).0=x_{\alpha_i}(\qp^{-n}).0$ if and only if $x_{\alpha_i}(-\qp^{-n})x_{\alpha_i}(\alpha_i(t)\qp^{-n}).t.0=0$. We have $\rho_{+\infty}(x_{\alpha_i}(-\qp^{-n})x_{\alpha_i}(\alpha_i(t)\qp^{-n}).t.0)=t.0$ and thus if $t.x_{\alpha_i}(\qp^{-n}).0=x_{\alpha_i}(\qp^{-n}).0$, we have $t.0=0$. Thus $t.x_{\alpha_i}(\qp^{-n}).0=x_{\alpha_i}(\qp^{-n}).0$ if and only if $x_{\alpha_i}((\alpha_i(t)-1)\qp^{-n}).0=0$ if and only if $\omega(\alpha_i(t)-1)\geq n.$
\end{proof}

For $n\in \N$, we set \begin{equation}\label{e_T_n_Phi}
T_{n,\Phi}=\{t\in T\mid \omega(\alpha_i(t)-1)\geq n, \forall i \in I\}
\end{equation}

\begin{Lemma}\label{l_T_fixes_I_for_large_valuation}
Let $y\in \I$. Then there exists $M\in \N$ such that $T_{M,\Phi}$ fixes $y$.
\end{Lemma}

\begin{proof}
By the Iwasawa decomposition ((MA III)), $y\in U^+.z$, where $z=\rho_{+\infty}(y)$.  Write $y=x_{\beta_1}(a_1)\ldots x_{\beta_k}(a_k).z$, with  $k\in \N$ and  $\beta_1,\ldots,\beta_k\in \Phi_+$. Let $t\in T$. We have $t.y=y$ if and only if \[\begin{aligned} z &= x_{\beta_k}(-a_k)\ldots x_{\beta_1}(-a_1)tx_{\beta_1}(a_1)\ldots x_{\beta_k}(a_k).z\\
&=x_{\beta_k}(-a_k)\ldots x_{\beta_2}(-a_2)t x_{\beta_1}\left((1-\beta_1(t^{-1}))a_1\right)x_{\beta_2}(a_2)\ldots x_{\beta_k}(a_k).z\end{aligned}\]

Let $M\in \N^*$. Assume that $\alpha_i(t)-1\in \qp^M \cO$ for every $i\in I$. 

Write $\beta_1=\sum_{i\in I} m_i\alpha_i$, with $m_i\in \N$ for every $i\in I$. Then $\beta_1(t)=\prod_{i\in I}\alpha_i^{m_i}(t)$. Therefore $\beta_1(t)\in 1+\qp^M\cO$. For $M\gg 0$, $x_{\beta_1}\left((1-\beta_1(t^{-1}))a_1\right)$ fixes $x_{\beta_2}(a_2)\ldots x_{\beta_k}(a_k).z$, by Lemma~\ref{l_BPHR}. By induction on $k$ we deduce that $t$ fixes $z$ for $M\gg 0$. 
\end{proof}

\begin{Lemma}\label{l_GFn_eventually_fix_I}
\begin{enumerate}
\item Let $F$ be a finite subset of $\I$. Then there exists $n\in \N^*$ such that $G_{F_n}$ fixes $F$.

\item Let $\lambda,\mu\in Y^+$ be regular. Then the filtrations $(G_{F_n(\lambda)})_{n\in \N^*}$ and $(G_{F_n(\mu)})_{n\in \N^*}$ are equivalent. 

\item The topology $\sT_{\Fix}$ is the topology associated with $\sT((G_{F_n})_{n\in \N})$. 
\end{enumerate}

\end{Lemma}

\begin{proof}
1) It suffices to prove that if $y\in \I$, then $G_{F_n}$ fixes $y$ for $n\gg 0$. 
Let $n\in \N^*$. Then $F_n\supset \{-n\lambda,n\lambda\}$ and by Proposition~\ref{p_BT_7.4.8}, we have $G_{F_n}\subset G_{[-n\lambda,n\lambda]}$.

By Lemma~\ref{l_BPHR}, there exists $n_1\in \N$ such that $U_{[-n\lambda,n\lambda]}^{pm+}$ and $U_{[-n\lambda,n\lambda]}^{nm-}$ fix $y$, for $n\geq n_1$. Let $n\geq n_1$. Let  $n_2=n_2(n)\geq n$ be such that for every $i\in I$, $\langle U_{[-n_2.\lambda,n_2.\lambda]}^{pm+},U_{[-n_2.\lambda,n_2.\lambda]}^{nm-}\rangle$ fixes $x_{\alpha_i}(\qp^{-n}).0$ for every $i\in I$.  Let $g\in G_{F_{n_2}}$. Using \eqref{e_decomposition_GFn}, we write $g=u_+u_-t$, with $(u_+,u_-,t)\in U_{[-n_2\lambda,n_2\lambda]}^{pm+}\times U_{[-n_2\lambda,n_2\lambda]}^{nm-}\times \fT(\cO)$. Then $g$ fixes $y$ if and only if $t$ fixes $y$. Moreover, $u_+ u_- t$ fixes $F_n$ and thus $t$ fixes $F_n$. By Lemma~\ref{l_condition_t_fixes_Fn}, we deduce $\omega((\alpha_i(t)-1)\geq n$ for every $i\in I$. By Lemma~\ref{l_T_fixes_I_for_large_valuation} we deduce that $t$ fixes $y$, for $n\gg 0$. Thus $G_{F_{n_2(n)}}$ fixes $y$ for $n\gg 0$. 

2) Follows from 1) by applying to $F=F_m(\mu)$ and $F_n=F_n(\lambda)$, for $m,n\in \N^*$  and by symmetry of the roles of $\lambda$ and $\mu$. 

3) As $F_n$ is finite for every $n\in \N^*$, $\sT((G_{F_n}))$ is coarser than $\sT_{\Fix}$. But by 1), $\sT_{\Fix}$ is coarser that $(\sT_{G_{F_n}})$. 
\end{proof}

\begin{Proposition}\label{p_characterization_Tfix}
The topology $\sT_\Fix$ is the coarsest topology of topological group  on $G$  such that $\fG^{\min}(\cO)$ is open.
\end{Proposition}

\begin{proof}
For $n\in \N^*$, $G_{F_n}\subset \fG^{\min}(\cO)$ and thus $\fG^{\min}(\cO)$ is open for $\sT_\Fix$. 

Let now $\sT'$ be a topology of topological group on $G$ such that $\fG^{\min}(\cO)$ is open. Let $n\in \N^*$. Then for every element $a$ of $F_n$, there exists $g_a\in G$ such that $g_a.0=a$. Then $G_{F_n}=\bigcap_{a\in F_n} g_a.\fG^{\min}(\cO).g_a^{-1}$ is open in $G$. Proposition follows. 
\end{proof}

\subsubsection{Relation between $\sT_\Fix$ and $\sT$}

In this subsection, we compare $\sT_\Fix$ and $\sT$. We prove that $\sT$ is finer than $\sT_\Fix$. When $\cK$ is Henselian , we prove that  $\sT=\sT_{\Fix}$ if and only if the fixator of $\I$ in $G$ is $\{1\}$ (see Proposition~\ref{p_comparison_T_TFix}). 

Let $\lambda\in Y^+$ be regular. For $n\in \N$, we define $F_n=F_n(\lambda)$ as in \eqref{e_def_F_n}.

Let $\cZ=\bigcap_{n\in \N} T_{n,\Phi}$\index{Z@$\cZ$} (where the $T_{n,\Phi}$ are defined in \eqref{e_T_n_Phi}) and $\cZ_\cO=\cZ\cap \fT(\cO)$\index{Z@$\cZ_\cO$}. Then $\cZ=\{t\in T\mid \alpha_i(t)=1,\forall i\in I\}$ is the center of $G$ by \cite[8.4.3 Lemme]{remy2002groupes}.

\begin{Lemma}\label{l_fixator_I}
The fixator $G_\I$ of $\I$ in $G$ is $\cZ_\cO$ and $\bigcap_{n\in \N^*} G_{F_n}=\cZ_\cO$.
\end{Lemma}

\begin{proof}
We have $G_\I\subset G_\A=\fT(\cO)$, by \cite[5.7 5)]{rousseau2016groupes}. By Lemma~\ref{l_condition_t_fixes_Fn}, $G_\I\subset T_{n,\Phi}\cap \fT(\cO)$ for every $n\in \N$ and thus $G_\I\subset \cZ_\cO$. Let $z\in \cZ_\cO$ and $x\in \I$. Write $x=g.a$, with $a\in \A$. Then $z.x=gz.a=g.a=x$ and $z\in G_\I$. 

Now let $g\in \bigcap_{n\in \N^*} G_{F_n}$. Then by Lemma~\ref{l_GFn_eventually_fix_I}, $g$ fixes $\I$, which proves the lemma. 
\end{proof}

\begin{Lemma}\label{l_comparison_T_Tfix}
 There exists an increasing map $M:\N\rightarrow \N$ whose limit is  $+\infty$ and such that for every $m\in \N$, \[G_{F_n}\subset \cV_{m\lambda}. (T_{M(n),\Phi}\cap \fT(\cO)),\] for $m,n\in \N^*$ such that $M(n)\geq m$ and $n\geq m$. 
\end{Lemma}

\begin{proof}
 Let $n\in \N^*$. Let $M(n)\in \N$ be maximum such that $U_{[-n\lambda,n\lambda]}^{nm-}$, $U_{[-n\lambda,n\lambda]}^{pm+}$, fix $F_{M(n)}$. By Lemma~\ref{l_BPHR}, $M(n)\rightarrow +\infty$.      Let $g\in G_{F_n}$.  Using \eqref{e_decomposition_GFn} we write $g=u_+u_-t$, with $u_+\in U_{[-n\lambda,n\lambda]}^{pm+}$, $u_-\in U_{[-n\lambda,n\lambda]}^{nm-}$ and $t\in \fT(\cO)$. Let $m'=M(n)$. Then $m'\leq n$ and $u_+,u_-$ fix $F_{m'}$. As $g$ fixes $F_{m'}$ we deduce $t$ fixes $F_{m'}$. By Lemma~\ref{l_T_fixes_I_for_large_valuation}, $t\in T_{m',\Phi}\cap \fT(\cO)$. Therefore  $g\in  \cV_{m\lambda}. (T_{M(n),\Phi}\cap \fT(\cO)),$ which proves the lemma. 
\end{proof}

\begin{Lemma}\label{l_choice_basis}
Let $\A^*=X\otimes \R$, $Q'=(\bigoplus_{i\in I} \Q \alpha_i)\cap X\subset \A^*$ and $d$ be the dimension of $Q'$ as a $\Q$-vector space. Then there exists a $\Z$-basis $(\chi_1,\ldots,\chi_\ell)$ of $X$ such that $(\chi_1,\ldots,\chi_d)$ is a $\Z$-basis of $Q'$.
\end{Lemma}

\begin{proof}
Let $x\in X$ and $n\in \Z\setminus \{0\}$. Assume that $nx\in Q'$. Then $x\in Q'$. Therefore $X/Q'$ is torsion-free. Let $(e_{d+1},\ldots,e_{\ell})\in (X/Q')^{\ell-d}$ be a $\Z$-basis of $X/Q'$. For $j\in \llbracket d+1,\ell\rrbracket$, take $\chi_j\in X$ whose reduction modulo $Q'$ is $e_j$. Choose a $\Z$-basis $(\chi_1,\ldots,\chi_d)$ of $X'$.  Then $(\chi_1,\ldots,\chi_\ell)$ satisfies the condition of the lemma. 
\end{proof}

\begin{Lemma}\label{l_decomposition_root_unity}
Assume $\cK$ to be Henselian. Let $a\in \cO$ and $m\in \N^*$. Assume  $\omega(a^m-1)>0$. Then we can write $a=b+c$, with $b\in \cO$ such that $b^m=1$ and $\omega(c)>0$. 
\end{Lemma}

\begin{proof}
Let $\kk=\cO/\fm$ be the residual field and $\pi_\kk:\cO\twoheadrightarrow \cO/\fm$ be the natural projection. Let $p$ be the characteristic of $\kk$. If $p=0$, we set $m'=m$ and $k=0$. If $p>0$, we write $m=p^k m'$, with $k\in \N$ and $m'\in \N$ prime to $p$. We have $ \pi_\kk(a^m)=\pi_\kk(a)^m=\pi_\kk(1)$. We have  $(\pi_\kk(a)^{m'}-1)^{p^k}=0$ and thus $\pi_\kk(a)^{m'}=\pi_\kk(1)$. Let $Z$ be an indeterminate. We have $\overline{Z^{m'}-1}=\overline{(Z-a)} Q_\kk$, where the bar denotes the reduction modulo $\fm[Z]$ and $Q_\kk\in \kk[Z]$ is prime to $\overline{Z-a}$. As $\cO$ is Henselian, we can write $Z^{m'}-1=(Z-b)Q$, where $b\in \cO$ is such that $\pi_\kk(b)=\pi_\kk(a)$ and $Q\in \cO[Z]$ is  such that $\overline{Q}=Q_\kk$. Then $\pi_\kk(b-a)=0$ and we get the lemma, with $c=b-a$. 
\end{proof}

The following lemma was suggested to me by Guy Rousseau.

\begin{Lemma}\label{l_decomposition_large_valuation}
Assume $\cK$ to be Henselian. Let $m\in \N^*$. Then there exists $K\in \N^*,K'\in \N$ such that  for every $n\in \N^*$ and  $a\in\cO$ such that $\omega(a^m-1)\geq n$, we can write $a=b+c$, with $b,c\in \cO$ such that $b^m=1$ and $\omega(c)\geq n/K-K'$. 
\end{Lemma}

\begin{proof}
We first assume that $\cK$ has characteristic $p>0$. Let $n\in \N^*$ and $a\in \cO$ be such that $\omega(a^m-1)\geq n$. Write $m=m'p^k$, with $k\in\N$ and $m'\in \N$ prime to $p$. We have $a^m-1=(a^{m'}-1)^{p^k}$ and thus $\omega(a^{m'}-1)=\omega(a^m-1)/p^k\geq n/p^k>0$. By Lemma~\ref{l_decomposition_root_unity}, we can write $a=b+c$, with $b,c\in \cO$, $b^{m'}=1$ and $\omega(c)>0$. We have $a^{m'}-1=\overline{m'}b^{m'-1}c+ \sum_{i=2}^{m'} \overline{\binom{m'}{i}} c^i b^{m-i}$, where $\overline{x}$ is the image of $x$ in $\cK$, if $x\in \Z$. As $m'$ is prime to $p$, $\overline{m'}$ is a root of $1$ and thus we have $\omega(\overline{m'})=0$. As $b^{m'}=1$, $\omega(b)=0$. Therefore $\omega(c)=\omega(\overline{m'} b^{m-1} c)< \omega(\overline{\binom{m'}{i}} b^{m'-i} c^i)$ for $i\in \llbracket 2,m'\rrbracket$. Consequently $\omega(a^{m'}-1)=\omega(c)$ and $\omega(a^m-1)=p^{k}\omega(a^{m'}-1)=p^k\omega(c)\geq n$. This proves the lemma in this case, with $K'=0$ and $K=p^k$.

We now assume that $\cK$ has characteristic $0$. Then by \cite[Theorem 1]{greenberg1966rational} and \cite[Annexe A4]{rousseau1977immeubles} (for the case where $\omega(\cK^*)$ is not discrete) applied with $F=\{Z^m-1\}$ (where $Z$ is an indeterminate), there exist $K\in \N^*$, $K'\in \N$ such that for every $n\in \N^*$, for every $a\in \cO$ such that  $\omega(a^m-1)\geq n$, we can write $a=b+c$, with $b,c\in \cO$, $b^m=1$ and $\omega(c)\geq n/K-K'$, which proves the lemma in this case.
\end{proof}

\begin{Lemma}\label{l_comparison_TnPhi_Tn}
Assume $\cK$ to be Henselian. There exist $K_1\in \R^*_+$,  $L\in \N$ such that for every $n\in \Z_{\geq L}$, $T_{n,\Phi}\cap \fT(\cO)\subset \cZ_\cO.T_{n/K_1}$.
\end{Lemma}

\begin{proof}
We keep the same notation as in Lemma~\ref{l_choice_basis}.  Let   $(\chi^\vee_1,\ldots,\chi_\ell^\vee) \in  Y^\ell$ be the dual basis of $(\chi_1,\ldots,\chi_\ell)$. For $i\in I$, we write $\alpha_i=\sum_{j=1}^\ell n_{j,i} \chi_j$, with $n_{j,i}\in \Z$ for all $i,j$. We have $n_{j,i}=0$ for $j\in \llbracket d+1,\ell\rrbracket$. Set $\tilde{t}=\prod_{j=d+1}^\ell \chi_j^\vee(\chi_j(t))\in \fT(\cO)$. Then $\alpha_i(\tilde{t})=1$ for every $i\in I$ and thus $\tilde{t}\in \cZ_\cO$.

 For $j\in \llbracket 1,d\rrbracket$, write $\chi_j=\sum_{i\in I} m_{i,j} \alpha_i$, with $m_{i,j}\in \Q$ for every $i\in I$. Take $m\in \N^*$ such that $m m_{i,j}\in \Z$ for every $(i,j)\in I\times \llbracket 1,d\rrbracket$. Let $j\in \llbracket 1,d\rrbracket$. We have \[\chi_j(t)^{m}=\prod_{j=1}^d \alpha_j(t)^{mm_{i,j}} \in 1+\qp^n \cO.\] Using Lemma~\ref{l_decomposition_large_valuation} we can write $\chi_j(t)=b_j+c_j$, with $b_j,c_j\in \cO$ such that $b_j^m=1$ and $\omega(c_j)\geq n/K-K'$, with the same notation as in Lemma~\ref{l_decomposition_large_valuation}. Set $c_j'=c_j b_j^{-1}\in \cO$. As $b_j$ is a root of $1$, we have $\omega(b_j)=0$ and thus $\omega(c_j)=\omega(c_j')\geq n/K-K'$. We have $b_j+c_j=b_j(1+c_j')$. 
 
  Set $t'=\prod_{j=1}^d \chi_j^\vee(b_j)$ and $t''=\prod_{j=1}^d \chi_j^\vee(1+c'_j)$.   Then $\chi_j(t')=b_j$ and $\chi_j(t'')=1+c'_j$, for $j\in \llbracket 1,d\rrbracket$.  For $i\in I$, we have $\alpha_i(t)=\alpha_i(t')\alpha_i(t'')\alpha_i(\tilde{t})=\alpha_i(t')\alpha_i(t'')$ and $\alpha_i(t'')\in 1+\qp^{n/K-K'} \cO$ (when $n/K-K'\notin \N$, $\qp^{n/K-K'}\cO$ is just a notation for $\cK_{\omega\geq n/K-K'}$). As $\alpha_i(t)\in 1+\qp^{n/K-K'} \cO$ we deduce $\alpha_i(t')\in 1+\qp^{n/K-K'} \cO$ (replacing $K$ by $K+1$ if $K\leq 1$).
  
Let $F=\{\xi\in \cO\mid \xi^m=1\}$. Then $F$ is finite. Let $L'=\max\{\omega(\xi-1)\mid \xi \in F\setminus\{1\}\}$. Let $L\in \N$ be such that $L/K-K'\geq L'$. For  $n\in \Z_{\geq L}$, we have $n/K-K'\geq L'$, we have $\alpha_i(t')=1$ for $i\in I$ and $t'\in \cZ_\cO$. Maybe increasing $L$, we can assume that $K_1:=1/K-K'/L >0$. Then for $n\geq L$, we have $n/K-K'\geq n(1/K-K'/L)\geq nK_1$.
  
  Consequently, for $n\in \Z_{\geq L}$ and $t\in T_{n,\Phi}$, we have $t=t'\tilde{t}t''$, with $t'\tilde{t}\in \cZ_\cO$ and $t''\in T_{n/K_1}$, which proves the lemma.
\end{proof}

\begin{Proposition}\label{p_comparison_T_TFix}
The topology $\sT$ is finer than $\sT_\Fix$. If $\cK$ is Henselian, then we have $\sT=\sT_\Fix$ if and only if $\cZ_\cO=\{1\}$ if and only if $\sT_{\Fix}$ is Hausdorff. 
\end{Proposition}

\begin{proof}
 Let $x\in \I$ and $m\in \N^*$.  Then by Lemma~\ref{l_BPHR}, $U_{[-m\lambda,m\lambda]}^{pm+}$ and $U_{[-m\lambda,m\lambda]}^{nm-}$ fix $x$, for $m\gg 0$. By Lemma~\ref{l_T_fixes_I_for_large_valuation}, $T_{2m N(\lambda)}$ fix $x$ and thus $\cV_{m\lambda}$ fixes $x$ for $m\gg 0$.  Thus if $n\in \N^*$,  $\cV_{m\lambda}\subset G_{F_n}$ for $m\gg 0$ and  $\sT$ is finer than $\sT_\Fix$. 
 
 If $\sT=\sT_{\Fix}$, then by Theorem~\ref{t_conjugation_invariance}, $\sT_\Fix$ is Hausdorff. Therefore $\bigcap_{n\in \N^*}G_{F_n}=\cZ_\cO=\{1\}$ by Lemma~\ref{l_fixator_I}. 
 
 Assume $\cK$ is Henselian.  Let $m\in \N^*$. Then by Lemma~\ref{l_comparison_T_Tfix} and Lemma~\ref{l_comparison_TnPhi_Tn}, there exist $K_1\in \R^*_+$ and $L\in \N$ such that \[G_{F_n}\subset \cV_{m\lambda} .\cZ_\cO .T_{M(n)/{K_1}},\] for $n\geq \min(m,L)$, with $M(n)\underset{n\to +\infty}{\rightarrow} +\infty$. Therefore if $\cZ_\cO=1$ we have  \[G_{F_n}\subset \cV_{m\lambda}.T_{M(n)/{K_1}}\subset \cV_{m\lambda},\] for $n$ such that $M(n)/K\geq 2m N(\lambda)$, and thus $(\cV_{n\lambda})$ and $(G_{F_n})$ are equivalent, which proves the proposition.  
\end{proof}

\begin{Remark}\label{r_comparison_T_TFix}
\begin{enumerate}
\item If $(\alpha_i)_{i\in I}$ is a $\Z$-basis of $X$, then $\sT=\sT_{\Fix}$. Indeed, assume that $(\alpha_i)_{i\in I}$ is a $\Z$-basis of $X$. Let $(\chi_i^\vee)_{i\in I}$ be the dual basis. Let $n\in \N^*$ and $t\in T_{n,\Phi}\cap \fT(\cO)$. Write $t=\prod_{i\in I}\chi_i^\vee(a_i)$, with $a_i\in \cO^*$ for $i\in I$. Then $\pi_n(t)=\prod_{i\in I} \chi_i^\vee(\pi_n(a_i))$ and $\pi_n(t)=1$ if and only if $\pi_n(a_i)=1$ for all $i\in I$. Now $\alpha_i(t)=a_i$ and thus $t\in T_{n,\Phi}$ if and only if $t\in T_n$.  Therefore $\cZ_\cO= \bigcap_{n\in \N} T_n=\{1\}$. 

\item Note that by Lemma~\ref{l_comparison_T_Tfix},  the set of left $\fT(\cO)$-invariant open subsets of $G$ are the same for $\sT_\Fix$ and $\sT$. Indeed, let $V\subset G$ be a non empty left $\fT(\cO)$-invariant open subset of $G$ for $\sT_\Fix$. Then for every $v\in V$, there exists $n\in \N^*$ such that $vG_{F_n}\subset V$. By Lemma~\ref{l_comparison_T_Tfix}, $\fT(\cO). G_{F_n}\subset \cV_{n\lambda}$ and thus $V$ is open for $\sT$. 

\item Assume that $\cK$ is local. By 2), if $\tau \in \Hom_{\mathrm{Gr}}(Y,\C^*)$, then $I(\tau)_{\sT}=I(\tau)_{\sT_\Fix}$ (see \eqref{e_principal_series_representation} for the definition). Indeed, $\delta^{1/2}$ and $\tau$ are maps from $Y=T/\fT(\cO)$ to $\C^*$ and thus their extensions to $B$ are left $\fT(\cO)$-invariant. Therefore any element of $\widehat{I(\tau)}$ is left $\fT(\cO)$-invariant. 
\end{enumerate}
\end{Remark}

\section{Properties of the topologies}\label{s_Prop_topologies}

In this section, we study the properties of the topologies $\sT$ and $\sT_{\Fix}$. In \ref{ss_comparison_TKP}, we prove that when $G$ is not reductive, $\sT$ is strictly coarser than the Kac-Peterson topology on $G$ (Proposition~\ref{p_T_coarser_TKP}). In \ref{ss_Properties_usal_subgroups}, we prove that certain subgroups of $G$ are closed for $\sT$. In \ref{ss_non_compactness}, we prove that the compact subsets of $G$ have empty interior. In \ref{ss_example_affine_SL2}, we describe the topology in the case of affine $\mathrm{SL}_2$, under some assumption. 

\subsection{Comparison with the Kac-Peterson topology on $G$}\label{ss_comparison_TKP}

In \cite{kac1983regular}, Kac and Peterson defined a topology of topological group on $\fG(\C)$. This topology was then studied in \cite{hartnick2013topological} and generalized in  \cite[7]{hartnick2013topological}: Hartnick, Köhl and Mars define a topology of topological group on $\fG(\cF)$ for $\cF$ a local field (Archimedean or not), taking into account the topology of $\cF$.  The aim of this section is to prove that the topologies we defined on $G=\fG(\cK)$ are strictly coarser than the Kac-Peterson topology on $G$, unless $G$ is reductive. As $\sT_\Fix$ is coarser than $\sT$ it suffices to prove that $\sT$ is strictly coarser than $\sT_{KP}$. To that end, we prove that $\sT$ is coarser than $\sT_{KP}$ and that  the topologies induced by $\sT$ and $\sT_{KP}$ on a subset of  $B:=TU^+$\index{b@$B$}  differ, using the description of $\sT_{KP}|_B$ given in \cite[7]{hartnick2013topological}.

We assume that $\cK$ is local, in particular, $\omega(\cK^*)=\Z$. We equip $\mathrm{SL}_2(\cK)$ with the topology  associated to  $(\ker \pi^{\mathrm{SL_2}}_n)_{n\in \N^*}$, where $\pi_n^{\mathrm{SL}_2}:\mathrm{SL}_2(\cO)\rightarrow \mathrm{SL}_2(\cO/\qp^n \cO)$ is the natural projection.   We denote by $x_+$ (resp. $x_-$) the morphism of algebraic groups $a\mapsto \begin{psmallmatrix} 1 & a\\ 0 & 1\end{psmallmatrix}$ (resp. $a\mapsto \begin{psmallmatrix} 1 & 0\\ a & 1\end{psmallmatrix}$) for $a$ in a ring $\sR$.  
Using  Corollary~\ref{c_ker_pin}, it is easy to check that \[\ker \pi_n^{\mathrm{SL}_2}=x_+(\qp^n \cO).x_-(\qp^n\cO).\left(\begin{psmallmatrix} 1+\qp^n\cO & 0\\ 0 & 1+\qp^n \cO\end{psmallmatrix}\cap \mathrm{SL}_2(\cK)\right),\] and thus $(\cV_{n\lambda}^{\mathrm{SL}_2})$ is equivalent to ($\ker \pi_n^{\mathrm{SL}_2})$ for any regular $\lambda\in Y_{\mathrm{SL}_2}$. 

We equip $T$ with its usual topology $\sT_T$\index{t@$\sT_T$}, via the isomorphism $T\simeq (\cK^*)^m$, for $m$ the rank of $X$. This is the topology $\sT((\ker \pi_n|_T)_{n\in \N^*})$. As we shall see (Proposition~\ref{p_T_coarser_TKP}), this is the topology induced by $\sT$ on $T$.

For $\alpha\in \Phi_+$, let $\varphi_\alpha:\mathrm{SL}_2(\cK)\rightarrow G$ be defined by  $\varphi_{\alpha}\circ x_{\pm}=x_{\pm \alpha}$ and let $G_\alpha$\index{g@$G_\alpha$} be its image in $G$. We equip $G_\alpha$ with the quotient topology $\sT_{G_\alpha}$\index{t@$\sT_{G_\alpha}$} inherited from $\mathrm{SL}_2(\cK)$ via $\varphi_\alpha$.
Let $\Sigma=\{\alpha_i\mid i\in I\}$\index{s@$\Sigma$} and $\Sigma^{(\N)}$ be the set of finite sequences of elements of $\Sigma$. For $\underline{\alpha}=(\alpha_0,\ldots,\alpha_k)\in \Sigma^{(\N)}$, where $k\in \N^*$, one sets $G_{\underline{\alpha}}=G_{\alpha_0}\ldots G_{\alpha_k}\subset G$. Note that $G_{\underline{\alpha}}$ is not a subgroup of $G$ in general. If $\underline{\alpha},\underline{\beta}\in \Sigma^{(\N)}$, we write $\underline{\alpha}\leq \underline{\beta}$ if $\underline{\alpha}$ appears as an ordered subtuple of $\underline{\beta}$. Then $\leq$ is a preorder on  $\Sigma^{(\N)}$ and $(\Sigma^{(\N)},\leq)$ is a directed poset.

We equip $TG_{\underline{\alpha}}$ with the topology $\sT_{TG_{\underline{\alpha}}}$ obtained as the quotient topology with respect to the multiplication map 
\begin{equation}\label{e_multiplication}
m_{\underline{\alpha}}:(T,\sT_T) \times (G_{\alpha_1},\sT_{G_{\alpha_1}})\times \ldots \times  (G_{\alpha_k},\sT_{G_{\alpha_k}})\twoheadrightarrow TG_{\underline{\alpha}}.
\end{equation}

In \cite[Definition 7.8]{hartnick2013topological}, the authors define the Kac-Peterson topology $\sT_{KP}$ on $G$ as the direct limit of the directed system $\{(TG_{\underline{\alpha}},\sT_{TG_{\underline{\alpha}}})\mid \underline{\alpha}\in \Sigma^{(\N)}\}$. In other words, a subset $V$ of $G$ is open for $\sT_{KP}$ if and only if $V\cap TG_{\underline{\alpha}}$ is open for every $\underline{\alpha}\in \Sigma^{(\N)}$.

\begin{Lemma}\label{l_weak_order_non_bounded}
Let $w\in W^v$. Assume that $wr_i<w$ for every $i\in I$ (for the Bruhat order on $W^v$). Then $W^v$ is finite.
\end{Lemma}

\begin{proof}
By \cite[1.3.13 Lemma]{kumar2002kac}, we have $w.\alpha_i\in \Phi_-$ for every $i\in I$. Let $\lambda\in C^v_f$. Then $\alpha_i(w^{-1}.\lambda)<0$ for every $i\in I$ and thus $w^{-1}.\lambda\in -C^v_f$. Thus $\lambda\in \mathring{\T}\cap -\mathring{\T}$. By \cite[1.4.2 Proposition]{kumar2002kac} we deduce that $\Phi$ is finite and thus $W^v$ is finite. 
\end{proof}

We equip $W^v$ with the \textbf{right weak Bruhat order} $\preceq$: for every $v,w\in W^v$, $v\preceq w$ if $\ell(v)+\ell(v^{-1}w)=\ell(w)$. We assume that $W^v$ is infinite.  By Lemma~\ref{l_weak_order_non_bounded}, there exists a sequence $(w_i)_{i\in \N}\in (W^v)^\N$ such that $w_0=1$, $\ell(w_{i+1})=\ell(w_i)+1$ and $w_i\preceq w_{i+1}$ for every $i\in \N$. 

For $w\in W^v$, one sets $\Inv(w)=\{\alpha\in \Phi_+\mid w^{-1}.\alpha\in \Phi_-\}$\index{I@$\Inv(w)$}. Let $U_w=\langle U_\alpha\mid \alpha\in \Inv(w)\}$. By \cite[Lemma 5.8]{caprace2009group}, if $w=r_{i_1}\ldots r_{i_k}$, with $k=\ell(w)$ and $i_1,\ldots,i_k\in I$, then $U_w=U_{\alpha_{i_1}}.U_{r_{i_1}.\alpha_{i_2}}\ldots U_{r_{i_1}\ldots r_{i_{k-1}}.\alpha_{i_k}}$ and every element of $U_w$ admits  a unique decomposition in this product. By \cite[Proposition 7.27]{hartnick2013topological}, as a topological space, $B$ is the colimit $\lim_{\rightarrow} TU_w$ (note that $(W^v,\preceq)$ is not directed). Let $U'=\bigcup_{n\in \N}U_{w_n}$. Then the topology induced on $TU'$ by $\sT_{KP}$ is the topology of the direct limit $\lim_{\rightarrow} TU_{w_n}$: a subset $V$ of $TU'$ is open if and only if $V\cap TU_{w_n}$ is open for every $n\in \N$.

For $n\in \N$, write $w_{n+1}=w_n r_i$, where $i\in I$. Set $\beta[n]=w_n.\alpha_i$. Then $\Inv(w_n)=\{\beta[i]\mid i\in \llbracket 1,n\rrbracket\}$, by \cite[1.3.14 Lemma]{kumar2002kac}.

By \cite[Lemma 7.26]{hartnick2013topological}, if $n\in \N^*$ then  the map $m=m_n:T\times(\cK)^n\rightarrow T U_{w_n}$ defined by $m(t,a_1,\ldots,a_n)=tx_{\beta[1]}(a_1)\ldots x_{\beta[n]}(a_n)$ is a homeomorphism, when $TU_{w_n}$ is equipped with the the restriction of $\sT_{KP}$.

Recall that $\sT=\sT\left( (\cV_{n\lambda})\right)$ for any $\lambda\in Y^+\cap C^v_f$.

Define $\htt:Q_+=\bigoplus_{i\in I} \Z \alpha_i \rightarrow \Z$ by $\htt(\sum_{i\in I} n_i\alpha_i)=\sum_{i\in I} n_i$, for $(n_i)\in \Z^I$.

\begin{Lemma}\label{l_nonequality_topologies}
Assume that $W^v$ is infinite. For $n\in  \N^*$, set   $V_n=T\prod_{i=1}^n x_{\beta[i]}(\qp^{(\htt(\beta[i])!} \cO)$ and set $V=\bigcup_{n\in \N} V_n$. Then $V$ is open in $(TU',\sT_{KP})$ but not in $(TU',\sT)$. In particular, $\sT$ and $\sT_{KP}$ are different. 
\end{Lemma}

\begin{proof}
Let $n\in \N^*$. Let $v\in V\cap TU_{w_n}$ and choose $k\in \N^*$ such that $v\in V_k$. If $k\leq n$, then $v\in V_k\subset V_n$. Suppose now $k\geq n$. Write $v=t\prod_{i=1}^k x_{\beta[i]}\left(\qp^{(\htt(\beta_i)!}a_i\right)$, with $a_1,\ldots,a_k\in \cO$ and $t\in T$. By \cite[Lemma 7.26]{hartnick2013topological}, we have $a_i=0$ for every $i\in \llbracket n+1,k\rrbracket$ and thus $v\in V_n$. Therefore $V\cap TU_{w_n}=V_n$. By \cite[Lemma 7.26]{hartnick2013topological}, $V_n$ is open in $TU_{w_n}$ and thus $V$ is open in $(TU',\sT_{KP})$.

 Let $\lambda\in C^v_f\cap Y$ be such that $\alpha_i(\lambda)=1$ for every $i\in I$. Let us prove that for every $n\in \N^*$, $U'\cap U_{-n\lambda}^{pm+}$ is not contained in $V$. For $k,n\in \N^*$, set $x_{k,n}=\prod_{i=1}^k x_{\beta[i]}(\qp^{n\htt(\beta[i])})\in U'\cap U_{-n\lambda}^{pm+}$. Let $n\in \N^*$. By \cite[Lemma 7.26]{hartnick2013topological}, if $x_{k,n}\in V$, then $n\htt(\beta[i])\geq \left(\htt(\beta[i])\right)!$, for every $i\in \llbracket 0,k\rrbracket$. As $\htt(\beta[i])\underset{i\to +\infty}{\longrightarrow} +\infty$, there exists $k\in \N$ such that $x_{k,n}\in U'\cap U_{-n\lambda}^{pm+}\setminus V$ and thus, $U'\cap U_{-n\lambda}^{pm+}\not\subset V$. Using Lemma~\ref{l_uniqueness_decomposition_UUT} we deduce that  there exists no $n\in \N^*$ such that $\cV_{n\lambda}\cap TU'\subset V$ and thus $V$ is not open for $\sT$. 
\end{proof}

\begin{Proposition}\label{p_sufficient_condition_coarser}
Let $\underline{V}=(V_n)_{n\in \N^*}$ be a conjugation-invariant filtration of $G$. Let $\sT_{\underline{V}}$ be the associated topology on $G$. We assume that for every $\alpha\in \Sigma$, the induced topology on $G_{\alpha}$ is $\sT_{G_{\alpha}}$ and that the induced topology on $T$ is $\sT_T$. Then $\sT_{\underline{V}}$ is coarser than $\sT_{KP}$.
\end{Proposition}

\begin{proof}
Let $n\in \N^*$. Let us prove that $V_n$ is open for $\sT_{KP}$. Let $\underline{\alpha}=(\alpha_k,\ldots,\alpha_1)\in \Sigma^{(\N)}$. Let us prove that $m_{\underline{\alpha}}^{-1}(V_n)$ is open in $T\times G_{\alpha_k}\times \ldots \times G_{\alpha_1}$, with the notation of \eqref{e_multiplication}. Let $v\in TG_{\underline{\alpha}}\cap V_n$ and  $(t,v_k,\ldots,v_1)\in m_{\underline{\alpha}}^{-1}(\{v\})$. We have  $v=tv_k \ldots v_1$. We set $n_1=n$ and we choose $n_2,\ldots,n_k,n_{k+1}\in \N^*$ such that for all $j\in \llbracket 2,k\rrbracket$, we have $V_{n_j}v_{j-1}\ldots v_1\subset v_{j-1}\ldots v_1 V_n$, which is possible since $\underline{V}$ is conjugation-invariant. Then we have: \begin{align*}
V_{n_{k+1}}v_kV_{n_k} v_{k-1}V_{n_{k-1}}\ldots V_{n_3} v_2 V_{n_2}v_1 V_{n_1}&\subset V_{n_{k+1}}v_kV_{n_k} v_{k-1}V_{n_{k-1}}\ldots V_{n_3} v_2 v_1 V_n\\ &\subset \ldots \subset v_k\ldots v_1 V_n .\end{align*} Consequently \[\left(t V_{n_{k+1}}\cap T\right)\left(v_kV_{n_k}\cap G_{\alpha_k}\right) \ldots \left(v_1 V_{n_1}\cap G_{\alpha_1}\right)\subset tv_k\ldots v_1V_n=v V_n=V_n\] and hence \[m_{\underline{\alpha}}\left(\left(t V_{n_{k+1}}\cap T\right)\times\left(v_kV_{n_k}\cap G_{\alpha_k}\right) \times\ldots\times  \left(v_1 V_{n_1}\cap G_{\alpha_1}\right)\right)\subset V_n\cap TG_{\underline{\alpha}}.\] Therefore $m_{\underline{\alpha}}^{-1}(V_n\cap TG_{\underline{\alpha}})$ is open or equivalently $V_n\cap TG_{\underline{\alpha}}$ is open. As this is true for every $\underline{\alpha}\in \Sigma^{(\N)}$, we deduce that $V_n$ is open for $\sT_{KP}$. As $(G,\sT_{KP})$ is a topological group, we deduce that $xV_n$ is open in $\sT_{KP}$ for every $x\in G$ and $n\in \N^*$ and we deduce that $\sT_{\underline{V}}$ is coarser than $\sT_{KP}$.
\end{proof}

\begin{Proposition}\label{p_T_coarser_TKP}
\begin{enumerate}
\item Let $\alpha\in \Phi_+$ and $\varphi_\alpha:\mathrm{SL}_2(\cK)\rightarrow G$ be the group morphism defined by $\varphi_\alpha\circ x_{\pm} =x_{\pm\alpha}$. Fix a basis $(\chi_1^\vee,\ldots,\chi_\ell^\vee)$ of $Y$ and define $\iota:(\cK^*)^\ell\overset{\sim}{\rightarrow} T\subset G$ by $\iota\left((a_1,\ldots,a_\ell)\right)=\chi_1^\vee(a_1)\ldots  \chi_\ell^\vee(a_\ell)$, for $a_1,\ldots,a_\ell\in \cK^*$. Then the $\varphi_\alpha$, $\alpha\in \Phi$ and $\iota$ are continuous when $G$ is equipped with $\sT$.

\item The topology induced by $\sT$ on $T$ is $\sT_T$ and if $\alpha\in \Phi_+$, then the topology induced by  $\sT$ on $G_\alpha$ is $\sT_{G_\alpha}$

\end{enumerate}
\end{Proposition}

\begin{proof}
Let $\alpha\in \Phi$ and $\lambda\in Y^+$ be regular. Let $g\in \varphi_\alpha^{-1}(\cV_\lambda)$. By \cite[3.16]{rousseau2016groupes} and \cite[1.2.4 Proposition]{remy2002groupes}, we have the Birkhoff decomposition in $\mathrm{SL}_2(\cK)$ (where $N_{\mathrm{SL}_2}$ is the set of monomial matrices with coefficient in $\cK^*$) and $G$: \[\mathrm{SL}_2(\cK)=\bigsqcup_{n\in N_{\mathrm{SL_2}}} x_+(\cK) n x_-(\cK) \text{ and }G=\bigsqcup_{n\in N} U^+ n U^-.\] Let $n\in N_{\mathrm{SL}_2}$  be such that $g\in x_+(\cK) n x_-(\cK)$. If $n\notin T_{\mathrm{SL}_2}$, then $\varphi_\alpha(g)\in U^+ \varphi_\alpha(n) U^-$ and $\nu^v(\varphi_\alpha(n))$ acts as the reflection with respect to $\alpha$ on $\A$. Then $\varphi_\alpha(g)\notin U^+ T U^-$ which contradicts Corollary~\ref{c_ker_pin}. Therefore $n\in T$. Write $g=x_+(a_+) n x_-(a_-)$, with $a_+,a_-\in \cK$. Then by Lemma~\ref{l_uniqueness_decomposition_UUT}, we have $x_\alpha(a_+)\in U_{[-\lambda,\lambda]}^{pm+}$, $x_{-\alpha}(a_-)\in U_{[-\lambda,\lambda]}^{nm-}$ and $\varphi_\alpha(n)\in T_{2N(\lambda)}$. Consequently, $\omega(a_+), \omega(a_-)\geq |\alpha(\lambda)|$ and \[t\in T_{\mathrm{SL_2},2N(\lambda)}:=\begin{psmallmatrix} 1+\qp^{2N(\lambda)}\cO & 0 \\ 0 & 1+\qp^{2N(\lambda)} \cO\end{psmallmatrix}\cap \mathrm{SL}_2(\cK).\] Therefore $\varphi_\alpha^{-1}(\cV_{\lambda})\subset x_{+}(\cK_{\omega\geq |\alpha(\lambda)|})  T_{\mathrm{SL_2},2N(\lambda)} x_-(\cK_{\omega\geq |\alpha(\lambda)|})$. Conversely, $\varphi_\alpha(x_-(\cK_{\omega\geq |\alpha(\lambda)|}))$, $\varphi_\alpha(x_+(\cK_{\omega\geq |\alpha(\lambda)|})$, $\varphi_\alpha(T_{\mathrm{SL_2},2N(\lambda)})\subset \cV_{\lambda}$ and thus $\varphi_\alpha^{-1}(\cV_{\lambda})= x_{+}(\cK_{\omega\geq |\alpha(\lambda)|})  T_{\mathrm{SL_2},2N(\lambda)} x_-(\cK_{\omega\geq |\alpha(\lambda)|})$ is open in $\mathrm{SL}_2(\cK)$. Therefore $\varphi_\alpha$ is continuous and $\cV_\lambda\cap G_\alpha$ is open. 

Let $n\in\N^*$. Then $\iota^{-1}(T_n)=(1+\qp^n\cO)^\ell$ and thus $\iota$ is continuous. Moreover, we have $T\cap \cV_\lambda=T_{2N(\lambda)}$, by the Birkhoff decomposition, which proves that $\sT$ induces $\sT_T$ on $T$.
\end{proof}

\begin{Corollary}\label{c_T_different_TKP}
If $\Phi$ is infinite, the topologies $\sT$ and $\sT_{\mathrm{Fix}}$ are strictly coarser than $\sT_{KP}$.
\end{Corollary}

\begin{proof}
Let $\lambda\in Y^+$ be any regular element. By Propositions~\ref{p_T_coarser_TKP} and \ref{p_sufficient_condition_coarser}, applied with $\underline{V}=(\cV_{n\lambda})$, $\sT$ is coarser than $\sT_{KP}$. By Lemma~\ref{l_nonequality_topologies}, $\sT$ is different from $\sT_{KP}$. As $\sT_{\mathrm{Fix}}$ is coarser than $\sT$ (by Proposition~\ref{p_comparison_T_TFix}), we deduce the result.
\end{proof}

\subsection{Properties of usual subgroups of $G$ for $\sT$ and $\sT_\Fix$}\label{ss_Properties_usal_subgroups}

In this subsection, we prove that many subgroups important in this theory (such as $B$, $T$, $U_\alpha$, $\alpha\in \Phi$, etc.) are open or closed. We have $\cT_{\Fix}\subset \cT$ and thus every subset of $G$ open or closed for $\cT_{\Fix}$ is open or closed for $\cT$. As the Kac-Peterson topology $\sT_{KP}$ is finer than $\sT$, this improves the corresponding results of \cite{hartnick2013topological}. Note that we consider $B=B^+$ and $U^+$, but the same results hold for $B^-$ and $U^-$, by symmetry.

If $g\in G$, we say that $g$ \textbf{stabilizes} (resp. \textbf{pointwise fixes}) $+\infty$ if $g.+\infty=+\infty$ (resp. if there exists $Q\in +\infty$ such that $g$ pointwise fixes $Q$). We denote by $\Stab_G(+\infty)$ the stabilizer of $+\infty$ in $G$. 

By \cite[3.4.1]{hebert2018study}, $\Stab_G(+\infty)=B:=TU^+$. We denote by $\Ch (\partial \I^+)$ the set of positive sector-germs at infinity of $\I$.  For $c\in \Ch (\partial \I^+)$ and $x\in \I$, there exists an apartment $A$ containing $x$ and $c$. We denote by \begin{equation}\label{e_def_x+c}
x+c
\end{equation} the convex hull of $x$ and $c$ in this apartment. This does not depend on the choice of $A$, by (MA II).  Fix $\lambda_0\in C^v_f$. For $r\in \R_+$, we set $\cC_{r}=\{c\in \Ch(\partial \I^+)\mid [0,r.\lambda_0]\subset 0+c\}$. This set is introduced in \cite[Definition 3.1]{ciobotaru2020cone} where it is denoted $U_{0,r,c}$ or $U_{r,c}$.

\begin{Proposition}\label{p_B_closed}
\begin{enumerate}
\item The subgroup $B$ is closed in $G$ for $\sT$ and $\sT_{\Fix}$.

\item The subgroup $U^+$ of $G$ is closed for $\sT$. It is closed for $\sT_{\Fix}$ if and only if $\sT=\sT_{\Fix}$. 

\end{enumerate}
\end{Proposition}

\begin{proof}
1) Let $g\in G\setminus B$. Then  $g.(+\infty)\neq +\infty$.  By \cite[Lemma 7.6]{ciobotaru2020cone}, $\bigcap_{r\in \R_+} \cC_r=\{+\infty\}$. Thus there exists $n\in \N^*$ such that $g^{-1}.+\infty\notin \cC_n$.  Then $\cV_{n\lambda_0}$ fix $[0,n \lambda_0]$. Let $v\in \cV_{n\lambda_0}$. Then \[v.(0+\infty)=v.0+(v.+\infty)=0+(v.+\infty)\supset v.[0,n\lambda_0]=[0,n\lambda_0].\] Therefore $\cV_{n\lambda_0}.(+\infty)\subset \cC_{n}$. Consequently, $g^{-1}.(+\infty)\notin \cV_{n\lambda_0} .+\infty$ and thus $g\cV_{n\lambda_0}.+\infty \not\ni +\infty$. Thus $g.\cV_{n\lambda_0}\subset G\setminus B$, which proves that $G\setminus B$ is open for $\sT_{\Fix}$. 

2)  Let $g\in G\setminus U^+$. If $g\in G\setminus U^+T$, then by 1), there exists $V\in \cT$ such that $gV\subset G\setminus U^+ T\subset G\setminus U^+$. We now assume $g\in U^+T\setminus U$. Write $g= u_+t$, with $u_+\in U^+$ and $t\in T\setminus \{1\}$. Let $\lambda\in Y\cap C^v_f$ and assume that $g \cV_\lambda\cap U^+\neq \emptyset$. Then there exists $(u_+',u_-',t')\in U_{[-\lambda,\lambda]}^{pm+} \times U_{[-\lambda,\lambda]}^{nm-}\times T_{2N(\lambda)}$ such that $u_+t u'_+ u'_- t'=u''_+$, where $u_+''\in U^+$.  As $t$ normalizes $U^+$ and $U^-$, we can write $tu'_+ u'_-=u^{(3)}_+ u_-^{(3)} t$, for some $u_+^{(3)}\in U^+,u_-^{(3)}\in U^-$. Then we have \[u_+''^{-1} u_+ u^{(3)}_+ u^{(3)}_- tt'=1. \] By Lemma~\ref{l_uniqueness_decomposition_UUT} we deduce $tt'=1$. Therefore $t\in T_{2N(\lambda)}$. Thus if $\lambda$ is sufficiently dominant, $g \cV_\lambda\cap U^+=\emptyset$ and $g\cV_\lambda \subset G\setminus U^+$. We deduce that $G\setminus U^+$ is closed for $\cT$.

Suppose now $\sT\neq \sT_{\Fix}$. Then by Proposition~\ref{p_comparison_T_TFix},  $\cZ_\cO\neq \{1\}$. Then  every non empty open subset of $G$ for $\cT_{\Fix}$ contains $\cZ_\cO$.  Take $z\in \cZ_\cO\setminus \{1\}$. As $z\in \fT(\cO)$,  $z\in G\setminus  U^+$. Moreover, for any non empty open subset $V$ of $G$, $z V\ni 1\in U^+$. Therefore $G\setminus U^+$ is not open, which completes the proof of the proposition. 

\end{proof}

\begin{Proposition}\label{p_fixator_open}
\begin{enumerate}
\item Let $x\in \I$. Then the fixator $G_x$ of $x$ in $G$ is open (for $\cT_{\Fix}$ and $\cT$). In particular, $\fG^{\min}(\cO)$ is open in $G$. 

\item Let $E\subset \I$. Then the fixator and the stabilizer of $E$ in $G$ are closed for $\cT_{\Fix}$ and $\cT$.

\end{enumerate}
\end{Proposition}

\begin{proof}
1) By Lemma~\ref{l_GFn_eventually_fix_I}, $G_{F_n}\subset G_x$ for $n\gg 0$, which proves 1). 

2) Let $g\in G\setminus G_E$. Let $x\in E$ be such that $g.x\neq x$. Then $g.G_x\subset G\setminus G_E$ and hence $G_E$ is open. Let $g\in G\setminus \Stab_G(E)$. Let $x\in E$ be such that $g.x\notin E$. Then $g.G_x\subset G\setminus \Stab_G(E)$ and hence $\Stab_G(E)$ is closed.

\end{proof}

\begin{Corollary}\label{c_N_T_open}
The subgroups $N$ and $T$ are closed in $G$ for $\cT_{\Fix}$ and $\cT$.
\end{Corollary}

\begin{proof}
By Proposition~\ref{p_fixator_open},  $N=\Stab_G(\A)$ is closed. We have $T=\Stab_G(+\infty)\cap N$. Indeed, it is clear that $T\subset \Stab_G(+\infty)\cap N$. Conversely, let $g\in \Stab_G(+\infty)\cap N$. Let $w\in W^v$ and $\lambda\in \Lambda$ be such that $g.x=\lambda+w.x$ for every $x\in  \A$. Then $w.C^v_f=C^v_f$ and thus $w=1$. Therefore $g$ acts by translation on $\A$ and hence $g\in T$. This proves that $T=\Stab_G(+\infty)\cap N$ and we conclude with Proposition~\ref{p_B_closed}. 
\end{proof}

\begin{Remark}
\begin{enumerate}
\item The fixator $K_I$ of $C_0^+$ is open for $\sT_\Fix$. Indeed, let $\lambda\in C^v_f\cap Y$. Then $G_{[0,\lambda]}=G_0\cap G_\lambda$ is open for $\sT_{\Fix}$ and $G_{[0,\lambda]}\subset K_I$.

\item For $x,y\in \I$, one writes $x\leq y$ if there exists $g\in G$ such that $g.x,g.y\in \A$ and $g.y-g.x\in \T=\bigcup_{w\in W^v} w.\overline{C^v_f}$. By $W^v$-invariance of $\T$, this does not depend on the choice of $g$ and by \cite[Théorème 5.9]{rousseau2011masures}, $\leq$ is a preorder on $\I$. One sets \[G^+=\{g\in G\mid g.0\geq 0\}.\] This is a subsemigroup of $G$  which is crucial for the definition of the Hecke algebras associated with $G$ (when $\cK$ is local), see \cite{braverman2011spherical}, \cite{braverman2016iwahori}, \cite{gaussent2014spherical} or  \cite{bardy2016iwahori}. Then $G^+\supset G_0=\fG^{\min}(\cO)$ and thus $G^+$ is open in $G$. 

\end{enumerate}
\end{Remark}

\begin{Lemma}\label{l_action_elements_staying_A}
Let $g\in G$. Then there exists $n\in N$ such that $g.a=n.a$ for every $a\in \A\cap g^{-1}.\A$. 
\end{Lemma}

\begin{proof}
Let $h\in G$ be such that $hg.\A=\A$ and $h$ fixes $\A\cap g.\A$, which exists by (MA II). Then $n:=hg$ stabilizes $\A$ and thus it belongs to $N$. Moreover, $hg.a=n.a=g.a$ for every $a\in \A\cap g^{-1}.\A$, which proves the lemma. 
\end{proof}

\begin{Lemma}\label{l_description_UalphaT_stabilizer}
Let $\alpha\in \Phi$. Write $\alpha=\epsilon w.\alpha_i$, for $w\in W^v$, $\epsilon\in \{-,+\}$  and $i\in I$. Let $\fQ$ be the sector-germ at infinity of   $-\epsilon wr_i(C^v_f)$. Then $U_{\alpha} T=\Stab_{G}(w.\epsilon\infty)\cap \Stab_G(\fQ)$. 

\end{Lemma}

\begin{proof}
There is no loss of generality in assuming that  $w=1$ and $\epsilon=+$.  Let $u\in U_{\alpha_i}$. Then $u$ fixes a translate of $\alpha_i^{-1}(\R_+)$. Therefore $TU_{\alpha_i}$ stabilizes $\fQ$ and $+\infty$. Conversely, let $g\in\Stab_{G}(+\infty)\cap \Stab_G(\fQ)$. Then   there exist $x,x'\in \A$ such that $g.(x+C^v_f)= x'+C^v_f$. Then by Lemma~\ref{l_action_elements_staying_A}, there exists $n\in N$ such that $g.x''=n.x''$ for every $x''\in \A\cap g^{-1}.\A$. Then $n$ fixes $+\infty$ and thus $n\in T$ (by the proof of Corollary~\ref{c_N_T_open}). Then $n^{-1}g.x=x$. Considering $n^{-1}g$ instead of $g$,  we may assume that $g$ pointwise fixes $+\infty$. Therefore $g$ pointwise fixes $\fQ$. There exists $a,a'\in \A$ such that $g$ fixes $a+C^v_f$ and $g$ fixes $a'-r_i(C^v_f)$. Let $A=g.\A$. Then $A\cap \A$ is a finite intersection of half-apartments by (MA II) and thus either $A=\A$ or $A\cap \A$ is a translate of $\alpha_i^{-1}(\R_+)$. Moreover, $g$ fixes $A\cap \A$ since it fixes an open subset of $A\cap \A$. By \cite[5.7 3)]{rousseau2016groupes}, $g\in U_{\alpha_i} \fT(\cO)$. Consequently $\Stab_{G}(+\infty)\cap \Stab_G(\fQ)\subset U_{\alpha_i}T$, and the lemma follows.
\end{proof}

\begin{Proposition}
Let $\alpha\in \Phi$. 

\begin{enumerate}
\item The group $U_\alpha T$ is closed for $\cT$ and $\cT_{\Fix}$. 

\item The group $U_{\alpha}$ is closed for $\cT$. 

\item If $\sT\neq \sT_\Fix$, then $U_\alpha$ is not closed for $\sT_\Fix$. 
\end{enumerate}
\end{Proposition}

\begin{proof}
1)  Let $\fQ$ be a sector-germ of $\I$ (positive or negative). Then by Proposition~\ref{p_B_closed} (or the similar proposition for $B^-$ if $\fQ$ is negative), $\mathrm{Stab}_G(\fQ)$ is closed in $G$ for $\sT_{\Fix}$.  By  Lemma~\ref{l_description_UalphaT_stabilizer} and Proposition~\ref{p_fixator_open}, we have 1). 

2) We have $U_\alpha=U_\alpha T\cap U^+$, by Lemma~\ref{l_uniqueness_decomposition_UUT} and thus 2) follows from 1) and Proposition~\ref{p_B_closed}. The proof of 3) is similar to the proof of the corresponding result of Proposition~\ref{p_B_closed}. 
\end{proof}

\subsection{Compact subsets have empty interior}\label{ss_non_compactness}

By  \cite[Theorem 3.1]{abdellatif2019completed}, for any topology of topological group on $G$, $G_{C_0^+}$ or $G_0$ are not compact and open. In particular, $G_0$ and $G_{C_0^+}$ are not compact for $\sT$.  With a similar reasoning, we have the following.

\begin{Proposition}\label{p_nonCompactness}
Assume that $W^v$ is infinite. \begin{enumerate}
\item Let $n\in \N^*$ and $\lambda\in Y^+$ regular. Then $\cV_{n\lambda}/\cV_{(n+1)\lambda}$ is infinite.

\item Every compact subset of $(G,\sT)$  has empty interior. 
\end{enumerate}  
\end{Proposition}

\begin{proof}
1) Set $H=G_{[-n\lambda,(n+1)\lambda]}\subset G$. Then $H\supset \cV_{(n+1)\lambda}$, by \eqref{e_cV_as_fixators}. Thus $|\cV_{n\lambda}/\cV_{(n+1)\lambda}|\geq |\cV_{n\lambda}/(H\cap \cV_{n\lambda})|$ and  it suffices to prove that $\cV_{n\lambda}/ (H\cap \cV_{n\lambda})$ is infinite. We have $\cV_{n\lambda}=\bigsqcup_{v\in \cV_{n\lambda}/(H\cap \cV_{n\lambda})} v.(H\cap \cV_{n\lambda})$. Moreover if $v,v'\in \cV_{n\lambda}$, then $v.((n+1)\lambda)=v'.((n+1)\lambda)$ if and only if $v'. (G_{(n+1)\lambda}\cap \cV_{n \lambda})=v.(G_{(n+1)\lambda}\cap \cV_{n\lambda})$. 

Let us prove that $G_{(n+1)\lambda}\cap \cV_{n\lambda}=H\cap \cV_{n\lambda}$.  Let $g\in G_{(n+1)\lambda}\cap \cV_{n\lambda}$. Then by \eqref{e_cV_as_fixators}, $g$ fixes $[-n\lambda,n\lambda]$ and $(n+1)\lambda$. Then $g.\A$ is an apartment containing $[-n\lambda,n\lambda]\cup \{(n+1)\lambda\}$. As $g.\A\cap \A$ is convex, $g.\A$ contains $[-n\lambda,(n+1)\lambda]$. By (MA II), there exists $h\in G$ such that $g.\A=h.\A$ and $h$ fixes $\A\cap g.\A$. Then $h^{-1}g.\A=\A$ and $h^{-1}g$ acts on $\A$ by an affine map. As $h^{-1}g$ fixes $[-n\lambda,n\lambda]$, it fixes $[-n\lambda,(n+1)\lambda]$. Therefore $g$ fixes $[-n\lambda,(n+1)\lambda]$ and thus $g\in H$. Therefore $G_{(n+1)\lambda}\cap \cV_{n\lambda}=H\cap \cV_{n\lambda}$. Consequently, \[\cV_{n\lambda}.((n+1)\lambda)=\bigsqcup_{v\in \cV_{n\lambda}/(H\cap \cV_{n\lambda})} \{v.(n+1)\lambda\}\text{ and }|\cV_{n\lambda}/(H\cap \cV_{n\lambda})|=|\cV_{n\lambda}.\left((n+1)\lambda\right)|.\]

Let $(\beta_\ell)\in (\Phi_+)^\N$ be an injective sequence. Write $\beta_\ell=\sum_{i\in I} m_i^{(\ell)} \alpha_i$, with $m_i^{(\ell)}\in \N$ for $\ell\in \N$. Then $\beta_\ell(\lambda)\geq (\sum_{i\in I} m_i^{(\ell)} )(\min_{i\in I} \alpha_i(\lambda))\underset{\ell\to\infty}{\rightarrow} +\infty$. Let $\ell\in \N$. For $k\in \llbracket \beta_\ell(n\lambda),\beta_\ell((n+1)\lambda)-1\rrbracket$, $x_{-\beta_\ell}(\qp^k)\in U_{[-n\lambda,n\lambda]}^{nm-}\subset \cV_{n\lambda}$.  Set \[x_k=x_{-\beta_\ell}(\qp^k).((n+1)\lambda)\in \cV_{n\lambda}.((n+1)\lambda).\] Let $k'\in \llbracket \beta_\ell(n\lambda),\beta_\ell((n+1)\lambda)-1\rrbracket$. Then $x_k=x_{k'}$ if and only if  $x_{-\beta_\ell}(\qp^k).((n+1)\lambda)=x_{-\beta_\ell}(\qp^{k'}).((n+1)\lambda)$ if and only if $x_{-\beta_\ell}(\qp^k-\qp^{k'}).((n+1)\lambda)=(n+1)\lambda$ if and only if $\omega(\qp^k-\qp^{k'})\geq (n+1)\beta_\ell(\lambda)$ if and only $k=k'$. Therefore $|\cV_{n\lambda}.((n+1)\lambda)|\geq \beta_\ell(\lambda)$. As this is true for every $\ell\in \N$, $|\cV_{n\lambda}.((n+1)\lambda)|$ if infinite, which proves 1).

2) Let $V$ be a compact subset of $G$ and assume that $V$ has non empty interior. Considering $v^{-1}.V$ instead of $V$, we may assume $1\in V$. Then there exists $\lambda\in Y^+\cap C^v_f$ such that $\cV_\lambda\subset V$, and we have $\cV_{2\lambda}\subset \cV_\lambda$. As $\cV_\lambda$ is closed, it is compact. By 1), $\cV_\lambda/\cV_{2\lambda}$  is infinite.  Therefore $\cV_\lambda=\bigsqcup_{v\in \cV_{\lambda}/\cV_{2\lambda}}v.\cV_{2\lambda}$ is a cover of $\cV_\lambda$ by open subsets from which we can not extract a finite subcover: we reach a contradiction. Thus every compact subset of $G$ has empty interior. 

\end{proof}

\subsection{Example of affine $\mathrm{SL}_2$}\label{ss_example_affine_SL2}

In this subsection, we determine an explicit filtration equivalent to $(\cV_{n\lambda})$ in the case of affine $\mathrm{SL}_2$ (quotiented by the central extension).

  Let $Y=\Z{\mathring{\alpha}}^\vee\oplus  \Z d$, where ${\mathring{\alpha}}^\vee,d$ are some symbols, corresponding to the positive root of $\mathrm{SL}_2(\cK)$   and to the semi-direct extension by $\cK^*$ respectively. Let $X=\Z {\mathring{\alpha}}\oplus \Z \delta$, where ${\mathring{\alpha}}$\index{a@${\mathring{\alpha}}$}, $\delta$\index{d@$\delta$} $:Y\rightarrow \Z$ are the $\Z$-module morphisms defined by ${\mathring{\alpha}}({\mathring{\alpha}}^\vee)=2$, $ {\mathring{\alpha}}(d)=0$,  $\delta({\mathring{\alpha}}^\vee)=0$ and  $\delta(d)=1$.  Let $\alpha_0=\delta-{\mathring{\alpha}}$, $\alpha_1={\mathring{\alpha}}$, $\alpha_0^\vee=-{\mathring{\alpha}}^\vee$ and $\alpha_1^\vee={\mathring{\alpha}}^\vee$. Then $\mathcal{S}=(\left(\begin{smallmatrix} 2 & -2\\ -2 & 2\end{smallmatrix}\right),X,Y,\{\alpha_0, \alpha_1\},\{\alpha_0^\vee,\alpha_1^\vee\})$ is a root generating system. Let $\fG$ be the  Kac-Moody group associated with $\mathcal{S}$ and $G=\fG(\cK)$. Then by \cite[13]{kumar2002kac} and \cite[7.6]{marquis2018introduction}, $G=\mathrm{SL}_2\left(\cK[u,u^{-1}]\right)\rtimes \cK^*$, where $u$ is an indeterminate and if $(M,z),(M_1,z_1)\in G$, with $M=\begin{psmallmatrix}
a(\qp,u) & b(\qp,u)\\
c(\qp,u) & d(\qp,u)
\end{psmallmatrix}$, $M_1=\begin{psmallmatrix}
a_1(\qp,u) & b_1(\qp,u)\\ c_1(\qp,u) & d_1(\qp,u)
\end{psmallmatrix}$,  we have \begin{equation}\label{eqDef_semi_direct_product}
(M,z).(M_1,z_1)=(M\begin{psmallmatrix}
a_1(\qp,z u) & b_1(\qp,zu)\\ c_1(\qp,z u) & d_1(\qp,z u)
\end{psmallmatrix},zz_1) .
\end{equation}

Note that the family $(\alpha_0^\vee,\alpha_1^\vee)$ is not free. We have $\Phi=\{\alpha+k\delta\mid\alpha\in \{\pm {\mathring{\alpha}}\}, k\in \Z\}$ and $(\alpha_0,\alpha_1)$ is a basis of this root system. We denote by $\Phi^+$ (resp. $\Phi^-$) the set $\Phi\cap (\N\alpha_0+\N\alpha_1)$ (resp $-\Phi_+$). For $k\in \Z$ and $y\in \cK$, we set $x_{{\mathring{\alpha}}+k\delta}(y)=(\begin{psmallmatrix}
1 & u^k y\\ 0 & 1
\end{psmallmatrix},1)\in G$ and $x_{-{\mathring{\alpha}}+k\delta}(y)=(\begin{psmallmatrix}
1 & 0 \\ u^k y& 1 
\end{psmallmatrix},1)\in G$.

Let  $f,g\in \cK$ be  such that $\omega(f)=\omega(g)=0$. Let $\ell,n\in \Z$.  Then $\left(\begin{psmallmatrix} f \qp^\ell  & 0 \\ 0 & f^{-1} \qp^{-\ell} \end{psmallmatrix},g \qp^n\right)$ acts on $\A$ by the translation of vector $-\ell {\mathring{\alpha}}^\vee-nd$. For $\mu =\ell {\mathring{\alpha}}^\vee+nd\in Y$,  we set $t_\mu=\left(\begin{psmallmatrix} \qp^{-\ell}  & 0 \\ 0 &  \qp^{\ell} \end{psmallmatrix}, \qp^{-n}\right)$, which acts by the translation of vector $\mu$ on $\A$.  We set $\lambda={\mathring{\alpha}}^\vee+3d$. We have  $\alpha_0(\lambda)=1$, $\alpha_1(\lambda)=2$ and thus $\lambda\in C^v_f$. 

By \cite[4.12 3 b]{rousseau2016groupes}, $U_{0}^{pm+}=\left(\begin{psmallmatrix} 1+u\cO[u] & \cO [u]\\
u\cO[u] & 1+u\cO[u]\end{psmallmatrix},1\right)\cap G$ and similarly \[U_0^{nm-}=\left(\begin{psmallmatrix} 1+u^{-1}\cO[u^{-1}] & u^{-1}\cO [u^{-1}]\\
\cO[u^{-1}] & 1+u^{-1}\cO[u^{-1}]\end{psmallmatrix},1\right)\cap G.\]

 We make the following assumption: \begin{equation}\label{e_assumption}
 \forall n\in \N^*, \ker \pi_n\subset \begin{psmallmatrix} 1+\qp^n \cO[u,u^{-1}] & 1+\qp^n \cO[u,u^{-1}]\\
 1+\qp^n \cO[u,u^{-1}] & 1+\qp^n\cO[u,u^{-1}]\end{psmallmatrix}\rtimes (1+\qp^n\cO). 
 \end{equation}
 
 If for any $n\in \N^*$, we have $\fG(\cO/\qp^n \cO)\simeq \begin{psmallmatrix} (\cO/\qp^n \cO)[u,u^{-1}] & (\cO/\qp^n\cO) [u,u^{-1}]\\
  (\cO/\qp^n\cO) [u,u^{-1}] &  (\cO/\qp^n\cO) [u,u^{-1}] \end{psmallmatrix}\rtimes (\cO/\qp^n \cO)^\times$ and $\pi_n$ is the canonical projection, then the assumption is satisfied. However we do not know if it is true. In \cite[7.6]{marquis2018introduction}, $\fG$ is described only on fields and in \cite[13]{kumar2002kac}, only on $\C$.

For $n\in \N^*$, we set $H_n=\ker(\pi_n)\cap \left(\begin{psmallmatrix} \cO[(\qp u)^n,(\qp u^{-1})^n] &  \cO[(\qp u)^n,(\qp u^{-1})^n]\\ 
 \cO[(\qp u)^n,(\qp u^{-1})^n]	&  \cO[(\qp u)^n,(\qp u^{-1})^n] \end{psmallmatrix},\cK^*\right)$.

 \begin{Proposition}
If \eqref{e_assumption} is true, then  the filtrations $(H_n)_{n\in \N^*}$ and $(\cV_{n\lambda})_{n\in \N^*}$ are equivalent.
 \end{Proposition}
 
 \begin{proof}
Let $n\in \N^*$. By Lemma~\ref{l_action_T_U}, we have $U_{[-n\lambda,n\lambda]}^{pm+}=U_{-n\lambda}^{pm+}=t_{-n\lambda} U_0^{pm+} t_{n\lambda}$. We have $t_{n\lambda}=t_{n{\mathring{\alpha}}^\vee} t_{3nd}$. We have \[t_{-3nd} U_{0}^{pm+} t_{3nd}\subset \begin{psmallmatrix} 1+(\qp^{3n} u) \cO[\qp^{3n} u] & \cO[\qp^{3n} u]\\   (\qp^{3n} u) \cO[\qp^{3n} u] & 1+(\qp^{3n} u) \cO[\qp^{3n} u]  \end{psmallmatrix} \ltimes \{1\} .\] We have \[\begin{aligned} &t_{-n{\mathring{\alpha}}^\vee} \left(\begin{psmallmatrix} 1+\qp^{3n} u \cO[\qp^{3n} u] & \cO[\qp^{3n} u]\\     (\qp^{3n} u) \cO[\qp^{3n} u] & 1+\qp^{3n} u\cO[\qp^{3n} u]  \end{psmallmatrix} \ltimes \{1\}\right) t_{n{\mathring{\alpha}}^\vee} \\ &  \subset \left(\begin{psmallmatrix} 1+\qp^{3n} u \cO[\qp^{3n} u] & \qp^{2n}\cO[\qp^{3n} u]\\    \qp^{-2n} (\qp^{3n} u)\cO[\qp^{3n} u] & 1+\qp^{3n} u \cO[\qp^{3n} u]  \end{psmallmatrix} \ltimes \{1\}\right)\\ 
 & \subset  \left(\begin{psmallmatrix} 1+(\qp^{n} u) \cO[\qp^{n} u] & \cO[\qp^{n} u]\\    \qp^n u \cO[\qp^{n} u] & 1+\qp^{n} u \cO[\qp^{n} u]  \end{psmallmatrix} \ltimes \{1\}\right)\\ &\subset  \mathrm{SL}_2(\cO[\qp^{n} u])\ltimes \{1\}\subset  \mathrm{SL}_2(\cO[\qp^{n} u,\qp^{n} u^{-1}])\ltimes \{1\} .\end{aligned}\]

 Similarly, $U_{[-n\lambda,n\lambda]}^{nm-}\subset \mathrm{SL}_2(\cO[\qp^n u,\qp^n u^{-1}])\ltimes \{1\}$.

 As $T_{2n}\subset \mathrm{SL}_2(\cO[\qp^n u,\qp^n u])\ltimes \{1+\qp^n  \cO\}$, we deduce $\cV_{n\lambda}\subset H_n$, since $\cV_{n\lambda}\subset \ker(\pi_n)$. 
 
 Now let $M\in  \mathrm{SL}_2(\cO[\qp^{2n} u,\qp^{2n} u^{-1}])\cap \ker(\pi_{2n})$ and $a\in \cK^*$. Using \eqref{e_assumption}, we write \[M=\begin{pmatrix}
1+  a_0 \qp^{2n} +\sum_{|i|\geq 1} a_i \qp^{2n |i|} u^i &  b_0 \qp^{2n} +\sum_{|i|\geq 1} a_i \qp^{2n |i|} u^i\\ 
    c_0 \qp^{2n} +\sum_{|i|\geq 1} c_i \qp^{2n |i|} u^i &  1+ d_0 \qp^{2n} +\sum_{|i|\geq 1} d_i \qp^{2n |i|} u^i
 \end{pmatrix},\]with $a_i,b_i,c_i,d_i \in  \cO$, for all $i$. Then \[t_{-nd} (M,a) t_{nd}=\left(\begin{pmatrix}
  1+ a_0 \qp^{2n} +\sum_{|i|\geq 1} a_i \qp^{n (2|i|-i)} u^i &  b_0 \qp^{2n} +\sum_{|i|\geq 1} a_i \qp^{n (2|i|-i)} u^i\\ 
    c_0 \qp^{2n} +\sum_{|i|\geq 1} c_i \qp^{n (2|i|-i)} u^i &  1+d_0 \qp^{2n} +\sum_{|i|\geq 1} d_i \qp^{n (2|i|-i)} u^i
 \end{pmatrix},a\right).\]
 
 Therefore \[\begin{aligned} & t_{-n{\mathring{\alpha}}^\vee-nd} (M,a) t_{n({\mathring{\alpha}}^\vee+d)} \\ & =\left(\begin{pmatrix}
  1+ a_0 \qp^{2n} +\sum_{|i|\geq 1} a_i \qp^{n (2|i|-i)} u^i &  \qp^{2n} (b_0 \qp^{2n} +\sum_{|i|\geq 1} a_i \qp^{n (2|i|-i)} u^i)\\ 
   \qp^{-2n} \left(c_0 \qp^{2n} +\sum_{|i|\geq 1} c_i \qp^{n (2|i|-i)} u^i\right) &  1+d_0 \qp^{2n} +\sum_{|i|\geq 1} d_i \qp^{n (2|i|-i)} u^i
 \end{pmatrix},a\right)\\ 
 &  \in \mathrm{SL}_2(\cO[u,u^{-1}])\ltimes \cO^*.\end{aligned}\]
 
By \cite[Lemma 6.10]{bardy2022twin},  $H_{2n}$ fixes $n\lambda'$, where $\lambda'={\mathring{\alpha}}^\vee+d$. Similarly, it fixes $-n\lambda'$. Therefore $H_{2n}\subset G_{[-n\lambda',n\lambda']}\cap \ker \pi_{2n}$. We have $G_ {[-n\lambda',n\lambda']}=U_{[-n\lambda',n\lambda']}^{pm+}.U_{[-n\lambda',n\lambda']}^{nm-}.\fT(\cO)$, by \eqref{e_good_fixator}.  Using the inclusion  $\langle U_{[-n\lambda',n\lambda']}^{pm+},U_{[-n\lambda',n\lambda']}^{nm-}\rangle\subset \ker \pi_n$, we deduce that $G_{[-n\lambda',n\lambda']}\cap \ker \pi_n\subset \cV_{n\lambda'}$. As $(\cV_{m\lambda})$ and $(\cV_{m\lambda'})$ are equivalent, we deduce that $(H_m)$ and $(\cV_{n\lambda})$ are equivalent. 
 \end{proof}

\printindex

\bibliography{bibliographie.bib}

\end{document}